\documentclass[11pt]{article}
\usepackage{amssymb}
\usepackage[centering,height=220mm,width=150mm]{geometry}
\usepackage[compress]{cite}
\usepackage[colorlinks,
linkcolor=blue,
anchorcolor=blue,
citecolor=blue
]{hyperref}
\usepackage{tikz}
\usepackage{xcolor}
\usepackage{mathrsfs}
\usepackage{amsfonts}
\usepackage{amsmath}
\numberwithin{equation}{section}
\allowdisplaybreaks[4]
\newtheorem{Theorem}{Theorem}[section]
\newtheorem{Proposition}[Theorem]{Proposition}
\newtheorem{Corollary}[Theorem]{Corollary}
\newtheorem{Lemma}[Theorem]{Lemma}

\ifcsname proof\endcsname
\else
  
  \def\endproof{\hfill $\Box$}
\fi
\def\rank{\mathrm{rank}\,}

\def\bbR{\mathbb{R}}
\def\diag{\mathrm{diag}\ }

\ifcsname texorpdfstring\endcsname
\else
  \def\texorpdfstring#1#2{#1}
\fi
\begin{document}

\title{Extension of Multilinear Fractional Integral Operators to Linear Operators on Mixed-Norm Lebesgue Spaces
\thanks{This work was partially supported by the
National Natural Science Foundation of China (11525104, 11531013, 11761131002 and 11801282).}}

\author{ Ting Chen\qquad  and \qquad Wenchang Sun\thanks{Corresponding author.} \\
\small School of Mathematical Sciences and LPMC,  Nankai University,
      Tianjin~300071, China\\
 \small  Emails:  t.chen@nankai.edu.cn,\quad  sunwch@nankai.edu.cn}
\date{}
\maketitle

\begin{abstract}
In [C. E. Kenig and E. M. Stein, Multilinear estimates and fractional integration, Math. Res. Lett., 6(1):1-15, 1999], the following type of multilinear fractional integral
\[
  \int_{\bbR^{mn}} \frac{f_1(l_1(x_1,\ldots,x_m,x))\cdots f_{m+1}(l_{m+1}(x_1,\ldots,x_m,x))}{(|x_1|+\ldots+|x_m|)^{\lambda}} dx_1\ldots dx_m
\]
was studied, where $l_i$ are linear maps from $\bbR^{(m+1)n}$ to $\bbR^n$ satisfying certain conditions.
They proved the boundedness of such multilinear fractional integral from $L^{p_1}\times \ldots \times L^{p_{m+1}}$ to $L^q$ when the indices satisfy
the homogeneity condition.
In this paper, we show that the above multilinear fractional integral extends to a linear operator
for functions in the mixed-norm Lebesgue space
$L^{\vec p}$
which contains
$L^{p_1}\times \ldots \times L^{p_{m+1}}$ as a subset.
Under less restrictions on the linear maps $l_i$,
we give a complete characterization of the indices
$\vec p$, $q$ and $\lambda$ for which such an operator is bounded from $L^{\vec p}$ to $L^q$.
And for $m=1$ or $n=1$, we give necessary and sufficient conditions  on $(l_1, \ldots, l_{m+1})$, $\vec p=(p_1,\ldots, p_{m+1})$, $q$ and $\lambda$ such that the operator
is bounded.
\end{abstract}


\section{Introduction and the Main Results}

The fractional integral operator is useful in the study of differentiability and smoothness of functions.
In \cite{KenigStein1999}, Kenig and  Stein studied the multilinear  fractional integral of the
following type,
\[
  \int_{\bbR^{mn}} \frac{f_1(l_1(x_1,\ldots,x_m,x))\cdots f_{m+1}(l_{m+1}(x_1,\ldots,x_m,x))}{(|x_1|+\ldots+|x_m|)^{\lambda}} dx_1\ldots dx_m,
\]
where $l_i(x_1,\ldots,x_m,x) = \sum_{i=1}^m A_{i,j}x_j + A_{i,m+1}x$ and  $A_{i,j}$ are $n\times n$ matrices.

They proved that
 the above  fractional integral is bounded
 from $L^{p_1}\times \ldots \times L^{p_{m+1}}$
 to $L^q$ if
  $1<p_i\le \infty$, $1\le i\le m+1$,
  $0<q<\infty$, $0<\lambda<mn$,
\begin{equation}\label{eq:ss:e0}
\frac{1}{p_1} + \ldots + \frac{1}{p_{m+1}} = \frac{1}{q} + \frac{mn-\lambda}{n},
\end{equation}
and the coefficient matrices $A_{i,j}$ satisfy the followings,
\begin{enumerate}
\item $A = (A_{i,j})_{1\le i,j\le m+1}$ is an $(m+1)n\times (m+1)n$  invertible matrix,
\item each $A_{i,m+1}$ is an $n\times n$ invertible matrix for $1\le i\le m+1$, and
\item $(A_{i,j})_{\substack{ 1\le i\le m+1, i\ne i_0 \\ 1\le j\le m}}$ is an $mn\times mn$ invertible matrix for every $1\le i_0\le m+1$.
\end{enumerate}

In this paper, we show that the miltilinear operator can be extended to a linear operator
defined on the mixed-norm Lebesgue space $L^{\vec p}$.  Recall that for $\vec p = (p_1,\ldots,p_k)$, where $0<p_1,\ldots,p_k\le \infty$ and $k\ge 1$,   $L^{\vec p}$ consists of all measurable functions $f$ for which
\[
  \|f\|_{L^{\vec p}} :=   \Big\|  \| f\|_{L_{x_1}^{p_1}} \cdots \Big\|_{L_{x_k}^{p_k}}<\infty.
\]
For convenience,  we also write the $L^{\vec p}$
norm as
$\|\cdot\|_{L_{x_k}^{p_k}(\ldots (L_{x_1}^{p_1}))}$ or $\|\cdot\|_{L_{(x_1,\ldots,x_k)
}^{(p_1,\ldots,p_k)}}$.

Benedek and Panzone   \cite{Benedek1961} introduced the
mixed-norm Lebesgue spaces   and proved   that such spaces have similar  properties as ordinary Lebesgue spaces. Further developments
  which include the boundedness of classical operators
and other generalizations can be found in
\cite{Benedek1962,%
Fernandez1987,%
Hart2018,%
Hormander1960,%
Kurtz2007,%
Rubio1986,%
Stefanov2004,%
Torres2015}.
Recently, mixed-norm spaces have been studied
in various aspects
\cite
{BandaliyevSerbetci2018,%
Boggarapu2017,%
CarneiroOliveiraSousa2019,%
ChenSun2017,%
Ciaurri2017,%
Cleanthous2017,%
Cleanthous2017c,%
Cordoba2017,%
Georgiadis2017,%
Ho2018,%
HuangLiuYangYuan2018,%
HuangLiuYangYuan2019,%
Johnsen2015,%
KarapetyantsSamko2018,%
Lechner2018,%
LiStinga2017,%
Sandik2018,%
WeiYan2018}.
And we refer to  \cite{HuangYang2019} for a  survey on mixed-norm spaces.

We focus on the factional integral on mixed-norm Lebesgue spaces.
Before stating our results, we introduce some notations. For $x=(x_1^{(1)}$, $\ldots$, $x_1^{(n)}$, $\ldots$, $x_{m+1}^{(1)}$, $\ldots$, $x_{m+1}^{(n)})^* \in \bbR^{(m+1)n}$,
where $*$ denotes the transpose of a vector or a matrix,
we also write
\begin{equation}\label{eq:x}
 x = \begin{pmatrix}
 x_1 \\ \vdots \\ x_{m+1}
 \end{pmatrix}, \qquad \mathrm{where}\,\,
 x_i =\begin{pmatrix}
 x_i^{(1)} \\ \vdots \\ x_i^{(n)}
 \end{pmatrix}.
\end{equation}


Let $A = (A_{i,j})_{1\le i,j\le m+1}$ be an $(m+1)n\times (m+1)n$ matrix. Define
\begin{equation}\label{eq:Tf}
  T_{\lambda}f (x_{m+1}) = \int_{\bbR^{mn}} \frac{f(Ax)}{(|x_1|+\ldots+|x_m|)^{\lambda}} dx_1\ldots dx_m.
\end{equation}
Denote
\begin{equation}\label{eq:rk 1}
 r_{m+2}=0\quad \mathrm{and}\quad
 r_k=    \rank
     \begin{pmatrix}
     A_{k,m+1} \\
     \vdots \\
     A_{m+1,m+1}
     \end{pmatrix},\quad 1\le k\le m+1.
\end{equation}
For certain matrix $A$, we give necessary and sufficient conditions
on the indices $\vec p$, $q$ and $\lambda$ such that
  $T_{\lambda}$ is bounded from $L^{\vec p}(\bbR^{(m+1)n})$ to
$L^q(\bbR^n)$.
Specifically, we prove the following.

\begin{Theorem}\label{thm:main}
Let $1\le p_i\le \infty$ for  $1\le i\le m+1$, $q>0$ and $0< \lambda<mn$ be  constants which satisfy (\ref{eq:ss:e0}).
Set  $\vec p = (p_1, \ldots, p_{m+1})$.
Suppose that both $A$ and its $(1,m+1)$-minor $(A_{i,j})_{2\le i\le m+1, 1\le j\le m}$ are invertible matrices.
Then  $T_{\lambda}$ is bounded from $L^{\vec p}$ to $L^q$
if and only if the following items are satisfied,
\begin{enumerate}
\item The rank of the $mn\times n$ matrix $(A_{i,m+1})_{2\le i\le m+1}$ is $n$.

\item There is some $2\le k\le m+1$ such that $1<p_k<\infty$.
       Let $k_0 =\max\{k:\, p_k>1, 2\le k\le m+1\}$ and $2\le k_1<\ldots<k_{\nu}\le m+1 $  be such that
\begin{equation}\label{eq:rk m a}
r_2=\ldots = r_{k_1}>r_{k_1+1}=\ldots > \ldots =r_{k_{\nu}}>r_{k_{\nu}+1}=\ldots=r_{m+2}.
\end{equation}
Then   the indices $\vec p$ and $q$  satisfy
    \begin{equation}\label{eq:pq:necessary}
      \min\{p_{k_l}:\, 1\le l\le \nu\}<q
      \quad
      \mathrm{and}
      \quad
      \max\{p_{k_l}:\, 0\le l\le \nu\}\le q<p_1.
    \end{equation}
\end{enumerate}
\end{Theorem}

When applying the above theorem to the multilinear case, we get the norm estimate
even if  some $p_i$ is equal to $1$. Moreover, we relax the restrictions on the matrix
$A$.
In particular, we have no hypothesis on the matrices $A_{i,m+1}$ for $1\le i\le m+1$.

There are two main points in the proof of Theorem~\ref{thm:main}.
First, we have to simply the integral in (\ref{eq:Tf})
 since all the variables
are tangled through the matrix $A$.
Second, to prove the sufficiency, we have to weave the integration order of the variables $x_1^{(1)}$, $\ldots$, $x_1^{(n)}$ and
$x_j$ for $2\le j\le m+1$  very carefully since the integration order
 of $x_1$, $\ldots$, $x_{m+1}$ is not switchable in general.
Moreover,
both the maximal function and the theory
of interpolation spaces   for mixed-norm Lebesgue spaces are involved.

We show that the hypothesis that
the $(1,m+1)$-minor $(A_{i,j})_{ 2\le i\le m+1, 1\le j\le m} $ of $A$ is invertible
is also necessary if $m=1$.
Moreover, for the case  $m=1$ or $n=1$, we give complete characterizations of
$A$, $\vec p$, $q$ and $\lambda$ such that
$T_{\lambda}$ is bounded from
$L^{\vec p}$ to $L^q$, see Theorems~ \ref{thm:m=1} and \ref{thm:A:n=1} for details.
As a consequence, we show that
the bi-linear fractional integral
\[
 (f_1, f_2) \rightarrow \int_{\bbR^n} \frac{f_1(x-t) f_2(x+t)}{|t|^{\lambda}} dt
\]
studied by Grafakos \cite{Grafakos1992},
Kenig and Stein \cite{KenigStein1999}, Grafakos and  Kalton
\cite{GrafakosKalton2001}, Moen \cite{Moen2014},  Grafakos and  Lynch
\cite{GrafakosLynch2015},
extends to a linear operator on mixed-norm Lebesgue spaces.

The multilinear Riesz potentials  of
the following type
\[
   J_{\lambda} (f_1\otimes \ldots \otimes f_m)(x)=\int_{\bbR^{mn}}  \frac{f_1(y_1)\ldots
     f_m(y_m)}{(|x-y_1|+\ldots+|x-y_m|)^{\lambda}} dy_1 \ldots dy_m
\]
was studied extensively.
It was shown in \cite{Grafakos1992,GrafakosKalton2001,KenigStein1999}  that whenever $1<p_i\le \infty$, $0<q<\infty$, $0<\lambda<mn$ and
\begin{equation}\label{eq:homo:p}
  \frac{1}{p_1} + \ldots + \frac{1}{p_m} = \frac{1}{q} + \frac{mn-\lambda}{n},
\end{equation}
$J_{\lambda}$ maps $L^{p_1}\times \ldots \times L^{p_m}$ continuously to $L^q$.
And we refer to
\cite{Cruz-UribeMoen2013,%
Cruz-UribePerez1999,%
Cruz-UribePerez2000,%
FontanaMorpurgo,%
HanLuZhu2012,%
Hytonen2018,%
LaceyThiele1997,%
LaceyLi2006,%
LaceyMoenPerezTorres2010,%
LiMoenSun2015,%
LiSun2016}
for various aspects of Riesz potentials.

In this paper, we show that $J_{\lambda}$ can be extended to a linear operator defined on $L^{\vec p}$ with $\vec p = (p_1,\ldots,p_m)$.
Specifically, for $f\in L^{\vec p}(\bbR^{mn})$, define
\[
  J_{\lambda}f (x)  = \int_{\bbR^{mn}}  \frac{f(y_1,\ldots, y_m)}{(|x-y_1|+\ldots+|x-y_m|)^{\lambda}} dy_1 \ldots dy_m.
\]
We show that $J_{\lambda}$ is bounded from $L^{\vec p}$
to $L^q$ for certain indices.
In fact, we prove the boundedness for more general
operators.

Let $y\in \bbR^{mn}$ be defined similarly to $x$  in (\ref{eq:x}) and set
\begin{equation}\label{eq:w:A}
D=    \begin{pmatrix}
  D_1\\
  \vdots \\
  D_m
  \end{pmatrix},
\end{equation}
where
$D_1, \ldots, D_m$ are $n\times n$ matrices.
For $0<\lambda< mn$, define
\begin{align*}
  J_{\lambda,D}f(x) &=
    \int_{\bbR^{mn}} \frac{f(y)}
    { |D x -y|^{\lambda}}
    dy\\
    &=    \int_{\bbR^{mn}} \frac{f(y_1,\ldots, y_m)}
    {(|D_1 x -y_1| +\ldots + |D_m x - y_m|)^{\lambda}}
    dy_1\ldots dy_m.
\end{align*}
We give a complete characterization of $D$, $\vec p$,
$q$ and $\lambda$ such that
$J_{\lambda, D}$ is bounded from
$L^{\vec p}$ to $L^q$.

Let
\begin{equation}\label{eq:rk def}
  \gamma_{m+1}=0, \quad   \gamma_i=  \rank
  \begin{pmatrix}
  D_i\\
  \vdots \\
  D_m
  \end{pmatrix},\qquad 1\le i\le m.
\end{equation}
Suppose that for $1\le i_1<\ldots<i_{\nu}\le m$,
\begin{equation}\label{eq:rk}
  \gamma_1=\ldots = \gamma_{i_1}
  >\gamma_{i_1+1}=\ldots >\ldots = \gamma_{i_{\nu}}> \gamma_{i_{\nu}+1}=\ldots = \gamma_{m+1}.
\end{equation}

\begin{Theorem}\label{thm:factional:D}
Suppose that  $\vec p = (p_1, \ldots, p_m)$ with
$1\le p_i\le \infty$, $1\le i\le m$, $0<q\le \infty$ and $0<\lambda<mn$.
Then $J_{\lambda,D}$ is bounded from $L^{\vec p}$ to
$L^q$ if and only if the following items are satisfied,

\begin{enumerate}
\item The rank of $D$ is equal to $n$.

\item The indices $\vec p$, $q$ and $\lambda$ meet (\ref{eq:homo:p}).

\item
      There is some $1\le i\le m$ such that
$1<p_i<\infty$. Set $i_0 = \max\{i:\, p_i>1\}$.

\item  Let $\gamma_i$ and $i_l$  be defined by (\ref{eq:rk def})
and (\ref{eq:rk}), respectively.
Then
\begin{equation}\label{eq:s:e5:D}
\min\{p_{i_l}:\,  1\le l\le \nu \}<q
\quad \mathrm{and}\quad
\max\{p_{i_l}:\, 0\le l\le \nu\}\le q<\infty.
\end{equation}
\end{enumerate}
\end{Theorem}

We point out that Theorem~\ref{thm:factional:D}
is used in the proof of Theorem~\ref{thm:main}.
In fact, it is related to the case
$p_1 = \infty$ in Theorem~\ref{thm:main}.

By setting $D_1=\ldots = D_m=I$ in  Theorem~\ref{thm:factional:D}, we get
  necessary and sufficient conditions
  for $J_{\lambda}$
to be  bounded from $L^{\vec p}$ to $L^q$.

\begin{Corollary}\label{Co:fractional integral a}
Suppose that $0<\lambda<mn$,  $\vec p    = (p_1, \ldots, p_m)$ with $1\le p_i\le   \infty$, $1\le i\le m$, and $0<q \le \infty$.
Then the norm estimate
\begin{equation}\label{eq:J lambda aa}
  \|J_{\lambda }f\|_{L^q} \lesssim \|f\|_{L^{\vec p}},\qquad \forall f\in L^{\vec p}(\bbR^{mn})
\end{equation}
is true if and only if

\begin{enumerate}

\item  The indices $\vec p$, $q$ and $\lambda$  meet  (\ref{eq:homo:p}),

\item  There is some $1\le i\le m$ such that
 $1<p_i<\infty$. Set $i_0 = \max\{i:\, p_i>1, 1\le i\le m\}$. Then
     $p_{i_0}\le q<\infty$ (for $i_0<m$) or
    $p_m<q<\infty$ (for $i_0=m$).
\end{enumerate}
\end{Corollary}

The paper is organized as follows. In Section 2,
we give some preliminary results.
For operators that commute with translations on mixed-norm Lebesgue spaces,
we give the relationship between the corresponding indices,   which generalizes
the result for Lebesgue spaces~\cite[Theorem 2.5.6]{Grafakos2008}.
In Sections 3 and 4,  we give  proofs
of Theorems~\ref{thm:factional:D} and \ref{thm:main}, respectively.
In Section 5, we consider the case  $m=1$.
We give necessary and sufficient
conditions on $A$, $\vec p$, $q$ and $\lambda$
such that
$T_{\lambda}$ is bounded from $L^{\vec p}$ to $L^q$.
In particular,
we show that the hypothesis
that the $(1,m+1)$-minor $(A_{i,j})_{2\le i\le m+1, 1\le j\le m}$ of $A$ is invertible
is also necessary when $m=1$.
In Section 6, for the case  $n=1$, we  also give the characterization of
$A$, $\vec p$, $q$ and $\lambda$ for which $T_{\lambda}$ is bounded from $L^{\vec p}$ to $L^q$.

\section{Preliminary Results}

\subsection{The Riesz Potential}
The boundedness of the fractional integral  can be found in many textbooks, e.g.,
see
\cite[Theorem 6.1.3]{Grafakos2008m},
\cite[Proposition 7.8]{Muscalu2013i}
or
\cite[Chapter 5.1]{Stein1970}.

\begin{Proposition}\label{prop:fractional}
Let $\lambda$ be a real number with $0<\lambda<n$
and let
$1\le p,q\le \infty$.
Then
\[
  f\mapsto \int_{\bbR^n} \frac{f(y)}{|x-y|^{\lambda}} dy
\]
maps $L^p(\bbR^n)$ continuously to $L^q(\bbR^n)$ if and only if $1<p<q<\infty$ and
\[
  \frac{1}{p} = \frac{1}{q} + \frac{n-\lambda}{n}.
\]
\end{Proposition}

\subsection{Operators That Commute with Translations}
It is known that if a  linear operator is bounded from $L^p(\bbR^n)$ to $L^q(\bbR^n)$ and commutes with translations, then $q\ge p$ (See \cite[Theorem 2.5.6]{Grafakos2008}).
We show that a similar result is true for operators on mixed-norm
Lebesgue spaces.

\begin{Theorem}\label{thm:commute}
Let $T$ be a linear operator which is bounded from
$L^{\vec p}$ to $L^{\vec q }$ and commutes with translations, where $1\le  p_i, q_i<\infty$. Then we have $q_i\ge p_i$, $1\le i\le m$.
\end{Theorem}

Before proving this theorem, we present some preliminary results.

It is well known that for $f\in L^p$ with $1\le p<\infty$,
$\lim_{|y|\rightarrow\infty} \|f+f(\cdot-y)\|_p=2^{1/p} \|f\|_p$.
For mixed-norm Lebesgue spaces, we show that the limit is path dependent.

\begin{Lemma}\label{Lm: translation}
Let $y = (y_1,\ldots,y_m)\in \bbR^{mn}$. Suppose that
$y_k\ne 0$
for some $1\le k\le m$  and $y_i=0$ for all $k+1\le i\le m$.     Then for any $f\in L^{\vec p}(\bbR^{mn})$, where $\vec p=(p_1, \ldots, p_m)$ with $1\le  p_i<\infty$, $1\le i\le m$, we have
\[
  \lim_{a\rightarrow \infty} \| f(\cdot - a y) + f\|_{L^{\vec p}}
  =
  2^{1/p_k}\|f\|_{L^{\vec p}}.
\]
\end{Lemma}

\begin{proof}
First, we prove the conclusion for $f\in C_c(\bbR^{mn})$.
For $a$ large enough, we have
\[
   |f(x-a y) + f(x) |^{p_1} = |f(x-a y)|^{p_1} + | f(x)|^{p_1}.
\]
Hence for $k>1$,
\begin{align*}
 \|f(x-a y)  &+ f(x)\|_{L_{x_1}^{p_1}} \\
  & = \|f(\cdot, x_2-a y_2,\ldots,x_m-ay_m )\|_{L_{x_1}^{p_1}}
\! +
  \|f(\cdot, x_2 ,\ldots,x_m )\|_{L_{x_1}^{p_1}}.
\end{align*}
Note that the above equations are true for all
$1\le p_1< \infty$.
Moreover, since $f$ is compactly supported,
one of $\|f(\cdot, x_2-a y_2,\ldots,x_m-ay_m )\|_{L_{x_1}^{p_1}}$
and $  \|f(\cdot, x_2 ,\ldots,x_m )\|_{L_{x_1}^{p_1}}$
must be zero when $k\ge 2$ and $a$ is large enough.
Taking the $L^{p_i}$  norm  with respect to $x_i$ on both sides of the above equation successively, $2\le i\le k$,
and keeping in mind that $ y_k\ne 0$ while $ y_j= 0$ for $k+1\le j\le m$,
we get
\[
  \|f(\cdot-a y) + f \|_{L_{x_k}^{p_k}(\ldots (L_{x_1}^{p_1}))}
  = 2^{1/p_k}\|f(\cdot, \ldots, \cdot, x_{k+1}, \ldots,x_m)\|_{L_{x_k}^{p_k}(\ldots (L_{x_1}^{p_1}))}.
\]
Hence for $a$ large enough,
\[
  \|f(\cdot-a y) + f\|_{L^{\vec p}}
  = 2^{1/p_k} \|f\|_{L^{\vec p}}.
\]

Next we consider the general case.
Fix some $f$ in $L^{\vec p}$.
Since $C_c(\bbR^n)$ is dense in $L^{\vec p}$,
for any $\varepsilon>0$, there is some $g\in C_c(\bbR^n)$ such that
$\|f-g\|_{L^{\vec p}}<\varepsilon$.

We see from the previous arguments that for $a$ large enough,
\[
  \|g(\cdot-a y) + g\|_{L^{\vec p}}
  = 2^{1/p_k} \|g\|_{L^{\vec p}}.
\]
Hence
\begin{align*}
 & \Big|\|f(\cdot-a y) + f\|_{L^{\vec p}} - 2^{1/p_k}\| f\|_{L^{\vec p}} \Big| \\
 &\le \Big|\|f(\cdot-a y) + f\|_{L^{\vec p}} - 2^{1/p_k}\| g\|_{L^{\vec p}} \Big| + 2^{1/p_k}\Big|\|g\|_{L^{\vec p}}- \|f\|_{L^{\vec p}}\Big| \\
 &=  \Big|\|f(\cdot-a y) + f\|_{L^{\vec p}} - \|g(\cdot-a y) + g\|_{L^{\vec p}} \Big|
 +2^{1/p_k}\Big|\|g\|_{L^{\vec p}}- \|f\|_{L^{\vec p}}\Big| \\
&\le  \|f(\cdot-a y) - g(\cdot-a y)\|_{L^{\vec p}}
   + \|f-g\|_{L^{\vec p}} +2^{1/p_k}\varepsilon\\
&\le (2+ 2^{1/p_k}) \varepsilon .
\end{align*}
Therefore,
\[
\lim_{a\rightarrow \infty}  \|f(\cdot-a y) + f\|_{L^{\vec p}}
  = 2^{1/p_k} \|f\|_{L^{\vec p}},\qquad \forall f\in L^{\vec p}.
\]
\end{proof}

We are now ready to give a proof of Theorem~\ref{thm:commute}.

\begin{proof}[Proof of Theorem~\ref{thm:commute}]
Fix some $1\le k\le m$.
Let $z\in\bbR^{mn}$ be such that
$z_i=0$ for $i\ne k$ and $z_k = (1,\ldots, 1)^*\in\bbR^n$. For any $f\in L^{\vec p}$
and $a>0$, denote $f_{az} = f(\cdot-az)$. We see from Lemma~\ref{Lm: translation} that
as $a\rightarrow\infty$,
\begin{align*}
\|T f + T f_{az}\|_{L^{\vec q }}
&\le \|T \|_{L^{\vec p}\rightarrow L^{\vec q }}
   \|f+f_{az}\|_{L^{\vec p}}
 \rightarrow  2^{1/p_k}  \|T\|_{L^{\vec p}\rightarrow L^{\vec q }} \|f\|_{L^{\vec p}}.
\end{align*}
On the other hand, since $T$ commutes with translations, we have
\begin{align*}
\|T f + T  f_{az}\|_{L^{\vec q }}
=\|T f + (T f)(\cdot+az)\|_{L^{\vec q }}.
\end{align*}
Using Lemma~\ref{Lm: translation} again, we get
\[
  \liminf_{a\rightarrow\infty}\|T f + T f_{az}\|_{L^{\vec q}}
  \ge 2^{1/q_k} \|T f\|_{L^{\vec q}}.
\]
Hence $p_k\le q_k$, $1\le k\le m$.
\end{proof}

\subsection{The Interpolation Theorem}

To prove the main results, we need the interpolation theorem in several aspects.
The following result is useful in dealing with the case
 $p_{k_{\nu}}=q$ for Theorem~\ref{thm:main} or  $   p_{i_{\nu}} = q$ for Theorem~\ref{thm:factional:D},
which is a consequence of the theory of interpolation spaces \cite{Bergh1976,Cwiel1974,Hytonen2016,Janson1988,Lions1964,Milman1981}.

\begin{Lemma}\label{Lm:interpolation}
Suppose that $T\!$ is a bounded linear operator
from $L^s(L^{s_i})$ to $L^{{t_i},\infty}(\!L^{\vec u})$, where $i=1,2$, $1\le s_1\ne s_2 <\infty$,
$1\le t_1\ne t_2 < \infty$,
$1\le  s< \infty$
and
$\vec u = (u_1, \ldots, u_m)$ with $1\le u_i\le \infty$, $1\le i\le m$.
Suppose also that for some $0<\theta<1$, we have
\[
  \frac{1}{s } = \frac{1-\theta}{s_1}+ \frac{ \theta}{s_2},\qquad
  \frac{1}{t } = \frac{1-\theta}{t_1}+ \frac{ \theta}{t_2}
  \quad \mathrm{and}\quad {  s\le t}.
\]
Then
$T$ extends to a  bounded operator from
$L^s(L^s)$ to $L^t(L^{\vec u})$.
\end{Lemma}

\begin{proof}
We use   common notations for interpolation spaces: given two quasi-normed spaces $X_0$ and $X_1$ and constants $0<\theta<1$, $q>0$,
$[X_0, X_1]_{\theta, q}$ stands for the real interpolation space (see  \cite{Bergh1976,Hytonen2016} for details). Since $T$ is bounded
from $L^s(L^{s_i})$ to $L^{{t_i},\infty}(L^{\vec u})$, 
by interpolation, $T$ is bounded from
$(L^s(L^{s_1}), L^s(L^{s_2}))_{\theta,s}$
to $(L^{{t_1},\infty}(L^{\vec u}),L^{{t_2},\infty}(L^{\vec u}))_{\theta,s}$.
By a well-known
result of Lions and Peetre \cite{Lions1964} (see also \cite[Theorem 2.2.10]{Hytonen2016}),
\[
(L^s(L^{s_1}), L^s(L^{s_2}))_{\theta,s} = L^s(L^s).
\]
On the other hand, we see from \cite[Theorem 5.3.1]{Bergh1976} (which was stated for ordinary functions, but also valid for functions with values in a Banach space)
that
\[
(L^{{t_1},\infty}(L^{\vec u}),L^{{t_2},\infty}(L^{\vec u}))_{\theta,s}
= L^{t,s}(L^{\vec u}).
\]
Recall that $L^{t,s}$ stands for the Lorentz space.
Hence
$T$ is bounded from
$L^s(L^s)$
to $L^{t,s}(L^{\vec u})$.
Since $s\le t$, $ L^{t,s}(L^{\vec u})\subset L^t(L^{\vec u})$.
This completes the proof.
\end{proof}

\section{Extension of The Multilinear Riesz Potentials to Linear Operators}

In this  section, we give
a proof of Theorem~\ref{thm:factional:D}.
We begin with some preliminary results.

The following lemma gives a method to compute
the $L^{\vec p}$ norm for certain functions whenever the last components of $\vec p$
are equal to
infinities.

\begin{Lemma}\label{Lm:end points D}
Suppose that $1\le p_i\le \infty$, $1\le q\le \infty$
and $0<\lambda<mn$ which meet (\ref{eq:homo:p}).
Suppose that  $p_{i_0+1}=\ldots=p_m=1$ for some $1\le i_0\le m-1$.
Let $D_i$ be $n\times n$ matrices, $1\le i\le m$. Denote $\vec{\tilde p}=(p_1,\ldots, p_{i_0})$. Then the following two items are equivalent:

\begin{enumerate}
  \item there is a constant $C_{\lambda,\vec p, q, n}$ such that for any $h\in \!L^{q'}$ and almost all $(y_{i_0+1}$, $\ldots$, $y_m) \in\bbR^{(m-i_0)n}$,
   \[
     \left\| \int_{\bbR^n}\frac{h(x)dx}{(\sum_{i=1}^m|D_ix-y_i| )^{\lambda}} \right\|_{
     L_{(y_1,\ldots,y_{i_0})}^{(p'_1,\ldots,p'_{i_0}) }}
     \le  C_{\lambda,\vec p, q, n}\|h\|_{L^{q'}}.
   \]

  \item  for any $h\in L^{q'}$,
   \[
     \left\| \int_{\bbR^n}\frac{h(x)dx}{(\sum_{i=1}^{i_0}|D_ix-y_i|
     +\sum_{i=i_0+1}^m|D_i x|)^{\lambda}} \right\|_{L^{\vec {\tilde p}'}}
     \lesssim \|h\|_{L^{q'}}.
   \]
\end{enumerate}
\end{Lemma}

\begin{proof}
First, we assume that (i) is true.
Note that $p'_i=\infty$ for $i_0+1\le i\le m$.
Clearly,
(i) is equivalent to
   \[
     \left\| \int_{\bbR^n}\frac{h(x)dx}{(\sum_{i=1}^m|D_ix-y_i| )^{\lambda}} \right\|_{
     L_y^{\vec p'}}
     \le  C_{\lambda,\vec p, q, n}\|h\|_{L^{q'}}.
   \]
Or equivalently,
\[
     \left\| \int_{\bbR^{mn}}\frac{f(y_1,\ldots,y_m)dy_1\ldots y_m}{(\sum_{i=1}^m|D_ix-y_i| )^{\lambda}} \right\|_{L_x^q}
     \le  C_{\lambda,\vec p, q, n}\|f\|_{L^{\vec p}},\quad f\in L^{\vec p}.
\]
By setting
\[
  f(y_1,\ldots,y_m)
  =\tilde f(y_1,\ldots,y_{i_0}) \prod_{i=i_0+1}^m
    \frac{1}{\delta^n} \chi^{}_{\{|y_i|\le \delta\}}
\]
and letting $\delta\rightarrow 0$, we see from
Fatou's lemma that
\[
     \left\| \int_{\bbR^{i_0n}}\frac{\tilde f(y_1,\ldots,y_{i_0})dy_1\ldots y_{i_0}}{(\sum_{i=1}^{i_0}|D_ix-y_i|
     +\sum_{i=i_0+1}^m|D_i x| )^{\lambda}} \right\|_{L_x^q}
     \le  C_{\lambda,\vec p, q, n}\|\tilde f\|_{L^{\vec {\tilde p}}},\quad \tilde f\in L^{\vec {\tilde p}},
\]
Hence for any $\tilde f\in L^{\vec {\tilde p}}$
and $h\in L^{q'}$, we have
\[
     \left | \int_{\bbR^{(i_0+1)n}}\frac{\tilde f(y_1,\ldots,y_{i_0})h(x)dxdy_1\ldots y_{i_0}}{(\sum_{i=1}^{i_0}|D_ix-y_i|
     +\sum_{i=i_0+1}^m|D_i x| )^{\lambda}} \right|
     \le  C_{\lambda,\vec p, q, n}\|\tilde f\|_{L^{\vec {\tilde p}}} \|h\|_{L^{q'}}.
\]
Therefore, (ii) is true.

Next we assume that (ii) is true.
If $D_{i_0+1}=\ldots=D_m=0$, then (i) is obvious.
Next we assume that not all of $D_{i_0+1}$, $\ldots$,
$D_m$ are zeros.

For
any $f\in L^{\vec {\tilde p}}$ and $h\in L^{q'}$, we have
\begin{align*}
        \left| \int_{\bbR^{(i_0+1)n}} \frac{f(y_1,\ldots,y_{i_0})
 h(x)dx dy_1\ldots dy_{i_0}}{(\sum_{i=1}^{i_0}|D_ix-y_i|
     +\sum_{i=i_0+1}^m|D_i x|)^{\lambda}} \right |
     &\le C_{\lambda,\vec p, q, n}  \|f\|_{\vec {\tilde p}}\|h\|_{L^{q'}}.
\end{align*}
Hence,
\begin{equation}\label{eq:w:e28}
       \left\| \int_{\bbR^{i_0n}}\frac{f(y_1,\ldots,y_{i_0})
        dy_1\ldots dy_{i_0}}{(\sum_{i=1}^{i_0}|D_ix-y_i|
     +\sum_{i=i_0+1}^m|D_i x|)^{\lambda}} \right\|_{L_x^q}
     \le C_{\lambda,\vec p, q, n} \|f\|_{L^{\vec{\tilde  p }}}.
\end{equation}
Let
\[
  \tilde D = \begin{pmatrix}
  D_{i_0+1} \\
  \vdots \\
  D_m
  \end{pmatrix}
  \mathrm{\,\, and\,\,}
  \tilde y = \begin{pmatrix}
  y_{i_0+1}
  \\
  \vdots\\
  y_m
  \end{pmatrix}.
\]
Suppose that $\rank(\tilde D)=r$. Then  there are $(m-i_0)n\times (m-i_0)n$ invertible matrix $U$ and
$n\times n$ invertible matrix $V$ such that
\[
  U \tilde D V   =  \binom{I_r \quad }{\quad 0} .
\]
Observe that
$\sum_{i=i_0+1}^m|D_i x|\approx |U\tilde Dx|$.
By a change of variable of the form $x\rightarrow Vx$,
(\ref{eq:w:e28}) turns out to be
\begin{equation}\label{eq:w:e29}
       \left\| \int_{\bbR^{i_0n}}\frac{f(y_1,\ldots,y_{i_0})
        dy_1\ldots dy_{i_0}}{(\sum_{i=1}^{i_0}|D_iVx-y_i|
     +\sum_{l=1}^r|x^{(l)}|)^{\lambda}} \right\|_{L_x^q}
     \le C_{\lambda,\vec p, q, n} \|f\|_{L^{\vec{\tilde  p }}}.
\end{equation}

Note that
\[
  \sum_{i=i_0+1}^m |D_ix-y_i|
  \approx  |\tilde Dx - \tilde y|
  \approx | U\tilde Dx - U\tilde y|.
\]
For  any $f\in L^{\vec {\tilde p}}$
 and $\tilde y \in\bbR^{ (m -i_0)n  }$,
by a change a variable of the form $x\rightarrow V(x+z)$,
where $z=(z^{(1)},\ldots,z^{(r)},0,\ldots,0)^*\in\bbR^n$
with $z^{(1)}, \ldots, z^{(r)}$ being the first $r$ components
of $U\tilde y$, we have
\begin{align}
&\left\|
  \int_{\bbR^{i_0n}}\frac{|f(y_1,\ldots,y_{i_0})|
        dy_1\ldots dy_{i_0}}{(\sum_{i=1}^{m}|D_ix-y_i|)^{\lambda}} \right\|_{L_x^q}
     \nonumber\\
&\approx \left\|
  \int_{\bbR^{i_0n}}\frac{|f(y_1,\ldots,y_{i_0})|
        dy_1\ldots dy_{i_0}}{(\sum_{i=1}^{i_0}|D_iVx-y_i+D_iVz|
        +|U\tilde DV(x+z) - U\tilde y|
        )^{\lambda}} \right\|_{L_x^q}
     \nonumber \\
&\approx \left\|
  \int_{\bbR^{i_0n}}\frac{|f(y_1,\ldots,y_{i_0})|
        dy_1\ldots dy_{i_0}}{(\sum_{i=1}^{i_0}|D_iVx-y_i+D_iVz|
        +\sum_{l=1}^r|x^{(l)}|
        +|z - U\tilde y|
        )^{\lambda}} \right\|_{L_x^q}
     \nonumber \\
&\lesssim
\left\|
  \int_{\bbR^{i_0n}}\frac{|f(y_1,\ldots,y_{i_0})|
        dy_1\ldots dy_{i_0}}{(\sum_{i=1}^{i_0}|D_iVx-y_i+D_iVz|
        +\sum_{l=1}^r|x^{(l)}|
        )^{\lambda}} \right\|_{L_x^q}. \label{eq:a:e2}
\end{align}
By a change of variable of the form
$(y_1,\ldots,y_{i_0})
\rightarrow (y_1 + D_1Vz,\ldots, y_{i_0}+D_{i_0}Vz)$ in (\ref{eq:a:e2}),
we get
\begin{align*}
&\left\|
  \int_{\bbR^{i_0n}}\frac{|f(y_1,\ldots,y_{i_0})|
        dy_1\ldots dy_{i_0}}{(\sum_{i=1}^{m}|D_ix-y_i|)^{\lambda}} \right\|_{L_x^q}
     \\
&\lesssim
\left\|
  \int_{\bbR^{i_0n}}\frac{|f(y_1+D_1Vz,\ldots,y_{i_0}
  +D_{i_0}Vz  )|
        dy_1\ldots dy_{i_0}}{(\sum_{i=1}^{i_0}|D_iVx-y_i|
        +\sum_{l=1}^r|x^{(l)}|
        )^{\lambda}} \right\|_{L_x^q}.
\end{align*}
By (\ref{eq:w:e29}), we have
\begin{align*}
&\left\|
  \int_{\bbR^{i_0n}}\frac{|f(y_1,\ldots,y_{i_0})|
        dy_1\ldots dy_{i_0}}{(\sum_{i=1}^{m}|D_ix-y_i|)^{\lambda}} \right\|_{L_x^q}
     \\
&\lesssim C_{\lambda,\vec p, q, n} \|f(\cdot+D_1Vz,\ldots, \cdot+D_{i_0}Vz)\|_{L^{\vec{\tilde  p }}} \\
&=  C_{\lambda,\vec p, q, n} \|f\|_{L^{\vec{\tilde  p }}}.
\end{align*}
Hence for any
   $f\in L^{\vec {\tilde p}}$,
    $h\in L^{q'}$
   and $(y_{i_0+1},\ldots, y_m)\in\bbR^{ (m -i_0)n  }$,
\[
  \left|\int_{\bbR^{(i_0+1)n}}\frac{f(y_1,\ldots,y_{i_0 })h(x)
        dy_1\ldots dy_{i_0} dx}{(\sum_{i=1}^{m}|D_ix-y_i|)^{\lambda}}\right|
     \le C_{\lambda,\vec p, q, n} \|f \|_{L^{\vec{\tilde  p }}} \|h\|_{L^{q'}}.
\]
Therefore, for any $(y_{i_0+1},\ldots, y_m)\in\bbR^{ (m -i_0)n  }$,
\[
  \left\|\int_{\bbR^n}\frac{ h(x)
         dx}{(\sum_{i=1}^{m}|D_ix-y_i|)^{\lambda}}
         \right\|_{L^{\vec{\tilde  p }'}}
     \le C_{\lambda,\vec p, q, n}  \|h\|_{L^{q'}}.
\]
This completes the proof.
\end{proof}

To estimate $J_{\lambda,D}f(x)$, we need to estimate $|Dx-y|$.
The following lemma gives the structure of $|Dx-y|$.

\begin{Lemma}\label{Lm:Dx-y}
Suppose that  $\vec p = (p_1, \ldots, p_m)$ with
$1\le p_i\le \infty$, $1\le i\le m$, $1\le q\le \infty$ and $0<\lambda<mn$.
Suppose also that $\rank(D)=n$. Let $\gamma_i$ and $i_l$  be defined by (\ref{eq:rk def})
and (\ref{eq:rk}), respectively.
Then  $J_{\lambda,D}$ if bounded from $L^{\vec p}$ to
$L^q$
if and only if
for any $f\in L^{\vec p}$ and $h\in L^{q'}$,
\[
   \left| \int_{\bbR^{(m+1)n}}\frac{f(y_1,\ldots,y_m)h(x)
        dy_1\ldots dy_m dx}{ W(x,y)^{\lambda}}\right|
     \lesssim  \|f \|_{L^{\vec p }} \|h\|_{L^{q'}},
\]
where
\[
  W(x,y) = \sum_{l=1}^n |x^{(l)} - y_{j_l}^{(k_l)}|
  + \sum_{\substack{1\le j\le m, 1\le k\le n\\
  (j,k)\ne (j_l,k_l), 1\le l\le n}} |y_j^{(k)}|
\]
and
\begin{equation}\label{eq:jl kl}
 \{(j_l,k_l):\, 1\le l\le n\}
  =\bigcup_{1\le s\le \nu}\{(i_s,t):\, n+1-(\gamma_{i_s}-\gamma_{i_{s+1}})\le t\le n\}
\end{equation}
satisfying $j_l\le j_{l+1}$ and if $j_l= j_{l+1}$, then $k_l<k_{l+1}$.
\end{Lemma}

\begin{proof}
Since $q\ge 1$,
  $J_{\lambda,D}$ if bounded from $L^{\vec p}$ to
$L^q$
if and only if
for any $f\in L^{\vec p}$ and $h\in L^{q'}$,
\begin{equation}\label{eq:s:e10}
   \left| \int_{\bbR^{(m+1)n}}
   \frac{ f(y_1,\ldots,y_m) h(x)dx dy_1\ldots dy_m}
   {(|D_1x-y_1|+ \ldots+|D_mx-y_m|)^{\lambda}}
 \right|  \lesssim \|f\|_{L^{\vec p}} \|h\|_{L^{q'}}.
\end{equation}

Denote the $k$-th row of $D_j$ by $D_{j,k}$.
By the hypothesis of $D$, there exists a
sequence of row vectors
$\{D_{j_l,k_l}:\, 1\le l \le n\}$, which is
linearly independent, such that for $1\le i\le m$,
$\{D_{j_l,k_l}:\,   j_l\ge i, 1\le l \le n\}$
is the maximal linearly independent set of
$\{D_{s,k}:\, i\le s\le m, 1\le k\le n\}$.

To avoid too complicated notations, we consider only the case
\[
  \gamma_1 > \gamma_2 = \ldots = \gamma_m>0.
\]
Other cases
  can be proved similarly.
In this case,   the maximal linearly independent set
comes from rows of $D_1$ and $D_m$.

Observe that for each $1\le i\le m$, the integration order of $y_i^{(1)}$, $\ldots$, $y_{i}^{(n)}$ is switchable in the computation of the $L^{\vec p}$ norm of $f$. We may assume that  the maximal  linearly independent set is
$ \{D_{1,j}:\, n-r+1\le j\le n\}
\cup \{D_{m,j}:\, r+1\le j\le n\}$, where $r=\gamma_1-\gamma_2$.

Denote $y_i=(y_i^{(1)}, \ldots, y_i^{(n)})^*$
and $y= (y_1^{(1)}, \ldots, y_1^{(n)},\ldots,
 y_m^{(1)}, \ldots, y_m^{(n)} )^*$.
We have
\[
  |Dx-y| \approx  \sum_{i=1}^m | D_ix-y_i|.
\]
By the choice of the maximal  linearly independent set, there is some  $mn\times mn$ invertible matrix $P$, which is the product of elementary matrices,  such that
\[
  P(D, \ -I_{mn}) = (PD, \ U),
\]
where only the $(n-r+1)$-th, $\ldots$, the  $n$-th
and the last $n-r$ rows of $PD$ are non-zero vectors
and $U$ is an upper
triangular matrix    whose diagonal entries are
$-1$'s. Denote the submatrix consisting of the
$n$ non-zero rows of $PD$ by $G$.
We have
\[
  PD= \begin{pmatrix}
   \tilde D_1 \\
   \vdots \\
   \tilde D_m  \\
  \end{pmatrix},\quad
    U=\begin{pmatrix}
  -1 & * & * & \ldots  & *  \\
     & -1  & * & \ldots  & *  \\
       &  & -1  &  \ldots  & *  \\
   & & &  \vdots \\
   & &        & & -1
  \end{pmatrix}.
\]
where $\tilde D_i=0$, $2\le i\le m-1$,
\[
  \tilde D_1 = \begin{pmatrix}
  0 \\
  \vdots \\
  0 \\
  D_{1,n-r+1}\\
  \vdots\\
  D_{1,n}
  \end{pmatrix},\quad
  \tilde D_m = \begin{pmatrix}
  0 \\
  \vdots \\
  0 \\
  D_{m,r+1}\\
  \vdots\\
  D_{m,n}
  \end{pmatrix},
\]
and
\[
  \tilde D_1 G^{-1} = \begin{pmatrix}
  0 & 0\\
  I_{r} & 0
  \end{pmatrix},
  \qquad
  \tilde D_m G^{-1} = \begin{pmatrix}
  0 & 0\\
  0 & I_{n-r}
  \end{pmatrix}.
\]
 Since
$|Dx-y| \approx |P(Dx-y)|$, by
substituting $P(Dx-y) $ for $Dx-y$
and a change of variable of the form $x\rightarrow G^{-1}x$ in (\ref{eq:s:e10}),
we may assume  that (\ref{eq:s:e10})
is of the form
\[
   \left| \int_{\bbR^{(m+1)n}}
   \frac{ f(y_1,\ldots,y_m) h(x)dxdy_1\ldots dy_m }
   {W(x,y)^{\lambda}}
 \right|  \lesssim \|f\|_{L^{\vec p}} \|h\|_{L^{q'}},
\]
where
\begin{align*}
W(x,y) &=
   \sum_{l=1}^{n-r}| y_1^{(l)} - w_{1,l}(y)|
   +\sum_{l=1}^r| x^{(l)}-y_1^{(n-r+l)} +
   w_{1,n-r+l}(y) | \\
 & \quad
  + \sum_{i=2}^{m-1}\sum_{l=1}^n |y_i^{(l)} - w_{i,l}(y)|
  +\sum_{l=1}^r| y_m^{(l)}-w_{m,l}(y)|\\
  &\quad +
   \sum_{l=r+1}^n| x^{(l)}-y_m^{(l)} + w_{m,l}(y)|
\end{align*}
and $w_{i,l}(y)$ is a linear combination of
$y_i^{(l+1)}$, $\ldots$, $y_{i}^{(n)}$,
$y_{i+1}^{(1)}$, $\ldots$, $y_{i+1}^{(n)}$,
$\ldots$,
$y_m^{(1)}$, $\ldots$, $y_m^{(n)}$ ($w_{m,n}=0$).
Observe that
\begin{align*}
  \|f\|_{L^{\vec p}}
   &=  \| f(y_1^{(1)} +w_{1,1}(y),\ldots, y_1^{(n)}
     +w_{1,n}(y),\ldots, \\
   &\qquad   \qquad
  y_m^{(1)}+w_{m,1}(y), \ldots, y_m^{(n)}+w_{m,n}(y))\|_{L^{\vec p}}.
\end{align*}
By a change of variable  of the form
$(y_1^{(1)},\ldots, y_1^{(n)},\ldots, y_m^{(1)}, \ldots, y_m^{(n)})\rightarrow (y_1^{(1)}+w_{1,1}(y)$, $\ldots$, $y_1^{(n)}+w_{1,n}(y),\ldots,
  y_m^{(1)}+w_{m,1}(y), \ldots, y_m^{(n)}+w_{m,n}(y))$,
we may further assume that $W(x,y)$ is of the form
\begin{align*}
W(x,y) &=
   \sum_{l=1}^{n-r}| y_1^{(l)} |
   +\sum_{l=1}^r| x^{(l)}-y_1^{(n-r+l)}  | \\
 & \quad
  + \sum_{i=2}^{m-1}\sum_{l=1}^n |y_i^{(l)} |
  +\sum_{l=1}^r| y_m^{(l)} |  +
   \sum_{l=r+1}^n| x^{(l)}-y_m^{(l)}  |.
\end{align*}
This completes the proof.
\end{proof}

We
split the proof of Theorem~\ref{thm:factional:D} into two  parts:
one is for the necessity, and the other one is for the sufficiency.

\subsection{Proof of Theorem~\ref{thm:factional:D}: The Necessity}

In this subsection, we give a proof for the necessity part in Theorem~\ref{thm:factional:D}.

First, we show that $q<\infty$.
Assume on the contrary that $q=\infty$.
  Then for   any $f\in L^{\vec p}$ and $h\in L^1(\bbR^n)$, we have
\begin{align*}
\int_{\bbR^{(m+1)n}}  \frac{f(y_1,\ldots, y_m) h(x)dy_1 \ldots dy_m dx}
{(|D_1 x-y_1|+\ldots+|D_m x-y_m|)^{\lambda}}
\lesssim \|f\|_{L^{\vec p}} \|h\|_{L^1}.
\end{align*}
Hence
\begin{equation} \label{eq:ss:ee1}
  \bigg\|\int_{\bbR^n}  \frac{ h(x)}{(|D_1 x-y_1|+\ldots+|D_m x-y_m|)^{\lambda}}  dx\bigg\|_{L^{\vec p'}}
 \lesssim   \|h\|_{L^1}.
\end{equation}
Set $h=\chi^{}_{\{|x|<1\}}$. We have
\begin{align*}
g(y_1,\ldots,y_m)
&:= \int_{\bbR^n}  \frac{ h(x)}{(|D_1 x-y_1|+\ldots+|D_m x-y_m|)^{\lambda}}  dx \\
&\gtrsim \frac{1}{(|y_1|+\ldots + |y_m|)^{\lambda}},\quad \mathrm{for}\,\, |y_i|>2, 1\le i\le m.
\end{align*}
On the other hand, we see from the homogeneity condition
(\ref{eq:homo:p}) that
\[
  \lambda - \frac{n}{p'_1} - \ldots -
    \frac{n}{p'_m}=0.
\]
Hence $\|g\|_{L^{\vec p'}} =\infty$,
which contradicts with (\ref{eq:ss:ee1}).

Next we show that $q\ge 1$.
If $0<q<1$, then for non-negative
functions $f\in L^{\vec p}$, we see from Minkowski's inequality
that
\begin{align*}
&
\int_{\bbR^{mn}}
 \left\| \frac{f(y_1,\ldots, y_m) }
{(|D_1 x-y_1|+\ldots+|D_m x-y_m|)^{\lambda}}
\right\|_{L_x^q}
dy_1 \ldots dy_m\\
&\le
\left\|\int_{\bbR^{mn}}
  \frac{f(y_1,\ldots, y_m)dy_1 \ldots dy_m }
{(|D_1 x-y_1|+\ldots+|D_m x-y_m|)^{\lambda}}
\right\|_{L_x^q} \\
&\lesssim \|f\|_{L^{\vec p}}.
\end{align*}
Since $|D_1 x-y_1|+\ldots+|D_m x-y_m| \lesssim |x|+\sum_{i=1}^m|y_i|$,  we have
\[
  \int_{\bbR^{mn}}
 \left\| \frac{f(y_1,\ldots, y_m) }
{(|x|+\sum_{i=1}^m|y_i| )^{\lambda}}
\right\|_{L_x^q}
dy_1 \ldots dy_m
\lesssim  \|f\|_{L^{\vec p}}.
\]
That is,
\[
\left|
  \int_{\bbR^{mn}}
 \frac{f(y_1,\ldots, y_m) dy_1 \ldots dy_m }
 {(\sum_{i=1}^m|y_i|)^{\lambda-n/q}}
\right|
\lesssim  \|f\|_{L^{\vec p}},
\qquad \forall f\in L^{\vec p},
\]
which is impossible since $1/(\sum_{i=1}^m|y_i|)^{\lambda-n/q}\not\in
L^{\vec p '}$. Hence $q\ge 1$.

(i)\,\,
If $\rank(D)<n$, then there is some $n\times n$ invertible matrix $U$ such that one column of $DU$
is zero. Without loss of generality, assume that
the last column of $DU$ is zero.
Then $J_{\lambda,DU}f(x)$ is independent of $x^{(n)}$.
By a change of variable of the form
$x\rightarrow Ux$, we have for positive $f$,
\begin{align*}
\int_{\bbR^n} |J_{\lambda,D}f(x)|^q dx
&\approx
\int_{\bbR^n} |J_{\lambda,DU}f(x)|^q dx \\
&=\int_{\bbR} dx^{(n)}
  \int_{\bbR^{n-1}} |J_{\lambda,DU}f(x)|^q dx^{(1)}\ldots dx^{(n-1)} \\
&=\infty,
\end{align*}
which contradicts with the boundedness of $J_{\lambda,D}$.
Hence $\rank(D)=n$.

(ii)\,\,
For any  $f\in L^{\vec p}$,
considering the function $f_a  = a^{n/p_1+\ldots + n/p_m}
f(a\cdot)$, where $a>0$, we get the homogeneity
condition (\ref{eq:homo:p}).

(iii)\,\,
Now we show that there is some $1\le i\le m$ such that $1<p_i<\infty$.
Assume on the contrary that $p_i=1$ or $\infty$ for every $1\le i\le m$. Let $k=\#\{i:\, p_i=\infty, 1\le i\le m\}$, where $\#$ stands for the cardinality of a set.
We see from the homogeneity condition (\ref{eq:homo:p}) that
$\lambda - kn = n/q$.

Set $f = \chi^{}_E$, where
$E=\{(y_1,\ldots,y_m):\, |y_i|<1 \mbox{\, if\, } p_i=1\}$.
Then we have
\[
  J_{\lambda,D}f(x)
  \gtrsim \frac{1}{(1+|x|)^{\lambda-kn}}
  =\frac{1}{(1+|x|)^{n/q}}\not\in L^q.
\]
Consequently, there is some $i$ such that $1<p_i<\infty$.

(iv)\,\,
It remains to prove (\ref{eq:s:e5:D}).
First, we show that   $q\ge p_{i_l}$ for all  $1\le l\le \nu$.

By the definition of $\gamma_i$,
there is some $n\times n$ invertible submatrix of $D$  which consists of $\gamma_i-\gamma_{i+1}$ rows of $D_i$, $1\le i\le m$.
Consequently, there is some $mn\times mn$ upper triangular   matrix $P$ such that
\[
  P\begin{pmatrix}
  D_1 & -I &  &  \\
  D_2 &    &-I & \\
      &\ldots &  & \\
  D_{m}&   &   & -I
  \end{pmatrix}
=
\begin{pmatrix}
  \tilde D_1 & -Q_1 & * & * \\
  \tilde D_2 &    &-Q_2 & * \\
      &\ldots &  & *\\
  \tilde D_{m}&   &   & -Q_m
  \end{pmatrix},
\]
where $Q_i$ are $n\times n$ invertible matrices, $\tilde D_i$ are $n\times n$ matrices
and there are exactly $\gamma_i-\gamma_{i+1}$ rows of
$\tilde D_i$ which are none zero vectors.
Hence $\|J_{\lambda,D}f\|_{L^q}\lesssim \|f\|_{L^{\vec p}}$ is equivalent to
\[
  \left\| \int_{\bbR^{mn}} \frac{f(y_1,\ldots, y_m)}{(\sum_{i=1}^m|\tilde D_ix-Q_iy_i+w_i(y)|)^{\lambda}} dy_1\ldots dy_m\right\|_{L^q}
  \lesssim \|f\|_{L^{\vec p}},
\]
where $w_i(y)$ is a linear combination of
$y_{i+1}, \ldots, y_m$ for $1\le i\le m-1$
and $w_m(y)=0$.

By a change of variable of the form
$y_i\rightarrow Q_i^{-1}y_i + Q_i^{-1}w_i(y)$, we get
\begin{align*}
  &\left\| \int_{\bbR^{mn}} \frac{f(Q_1^{-1}y_1+Q_1^{-1}w_1(y),\ldots, Q_m^{-1}y_m+Q_m^{-1}w_m(y))}{(\sum_{i=1}^m|\tilde D_ix-y_i |)^{\lambda}} dy_1\ldots dy_m\right\|_{L^q}
  \\
  &\lesssim \|f\|_{L^{\vec p}}.
\end{align*}
By substituting   $f(Q_1\cdot, \ldots, Q_m\cdot)$ for $f$ and
$f(y_1+w_1(y),\ldots, y_m+w_m(y))$ for $f$
successively, we get
\begin{align*}
  &\left\| \int_{\bbR^{mn}} \frac{f( y_1 ,\ldots, y_m)}{(\sum_{i=1}^m|\tilde D_ix-y_i |)^{\lambda}} dy_1\ldots dy_m\right\|_{L^q}
   \lesssim \|f\|_{L^{\vec p}}.
\end{align*}
If $p_{i_l}=\infty$ for some $l$, by setting
$\vec{\tilde p}
=(p_1,\ldots,p_{i_l-1},p_{i_l+1},\ldots,p_m)$
and
$f(y) = \tilde f(\ldots, y_{i_l-1}$, $y_{i_l+1}$, $\ldots)$,
we get
\begin{align}
   \left\| \int_{\bbR^{(m-1)n}} \frac{\tilde f(\ldots, y_{i_l-1}, y_{i_l+1},\ldots)dy_1\ldots
   dy_{i_l-1} dy_{i_l+1}\ldots dy_m }{(\sum_{\substack{1\le i\le m\\
  i\ne i_l}}|\tilde D_ix-y_i |)^{\lambda-n}} \right\|_{L^q}
  &\lesssim \|\tilde f\|_{L^{\vec {\tilde p}}}. \label{eq:w:e63}
\end{align}
But the rank of the $(m-1)n\times n$ matrix
$(\tilde D_i)_{i\ne i_{l}}$ is less than $n$. Hence
(\ref{eq:w:e63}) can not be true.
Therefore, $p_{i_l}<\infty$ for all $1\le l\le \nu$.

Take some $1\le l\le \nu$.
Then there is some $z\in\bbR^n$ such that
\[
  D_{i_l}z\ne 0\quad \mathrm{and}\quad
  D_kz=0,\,\, i_l+1\le k\le m.
\]
Denote
\[
    f_a = f(\cdot - aDz).
\]
It is easy to see that $J_{\lambda,D} f_a =
(J_{\lambda,D}  f)(\cdot-az)$.
Hence
\begin{align}
\|J_{\lambda,D} f_a + J_{\lambda,D}   f\|_{L^q}
&=\|(J_{\lambda,D}   f) (\cdot-az)+ J_{\lambda,D}   f\|_{L^q}
 \nonumber  \\
&\rightarrow 2^{1/q} \|   J_{\lambda,D}   f\|_{L^q},
\quad a\rightarrow\infty. \label{eq:a:e1}
\end{align}
On the other hand, we see from Lemma~\ref{Lm: translation} that
\[
  \|f(\cdot - aDz) + f\|_{L^{\vec p}}
  \rightarrow 2^{1/p_{i_l}}\|f\|_{L^{\vec p}},\qquad \mathrm{as\ } \, a\rightarrow \infty.
\]
Hence
\[
  \|J_{\lambda,D} f_a+ J_{\lambda,D}   f\|_{L^q}
  \le  \|J_{\lambda,D}\|\cdot \|f(\cdot-aDz) + f\|_{L^{\vec p}} \rightarrow 2^{1/p_{i_l}}\|J_{\lambda,D}\| \cdot \|f\|_{L^{\vec p}}.
\]
By (\ref{eq:a:e1}), we get
\[
  2^{1/q}\|J_{\lambda,D}f\|_{L^q}
  \le 2^{1/p_{i_l}}\|J_{\lambda,D}\|\cdot  \|f\|_{L^{\vec p}}.
\]
Hence $q\ge p_{i_l}$, $1\le l\le \nu$.

Next we prove that $q\ge p_{i_0}$.
We begin with the case  $i_0=m$.
There are two subcases.

(A1)\,\,  $D_m\ne 0$.

  In this case, $\gamma_m>\gamma_{m+1}=0$. We see from previous arguments that  $q\ge p_m$.

(A2)\,\,   $D_m=0$.
Since $q\ge \max\{p_{i_l}:\, 1\le l\le \nu\}\ge 1$, $J_{\lambda,D}$ is bounded if
and only if
\[
    \bigg\|\int_{\bbR^n}  \frac{ h(x)}{(|D_1 x-y_1|+\ldots+|D_m x-y_m|)^{\lambda}}  dx\bigg\|_{L^{\vec p'}}
 \lesssim   \|h\|_{L^{q'}}.
\]
By Lemma~\ref{Lm:Dx-y}, we have
\[
    \bigg\|\int_{\bbR^n}  \frac{ h(x)}{W(x,y)^{\lambda}}  dx\bigg\|_{L^{\vec p'}}
 \lesssim   \|h\|_{L^{q'}},
\]
where
\[
  W(x,y) = \sum_{l=1}^n |x^{(l)} - y_{j_l}^{(k_l)}|
  + \sum_{\substack{1\le j\le m, 1\le k\le n\\
  (j,k)\ne (j_l,k_l), 1\le l\le n}} |y_j^{(k)}|
\]
and $\{(j_l,k_l):\, 1\le l\le n\}$ is defined by (\ref{eq:jl kl}).

If $q=1$, by setting $h\equiv 1$, we get
\[
 \left\| \int_{\bbR^n} \frac{dx}{W(x,y)^{\lambda}}
 \right\|_{L^{\vec p'}}
\approx
   \left\|   \frac{1}{(\sum_{\substack{1\le j\le m, 1\le k\le n\\
  (j,k)\ne (j_l,k_l), 1\le l\le n}} |y_j^{(k)}|)^{\lambda-n}}
 \right\|_{L^{\vec p'}}
 =\infty,
\]
which is a contradiction. Hence $q>1$.

Set $h(x) = 1/ |x|^{n/q'} (\log 1/|x|)^{(1+\varepsilon)/q'}$
for $|x|<1/2$ and $0$ for others,
where $\varepsilon>0$ is a constant,
we have $h\in L^{q'}$.

For $|x|<\sum_{i=1}^m |y_i|<1/2$, we have $\sum_{i=1}^m|D_i x-y_i|
\lesssim \sum_{i=1}^m |y_i|$. Hence
\begin{align}
&\int_{\bbR^n}  \frac{h(x)dx}{(\sum_{i=1}^m|D_i x-y_i|)^{\lambda}}  dx
  \nonumber \\
&\gtrsim
   \int_{(\sum_{i=1}^m |y_i|)^2<|x|<\sum_{i=1}^m |y_i|}
    \frac{ dx}{(\sum_{i=1}^m|D_i x-y_i|)^{\lambda}|x|^{n/q'} (\log 1/|x|)^{(1+\varepsilon)/q'}}
   \nonumber  \\
&\gtrsim
    \frac{ 1}{(\sum_{i=1}^m |y_i|)^{\lambda+n/q'-n} (\log 1/(\sum_{i=1}^m |y_i|))^{(1+\varepsilon)/q'}}. \label{eq:s:e42}
\end{align}
Therefore,
\begin{align*}
& \left\|
    \frac{ \chi^{}_{\{\sum_{i=1}^m |y_i|<1/2\}}(y_1,\ldots,y_m)}
    {(\sum_{i=1}^m |y_i|)^{\lambda+n/q'-n} (\log 1/(\sum_{i=1}^m |y_i|))^{(1+\varepsilon)/q'}}
    \right\|_{L_{y_1}^{p'_1}}\\
&\gtrsim
\left\|
    \frac{ \chi^{}_{\{\sum_{i=2}^m |y_i|<1/4\}}(y_2,\ldots,y_m)
     \chi^{}_{\{  (\sum_{i=2}^m |y_i|)^2<|y_1|<\sum_{i=2}^m |y_i|\}}(y_1)
    }
    {(\sum_{i=1}^m |y_i|)^{\lambda+n/q'-n} (\log 1/(\sum_{i=1}^m |y_i|))^{(1+\varepsilon)/q'}}
    \right\|_{L_{y_1}^{p'_1}}\\
&\gtrsim
    \frac{ \chi^{}_{\{\sum_{i=2}^m |y_i|<1/4\}}(y_2,\ldots,y_m)
          }
    {(\sum_{i=2}^m |y_i|)^{\lambda-n/q-n/p'_1} (\log 1/(\sum_{i=2}^m |y_i|))^{(1+\varepsilon)/q'}}.
\end{align*}
Computing the $L_{y_i}^{p'_i}$ norm successively, $2\le i\le m-1$, we get
\begin{align*}
& \left\|
    \frac{ \chi^{}_{\{\sum_{i=1}^m |y_i|<1/2\}}(y_1,\ldots,y_m)}
    {(\sum_{i=1}^m |y_i|)^{\lambda+n/q'-n} (\log 1/(\sum_{i=1}^m |y_i|))^{(1+\varepsilon)/q'}}
    \right\|_{L_{(y_1,\ldots,y_{m-1})}^{(p'_1,\ldots,p'_{m-1})}}\\
&\gtrsim
    \frac{ \chi^{}_{\{|y_m|<1/2^m\}}( y_m) }
    {  |y_m|^{\lambda-n/q-n/p'_1-\ldots-n/p'_{m-1}} (\log 1/  |y_m|)^{(1+\varepsilon)/q'}}.
\end{align*}
Note that
$\lambda-n/q-n/p'_1-\ldots-n/p'_{m-1} = n/p'_m$.
If $q< p_m$, then $p'_m<q'$. Hence we can choose $\varepsilon>0$
small enough such that
 $p'_m(1+\varepsilon)/q'<1$.
Therefore,
\[
\Big\|
  \frac{ \chi^{}_{\{|y_m|<1/2^m\}}( y_m) }
    {  |y_m|^{\lambda-n/q-n/p'_1-\ldots-n/p'_{m-1}} (\log 1/  |y_m|)^{(1+\varepsilon)/q'}}\Big\|_{L_{y_m}^{p'_m}} = \infty.
\]
Consequently,
\[
  \bigg\|\int_{\bbR^n}  \frac{ h(x)}{(|D_1 x-y_1|+\ldots+|D_m x-y_m|)^{\lambda}}  dx\bigg\|_{L^{\vec p'}}
=\infty,
\]
which is a contradiction.

Now we consider the case
$i_0<m$.
In this case,
\[
  \frac{1}{p_1}+\ldots+ \frac{1}{p_{i_0}}
  = \frac{1}{q} + \frac{i_0n-\lambda}{n}.
\]
There are two subcases.

(B1)\,\, $\gamma_{i_0}>\gamma_{i_0+1}$.

In this case, there is some $1\le l\le \nu$ such that $i_0=i_l$.
Hence $q\ge p_{i_0}$.

(B2)\,\, $\gamma_{i_0}= \gamma_{i_0+1}$.

By Lemma~\ref{Lm:end points D},
for any $h\in L^{q'}$,
\begin{equation}\label{eq:s:e12}
     \left\| \int_{\bbR^n}\frac{h(x)dx}{(\sum_{i=1}^{i_0}|D_ix-y_i|
     +\sum_{i=i_0+1}^m|D_i x|)^{\lambda}} \right\|_{L^{\vec {\tilde p}'}}
     \lesssim \|h\|_{L^{q'}},
\end{equation}
where $\vec{\tilde p} = (p_1,\ldots, p_{i_0})$.

By Lemma~\ref{Lm:Dx-y}, (\ref{eq:s:e12}) is equivalent to
\begin{equation}\label{eq:s:e12a}
     \left\| \int_{\bbR^n}\frac{h(x)dx}
      {W(x,y)^{\lambda}} \right\|_{L^{\vec {\tilde p}'}}
     \lesssim \|h\|_{L^{q'}},
\end{equation}
where
\[
  W(x,y) = \sum_{\substack{1\le l\le n \\
     j_l\le i_0}} |x^{(l)} - y_{j_l}^{(k_l)}|
  +\sum_{l=n-\gamma_{i_0}+1 }^n |x^{(l)} |
  + \sum_{\substack{1\le j\le i_0, 1\le k\le n\\
  (j,k)\ne (j_l,k_l), 1\le l\le n}} |y_j^{(k)}|
\]
and $\{(j_l,k_l):\, 1\le l\le n\}$ is defined by (\ref{eq:jl kl}).

If $q=1$, by setting $h\equiv 1$, we get
\begin{align*}
     \left\| \int_{\bbR^n}\frac{h(x)dx}
      {W(x,y)^{\lambda}} \right\|_{L^{\vec {\tilde p}'}}
\approx
  \bigg\|  \frac{1}
      {(\sum_{\substack{1\le j\le i_0, 1\le k\le n\\
  (j,k)\ne (j_l,k_l), 1\le l\le n}} |y_j^{(k)}|)^{\lambda-n}} \bigg\|_{L^{\vec {\tilde p}'}}  =\infty,
 \end{align*}
which is a contradiction. Hence $q>1$.

Assume that $ q< p_{i_0}$.
Let  $h$ be defined as in Case (A2).
For $|x|<\sum_{i=1}^{i_0} |y_i|$ $<1/2$, we have $\sum_{i=1}^{i_0}|D_ix-y_i|
     +\sum_{i=i_0+1}^{i_0}|D_i x|
\lesssim \sum_{i=1}^{i_0} |y_i|$. Hence
\begin{align*}
&\int_{\bbR^n}  \frac{h(x)dx}{(\sum_{i=1}^{i_0}|D_ix-y_i|
     +\sum_{i=i_0+1}^{i_0}|D_i x| )^{\lambda}}  dx
  \\
&\gtrsim
   \int_{\substack{|x|>(\sum_{i=1}^{i_0} |y_i|)^2\\
    |x|<\sum_{i=1}^{i_0} |y_i|}}
    \frac{ dx}{(\sum_{i=1}^{i_0}\!|D_ix\!-\!y_i|
    \! +\!\sum_{i=i_0+1}^{i_0}\!|D_i x| )^{\lambda}|x|^{n/q'} (\log 1/|x|)^{(1+\varepsilon)/q'}}
    \\
&\gtrsim
    \frac{ 1}{(\sum_{i=1}^{i_0} |y_i|)^{\lambda+n/q'-n} (\log 1/(\sum_{i=1}^{i_0} |y_i|))^{(1+\varepsilon)/q'}}.
\end{align*}
That is, the inequality (\ref{eq:s:e42}) is true with $m$ being replaced by $i_0$.
Therefore, with almost the same arguments (replacing $m$ by $i_0$) we get a contradiction. Hence
$q\ge p_{i_0}$.

Finally, we show that $\min\{p_{i_l}:\,   1\le l\le \nu\}<q$.

Since $q\ge p_{i_0}>1$, it suffices to consider the case     $\min\{p_{i_l}:\,   1\le l\le \nu\}>1$. In this case, $i_l\le i_0$
for all $1\le l\le \nu$.

Assume on the contrary   that $\min\{p_{i_l}:\, 1\le l\le \nu\}=q$.
Since $\max\{p_{i_l}:\, 1\le l\le \nu\}\le q$, we have
$q = p_{i_1} = \ldots = p_{i_{\nu}}$.

By Lemmas~\ref{Lm:end points D} and \ref{Lm:Dx-y}, $J_{\lambda,D}$ is bounded from $L^{\vec p}$
to $L^q$ if and only if
\begin{equation}\label{eq:s:e9}
   \left\| \int_{\bbR^{ n}}
   \frac{   h(x)dx  }
   {W(x,y)^{\lambda}}
 \right\|_{L^{\vec p'}}  \lesssim   \|h\|_{L^{q'}},
\end{equation}
where
\[
  W(x,y) = \sum_{l=1}^{n} |x^{(l)} - y_{j_l}^{(k_l)}|
       + w(y),  \quad w(y) =  \sum_{\substack{1\le j\le i_0 , 1\le k\le n\\
  (j,k)\ne (j_l,k_l), 1\le l\le n }} |y_j^{(k)}| ,
\]
and $\{(j_l,k_l):\, 1\le l\le n\}$ is defined by (\ref{eq:jl kl}).

 Let
 $
h(x)  =
\chi_{\{ |x^{(l)}|\le 1, 1\le l\le n\}}(x)$  and $\delta  = 1/(n+2)$.
We have
\begin{align*}
&\Big\{x^{(l)}:\,  |x^{(l)} - y_{j_l}^{(k_l)}|  <
     \sum_{s=l+1}^n  |x^{(s)} - y_{j_s}^{(k_s)}|
          +w(y),  \\
&\qquad \qquad       |x^{(s)} - y_{j_s}^{(k_s)}| < \delta  , l+1\le s\le n,
      \\
  & \qquad \qquad w(y)<\delta  , |y_{j_{\bar l}}^{(k_{\bar l})}|\le \delta ,  1\le {\bar l}\le n
 \Big\}\\
&
 \subset\{x^{(l)}:\, |x^{(l)}| < 1\},\qquad 1\le l\le n .
\end{align*}
Let $E = \{(y_1,\ldots,y_m):\,  w(y)<\delta  , |y_{j_l}^{(k_l)}|\le \delta  ,  1\le l\le n  \}$.
When $y=(y_1,\ldots,y_m)\in E$ and
 $|x^{(s)} - y_{j_s}^{(k_s)}| < \delta$, $2\le s\le n$,
 we have
\begin{align*}
&\int_{R} \frac{h(x)}{W(x,y)^{\lambda}} dx^{(1)}
\ge
\int_{|x^{(1)} - y_{j_1}^{(k_1)}|  <
      \sum_{s=2}^{n} |x^{(s)} - y_{j_s}^{(k_s)}|
             +w(y)}
     \frac{h(x)}{W(x,y)^{\lambda}} dx^{(1)}
     \\
&\gtrsim \frac{1}{( \sum_{s=2}^{n } |x^{(s)} - y_{j_s}^{(k_s)}|
     +w(y))^{\lambda-1}} .
\end{align*}
Integrating with respect to $x^{(2)}$, $\ldots$, $x^{(n )}$ successively,
we get
\begin{align}
\int_{R^n} \frac{h(x)dx }{W(x,y)^{\lambda}}
\gtrsim \frac{\chi^{}_{E}(y)}{  w(y)^{\lambda-n }}. \label{eq:s:h}
\end{align}
Set $n_i = n- (\gamma_i-\gamma_{i+1})$, $1\le i\le i_0$.
We have $w(y) = \sum_{i=1}^{i_0} \sum_{l=1}^{n_i} |x_i^{(l)}|$.
Since $q=p_{i_1} = \ldots = p_{i_{\nu}}$ and $\lambda = \sum_{i=1}^{i_0} n/p'_i + n/q$, we have
\begin{align*}
  \lambda -  n
  & = \sum_{i=1}^{i_0} \frac{n}{p'_i} + \frac{n}{q} - n\\
  & = \sum_{i=1}^{i_0} \frac{n-(\gamma_i-\gamma_{i+1})}{p'_i}
     + \sum_{i=1}^{i_0} \frac{ \gamma_i-\gamma_{i+1} }{p'_i}  + \frac{n}{q} - n\\
  & = \sum_{i=1}^{i_0} \frac{n-(\gamma_i-\gamma_{i+1})}{p'_i}
     + \sum_{l=1}^{\nu} \frac{ \gamma_{i_l}-\gamma_{i_l+1} }{p'_{i_l}}  + \frac{n}{q} - n\\
  &= \sum_{i=1}^{i_0} \frac{n_i}{p'_i}.
\end{align*}
Hence
\[
  \left\| \frac{\chi^{}_{E}(y)}{w(y)^{\lambda-n}}  \right\|_{L^{\vec p'}}
  =\infty,
\]
which contradicts with (\ref{eq:s:e9}).
This completes the proof for the necessity.
\endproof

\subsection{ Proof of  Theorem~\ref{thm:factional:D}: The Sufficiency}

In this subsection, we prove the sufficiency in Theorem~\ref{thm:factional:D}.
The proof is split into two parts. One  is  for the case    $p_{i_{\nu}}<q$,
and the other one is for the case
 $p_{i_{\nu}}=q$.

\subsubsection{The Case  \texorpdfstring{$p_{i_{\nu}}<q$}{p i nu<q} }

We begin with the case   $i_0=m$.
That is, $1<p_m<\infty$. There are three subcases.

(A1)\,\, $\rank(D_m)=n$.

In this case, $i_{\nu}=m$.
We see from (\ref{eq:homo:p}) that
\[
  \frac{1}{p_m} = \frac{1}{q} +
    \frac{n- (\lambda - n/p'_1-\ldots - n/p'_{m-1})}{n}.
\]
Since $1<p_m<q<\infty$, we have
\[
  0<\lambda - \Big(\frac{n}{p'_1} + \ldots + \frac{n}{p'_{m-1}}\Big)<n.
\]

Fix some $f\in L^{\vec p}$.
For $1< p_1<\infty$, we see from H\"older's inequality that
\begin{align*}
&\int_{\bbR^n} \frac{|f(y_1, \ldots,y_{m})|}{(|D_1x-y_1|+ \ldots+|D_m x-y_{m}|)^{\lambda}} dy_1  \\
&\le \|f(\cdot,y_2,\ldots,y_{m})\|_{L^{p_1}_{y_1}}
   \Big( \int_{\bbR^n} \frac{dy_1}{(|D_1 x-y_1| +\ldots+|D_m x-y_{m}|)^{\lambda p'_1}}    \Big)^{1/p'_1}
\\
&\approx \|f(\cdot,y_2,\ldots,y_{m})\|_{L^{p_1}_{y_1}}
   \Big( \int_0^{\infty} \frac{  r^{n-1}dr}{(r+|D_2 x-y_2|+\ldots+|D_m
   x-y_{m}|)^{\lambda p'_1}}    \Big)^{1/p'_1}
\\
&\lesssim \frac{\|f(\cdot,y_2,\ldots,y_{m})\|_{L^{p_1}_{y_1}}}
    {( |D_2 x-y_2|+\ldots+|D_m x-y_{m}|)^{\lambda - n/p'_1}}.
\end{align*}
Observe that the above
inequality is also true for $p_1=1$ or $p_1=\infty$.
By induction, it is easy to see that
\begin{align*}
&
\int_{\bbR^{(m-1)n}} \frac{f(y_1,\ldots, y_m)}
     {(|D_1x - y_1| + \ldots + |D_m x - y_m|)^{\lambda}}
    dy_1 \ldots dy_{m-1}\\
&\le \frac{\|f(\cdot, \ldots, \cdot,y_m)\|_{L_{(y_1,\ldots, y_{m-1})}^{(p_1,\ldots,p_{m-1})}}}
{|D_mx -y_m|^{\lambda - n/p'_1-\ldots - n/p'_{m-1}}}.
\end{align*}
Since $D_m$ is invertible, we get the conclusion as desired from Proposition~\ref{prop:fractional}.

(A2)\,\, $0<\rank(D_m)<n$.

In this case, we also have $i_{\nu}=m$.
We consider only the case
\[
  \gamma_1>\gamma_2   =\ldots = \gamma_m>0.
\]
Other cases can be proved similarly.

Set $r=\gamma_1 - \gamma_2$.
By Lemma~\ref{Lm:Dx-y}, it suffices to show that
for any $f\in L^{\vec p}$ and $h\in L^{q'}$,
\begin{equation}\label{eq:w:e11}
   \left| \int_{\bbR^{(m+1)n}}
   \frac{ f(y_1,\ldots,y_m) h(x)dxdy_1\ldots dy_m }
   {W(x,y)^{\lambda}}
 \right|  \lesssim \|f\|_{L^{\vec p}} \|h\|_{L^{q'}},
\end{equation}
where
\begin{align*}
W(x,y) &=
   \sum_{l=1}^{n-r}| y_1^{(l)} |
   +\sum_{l=1}^r| x^{(l)}-y_1^{(n-r+l)}  | \\
 & \quad
  + \sum_{i=2}^{m-1}\sum_{l=1}^n |y_i^{(l)} |
  +\sum_{l=1}^r| y_m^{(l)} |  +
   \sum_{l=r+1}^n| x^{(l)}-y_m^{(l)}  |.
\end{align*}

Let us estimate the integral in (\ref{eq:w:e11}).
First, we estimate
\[
    \left| \int_{\bbR^{n-r}}
   \frac{ f(y_1,\ldots,y_m) h(x) dy_1^{(1)}\ldots
   dy_1^{(n-r)}  }
   {W(x,y)^{\lambda}}\right| .
\]
We see from H\"older's inequality that
\begin{align*}
&\hskip -10mm \left| \int_{\bbR^{n-r}}
   \frac{ f(y_1,\ldots,y_m) h(x) dy_1^{(1)}\ldots
   dy_1^{(n-r)}  }
   {W(x,y)^{\lambda}}\right|\\
 &\lesssim
\frac{\| f(y_1,\ldots,y_m)\|_{L_{(y_1^{(1)},\ldots, y_1^{(n-r)})}^{p_1}} |h(x)| }
   {W_{1,n-r}(x,y)^{\lambda-(n-r)/p'_1}},
\end{align*}
where
\begin{align*}
W_{1,n-r}(x,y) &=
      \sum_{l= 1}^r| x^{(l)}-y_1^{(n-r+l)} |
      + \sum_{i=2}^{m-1}\sum_{l=1}^n |y_i^{(l)} | \\
 & \quad
    +\sum_{l=1}^r| y_m^{(l)} |
    +
   \sum_{l=r+1}^n| x^{(l)}-y_m^{(l)} |.
\end{align*}

Next we estimate the integration with respect to
$x^{(1)}$ and $y_1^{(n-r+1)} $.
Choose some real number $s$ such that
\[
  \frac{1}{p_1}+ \frac{1}{s} = \frac{1}{q} + 1.
\]
Since $\gamma_1>\gamma_2$, we see from the hypothesis
$q\ge
\max\{p_{i_l}:\, 1\le l\le \nu\}$
 that $q\ge p_1$.
Hence $s\ge 1$.
By Young's inequality, we get
\begin{align*}
&
    \left| \int_{\bbR^{2}}
   \frac{\| f(y_1,\ldots,y_m)\|_{L_{(y_1^{(1)},\ldots, y_1^{(n-r)})}^{p_1}} |h(x)| }
   {W_{1,n-r}(x,y)^{\lambda-(n-r)/p'_1}} dx^{(1)} dy_1^{(n-r+1)}\right| \\
 &  \lesssim
   \frac{ \| f(y_1,\ldots,y_m)\|_{L_{(y_1^{(1)},\ldots, y_1^{(n-r+1)})}^{p_1}} \|h(x)\|_{L_{x^{(1)}}^{q'}}   }
   {W_{1,n-r+1}(x,y)^{\lambda-(n-r+1)/p'_1-1/q}},
\end{align*}
where
\begin{align*}
  W_{1,n-r+1}(x,y)  &=
      \sum_{l= 2}^r| x^{(l)}-y_1^{(n-r+l)} |
      + \sum_{i=2}^{m-1}\sum_{l=1}^n |y_i^{(l)} | \\
 & \quad
    +\sum_{l=1}^r| y_m^{(l)} |
    +
   \sum_{l=r+1}^n| x^{(l)}-y_m^{(l)} |.
\end{align*}
Similar arguments show that
\begin{align*}
  &  \left| \int_{\bbR^{n+r}}
   \frac{ f(y_1,\ldots,y_m) h(x)dx^{(1)}\ldots  dx^{(r)} dy_1  }
   {W(x,y)^{\lambda}}\right|\\
 &  \lesssim
   \frac{ \|f(y_1,\ldots,y_m)\|_{L_{y_1}^{p_1}} \|h(x)\|_{L_{( x^{(1)},\ldots,x^{(r)})}^{q'}}   }
   {W_{2,1}(x,y)^{\lambda-n/p'_1-r/q}},
\end{align*}
where
\begin{align*}
  W_{2,1}(x,y)  =
        \sum_{i=2}^{m-1}\sum_{l=1}^n |y_i^{(l)} |
    +\sum_{l=1}^r| y_m^{(l)} |
    +
   \sum_{l=r+1}^n| x^{(l)}-y_m^{(l)} |.
\end{align*}

Next we compute the integral with respect to
 $y_2^{(1)}$, $\ldots$, $y_2^{(n)}$, $\ldots$,
 $y_m^{(1)}$, $\ldots$, $y_{m}^{(r)}$, successively.
We see from H\"older's inequality that
\begin{align*}
  &  \left| \int_{\bbR^{(m-1)n+2r}}
   \frac{ f(y_1,\ldots,y_m) h(x)dx^{(1)}\ldots  dx^{(r)} dy_1\ldots dy_{m-1} dy_m^{(1)} \ldots dy_m^{(r)} }
   {W(x,y)^{\lambda}}\right|\\
 &  \lesssim
   \frac{ \|f(y_1,\ldots,y_m)\|_{L_{( y_m^{(1)},\ldots,y_m^{(r)})}^{p_m}
   (L_{(y_1,\ldots,y_{m-1})}^{(p_1,\ldots,p_{m-1})})} \|h(x)\|_{L_{( x^{(1)},\ldots,x^{(r)})}^{q'}}   }
   {(\sum_{l=r+1}^n| x^{(l)}- y_m^{(l)} |)^{(n-r)/q+(n-r)/p'_m}}.
\end{align*}
Replacing $n$ by $n-r$ in Proposition~\ref{prop:fractional}, we
see from the above inequality that
\begin{align*}
     \left| \int_{\bbR^{(m+1)n}}
   \frac{ f(y_1,\ldots,y_m) h(x)dx  dy_1\ldots dy_m }
   {W(x,y)^{\lambda}}\right|
   \lesssim \|f\|_{L^{\vec p}} \|h\|_{L^{q'}}.
\end{align*}

(A3)\,\, $D_m=0$.

In this case, $i_{\nu}<m$.
Using Lemma~\ref{Lm:Dx-y} again, we only need to show that
\begin{equation}\label{eq:s:e19}
   \left| \int_{\bbR^{(m+1)n}}
   \frac{ f(y_1,\ldots,y_m) h(x)dxdy_1\ldots dy_m }
   {W(x,y)^{\lambda}}
 \right|  \lesssim \|f\|_{L^{\vec p}} \|h\|_{L^{q'}},
\end{equation}
where
\[
  W(x,y) = \sum_{l=1}^n |x^{(l)} - y_{j_l}^{(k_l)}|
  + \sum_{\substack{1\le j\le m-1, 1\le k\le n\\
  (j,k)\ne (j_l,k_l), 1\le l\le n}} |y_j^{(k)}|+|y_m|
\]
and $\{(j_l,k_l):\, 1\le l\le n\}$ is defined by (\ref{eq:jl kl}).

Using Young's inequality for $dx^{(l)} dy_{j_l}^{(k_l)}$ with $j_l<i_{\nu}$,
where the condition $q\ge p_{i_l}$ is used,
and using
H\"older's inequality for $dy_j^{(k)}$ with $(j,k)\not\in \{(j_l,k_l), 1\le l\le n\}$
and $j\le i_{\nu}-1$, respectively, we get
\begin{align*}
&   \left| \int_{\bbR^{(m+1)n}}
   \frac{ f(y_1,\ldots,y_m) h(x)dxdy_1\ldots dy_m }
   {W(x,y)^{\lambda}}
 \right| \\
&\lesssim
 \int_{\bbR^{(m-i_{\nu}+1)n+\gamma_{i_{\nu}}}}
   \frac{ \|f \|_{L_{(y_1,\ldots,y_{i_{\nu}-1})}^{
    (p_1,\ldots,p_{i_{\nu}-1})}} \cdot
    \| h (\tilde x, \bar x)\|_{L_{\tilde x}^{q'}}
    d \bar x
    dy_{i_{\nu}}\ldots dy_m }
   { W_1(x,y)^{\alpha}},
\end{align*}
where
\begin{align*}
  W_1(x,y)
  &= |\bar x-\bar y_{i_{\nu}
  }| + |\tilde y_{i_{\nu}}|
      + \sum_{i=i_{\nu}+1}^m |y_i|,
    \\
\tilde x &= (x^{(1)},\ldots, x^{(n-\gamma_{i_{\nu}})}),
   \\
\bar x &= (x^{(n-\gamma_{i_{\nu}}+1)},\ldots, x^{(n)}),
    \\
\tilde y_{i_{\nu}} &= (y_{i_{\nu}}^{(1)},\ldots, y_{i_{\nu}}^{(n-\gamma_{i_{\nu}})}),
   \\
\bar y_{i_{\nu}} &= (y_{i_{\nu}}^{(n-\gamma_{i_{\nu}}+1)},\ldots, y_{i_{\nu}}^{(n)}).
\end{align*}
and
\[
  \alpha =\lambda - \frac{n}{p'_1} - \ldots - \frac{n}{p'_{i_{\nu}-1}} - \frac{n-\gamma_{i_{\nu}}}{q}.
\]
Consequently, to prove (\ref{eq:s:e19}), it suffices to show that
\begin{equation}\label{eq:s:e18}
   \left\| \int_{\bbR^{  \gamma_{i_{\nu}}}}
   \frac{
    \| h\tilde x, \bar x \|_{L_{\tilde x}^{q'}}
    d\bar x}
   {W_1(x,y)^{\alpha}}
 \right\|_{L_{(y_{i_{\nu}},\ldots, y_m)}^{(p'_{i_{\nu}},\ldots,p'_m)}}
 \lesssim   \|h\|_{L^{q'}}.
\end{equation}

Choose two positive numbers $s_1$ and $s_2$
such that $1\le s_1 < q'<s_2\le p'_{i_{\nu}}$
and $t(s):=n/( {n}/{p'_m} - {\gamma_{i_{\nu}}}/{q'}
   + {\gamma_{i_{\nu}}}/s) > 1$ for $s=s_1,s_2$.

Consider the operator $S$ defined by
\[
  Sg (y_m) = \left\| \int_{\bbR^{\gamma_{i_{\nu}}}}\frac{  g(\bar x)d\bar x }
   {W_1(x,y)^{\alpha}}
   \right\|_{L_{(y_{i_{\nu}},\ldots,y_{m-1})}^{(p'_{i_{\nu}},\ldots,p'_{m-1})}},
   \quad g  \in L^s.
\]
Recall that $p_{i_{\nu}}\le q$. By Young's inequality, we have
\begin{align*}
\left\| \int_{\bbR^{\gamma_{i_{\nu}}}}\frac{  g(\bar x)d\bar x }{W_1(x,y)^{\alpha}}
   \right\|_{L_{\bar y_{i_{\nu}}}^{p'_{i_{\nu}}}}
 \lesssim
 \frac{\|g\|_{L^s}}{(|\tilde y_{i_{\nu}}|
      + \sum_{i=i_{\nu}+1}^m |y_i|)^{\alpha - \gamma_{i_{\nu}}/p'_{i_{\nu}}
      -\gamma_{i_{\nu}}/s'      }}
\end{align*}
Computing the norms with respect to $\tilde y_{i_{\nu}}$,
$y_{i_{\nu}+1}$,
$\ldots$,   $y_{m-1}$   directly,
we get
\[
     |Sg(y_m) |  \lesssim \frac{\|g\|_{L^s}}{|y_m|^{n/t(s)}},
\]
where we use the fact that
\[
  \alpha = \frac{n}{p'_{i_{\nu}}} + \ldots + \frac{n}{p'_m} + \frac{\gamma_{i_{\nu}}}{q}.
\]
Hence
$S$ is of weak type $(s_l, t(s_l))$, $l=1,2$.
Since $s_1<q'<s_2$ and $t(q') = p'_m\ge q'$,
by the   Marcinkiewicz interpolation theorem \cite[Corollary 1.4.21]{Grafakos2008},
  $S$ is of type $(q',p'_m)$.
By setting  $g(\bar x) = \| h(\tilde x, \bar x ) \|_{L_{\tilde x}^{q'}}$, we get  (\ref{eq:s:e18}).

Next we consider the case
$i_0<m$.
By Lemmas~\ref{Lm:end points D}
and \ref{Lm:Dx-y},
it suffices to show that
\begin{equation}\label{eq:w:e30}
   \int_{\bbR^{(i_0+1)n}} \frac{f(y_1,\ldots,y_{i_0})h(x)dxdy_1\ldots dy_{i_0}}{W(x,y)^{\lambda}}
  \lesssim \|f\|_{L^{\vec{\tilde p}'}}\|h\|_{L^{q'}},
  \quad h\in L^{q'},
\end{equation}
where $\vec{\tilde p} =(p_1,\ldots, p_{i_0})$,
\[
  W(x,y) =   \sum_{\substack{1\le l\le n\\
           j_l\le i_0 }} |x^{(l)} - y_{j_l}^{(k_l)}|
     +  \sum_{\substack{1\le l\le n\\
           j_l>i_0 }} |x^{(l)} |
   +  \sum_{\substack{1\le j\le i_0, 1\le k\le n\\
  (j,k)\ne (j_l,k_l), 1\le l\le n}} |y_j^{(k)}|.
\]
and $\{(j_l,k_l):\, 1\le l\le n\}$ is defined by (\ref{eq:jl kl}).

There are four subcases.

(B1)\,\, $\gamma_{i_0}=\gamma_{i_0+1}=0$.

In this case, $i_{\nu}<i_0$ and
$W(x,y)$ is of the following form,
\[
  W(x,y) = \sum_{l=1}^n  |x^{(l)} - y_{j_l}^{(k_l)}|
      + \sum_{\substack{1\le j\le i_0-1, 1\le k\le n\\
  (j,k)\ne (j_l,k_l), 1\le l\le n}} |y_j^{(k)}|
  + |y_{i_0}|.
\]
By setting $m=i_0$ in Case (A3) we get the conclusion as desired.

(B2)\,\, $\gamma_{i_0}=\gamma_{i_0+1}>0$.

In this case, $i_{\nu}\ge i_0+1$, $p_{i_{\nu}}=1$ and
$W(x,y)$ is of the following form,
\begin{align*}
  W(x,y) &= \sum_{\substack{1\le l\le n\\
           j_l< i_0 }} |x^{(l)} - y_{j_l}^{(k_l)}|
             +  \sum_{\substack{1\le l\le n\\
           j_l>i_0 }} |x^{(l)} |
      +\! \sum_{\substack{1\le j\le i_0-1, 1\le k\le n\\
  (j,k)\ne (j_l,k_l), 1\le l\le n}} |y_j^{(k)}|
  + |y_{i_0}|\\
 &= \sum_{l=1}^{n-\gamma_{i_0}} |x^{(l)} - y_{j_l}^{(k_l)}|
             +  \sum_{l=n-\gamma_{i_0}+1}^n |x^{(l)} |
      + \!\sum_{\substack{1\le j\le i_0-1, 1\le k\le n\\
  (j,k)\ne (j_l,k_l), 1\le l\le n}} |y_j^{(k)}|
  + |y_{i_0}| .
\end{align*}
Using Young's and H\"older's  inequalities alternately, we get
\begin{align}
 &  \int_{\bbR^{i_0 n}} \frac{f(y_1,\ldots,y_{i_0})h(x)dxdy_1\ldots dy_{i_0-1}}{W(x,y)^{\lambda}}\nonumber\\
&\lesssim
 \|f(\ldots, y_{i_0})\|_{L_{(y_1,\ldots,y_{i_0-1})}^{
   (p'_1,\ldots,p'_{i_0-1})}}
   \int_{\bbR^{\gamma_{i_0}}}
  \frac{\|h(\tilde x, \bar x)\|_{L_{\tilde x}^{q'}}
  d \bar x}
{( |\bar x| + |y_{i_0}|)^{n/p'_{i_0} + \gamma_{i_0}/q}},
\label{eq:s:e24}
\end{align}
where
\[
  \tilde x=(x^{(1)},\ldots, x^{(n-\gamma_{i_0})}),
  \qquad
  \bar x = (x^{(n-\gamma_{i_0}+1)}, \ldots, x^{(n)}).
\]

Take two positive numbers $s_1$ and $s_2$ such that
$1<s_1<q'<s_2<\infty$ and
$t(s):= n/({n}/{p'_{i_0}}
      - {\gamma_{i_0}}/{q'}
      +
      {\gamma_{i_0}}/{s})>1$ for $s=s_1, s_2$.

Define  the operator $S$   by
\[
  Sg (y_{i_0}) = \int_{\bbR^{\gamma_{i_0}}} \frac{g(\bar x)d \bar x}
{( |\bar x | + |y_{i_0}|)^{n/p'_{i_0} + \gamma_{i_0}/q}},
\quad g\in L^s(\bbR^{\gamma_{i_0}}).
\]
We see from H\"older's inequality that
\[
   |Sg(y_{i_0})| \le \frac{\|g\|_{L^s}}{|y_{i_0}|^{n/t(s)}},
   \qquad s_1\le s\le s_2.
\]
Hence $S$ is of weak type $(s, t(s))$.
Note that $t(q') = p'_{i_0} \ge q'$.
By the interpolation theorem,
$S$ is of type $(q', p'_{i_0})$. By setting
$g(\bar x) = \|h(\tilde x, \bar x)\|_{L_{\tilde x}^{q'}}$,
we get
\[
  \left\|
   \int_{\bbR^{\gamma_{i_0}}}
  \frac{\|h(\tilde x, \bar x)\|_{L_{\tilde x}^{q'}}
  d \bar x}
{( |\bar x| + |y_{i_0}|)^{n/p'_{i_0} + \gamma_{i_0}/q}}
 \right\|_{L_{y_{i_0}}^{p'_{i_0}}}
\lesssim \|h\|_{L^{q'}}.
\]
It follows from (\ref{eq:s:e24}) that (\ref{eq:w:e30}) is true.

(B3)\,\, $\gamma_{i_0} > \gamma_{i_0+1}=0$.

In this case, $i_{\nu} = i_0$ and
$W(x,y)$ is of the following form,
\[
  W(x,y) = \sum_{\substack{1\le l\le n\\
           j_l<i_0 }} |x^{(l)} - y_{j_l}^{(k_l)}|
     +  \sum_{\substack{1\le l\le n\\
           j_l=i_0 }}  |x^{(l)} - y_{j_l}^{(k_l)}|
   +  \sum_{\substack{1\le j\le i_0, 1\le k\le n\\
  (j,k)\ne (j_l,k_l), 1\le l\le n}} |y_j^{(k)}|.
\]
By setting $m=i_0$ in Case (A1) (for $\gamma_{i_0}=n$)
or in Case (A2) (for $0<\gamma_{i_0}<n$), we get the conclusion as desired.

(B4)\,\, $\gamma_{i_0} > \gamma_{i_0+1}>0$.

In this case, $i_{\nu} > i_0$, $p_{i_{\nu}}=1$ and
$W(x,y)$ is of the following form,
\begin{align*}
  W(x,y)
 & = \sum_{\substack{1\le l\le n\\
           j_l<i_0 }} |x^{(l)} - y_{j_l}^{(k_l)}|
     +   \sum_{\substack{1\le l\le n\\
           j_l=i_0 }}  |x^{(l)} - y_{j_l}^{(k_l)}|
     +   \sum_{\substack{1\le l\le n\\
           j_l>i_0 }} |x^{(l)} | \\
&\qquad
   +  \sum_{\substack{1\le j\le i_0, 1\le k\le n\\
  (j,k)\ne (j_l,k_l), 1\le l\le n}}\! |y_j^{(k)}|
  \\
   & = \sum_{\substack{1\le l\le n\\
           j_l<i_0 }} |x^{(l)} - y_{j_l}^{(k_l)}|
     +  \sum_{\substack{1\le l\le n\\
           j_l=i_0 }}  |x^{(l)} - y_{j_l}^{(k_l)}|
   +  \sum_{\substack{1\le j\le i_0-1, 1\le k\le n\\
  (j,k)\ne (j_l,k_l), 1\le l\le n}} |y_j^{(k)}| \\
 &\qquad
       +  \sum_{l=n-\gamma_{i_0+1}+1}^n |x^{(l)} |
   + \sum_{l=1}^{n-\gamma_{i_0}+\gamma_{i_0+1}} |y_{i_0}^{(l)} |.
\end{align*}
Denote
\begin{align*}
&\tilde x = (x^{(1)},\ldots, x^{(n-\gamma_{i_0+1})}),
&\bar x = ( x^{(n-\gamma_{i_0+1}+1)},
\ldots, x^{(n)}), \\
&\tilde y_{i_0} = (y_{i_0}^{(n-\gamma_{i_0}+\gamma_{i_0+1}+1)},
\ldots, y_{i_0}^{(n)}),
& \bar y_{i_0}  = (y_{i_0}^{(1)}, \ldots, y_{i_0}^{(n-\gamma_{i_0}+\gamma_{i_0+1})}).
\end{align*}
Using Young's and H\"older's  inequalities alternately, we get
\begin{align}
 &  \int_{\bbR^{(i_0+1) n}} \frac{f(y_1,\ldots,y_{i_0})h(x)dxdy_1\ldots dy_{i_0-1} d\tilde y_{i_0}}{W(x,y)^{\lambda}}\nonumber \\
&\lesssim
 \|f\|_{L_{(y_1,\ldots,y_{i_0-1}, \tilde y_{i_0})}^{
   (p'_1,\ldots,p'_{i_0-1},p'_{i_0})}}
   \int_{\bbR^{\gamma_{i_0+1}}}
  \frac{\|h(\tilde x, \bar x)\|_{L_{\tilde x}^{q'}}
  d\bar x}
{( |\bar x| + |\bar y_{i_0}|)^{(n-\gamma_{i_0}+\gamma_{i_0+1})/p'_{i_0} + \gamma_{i_0+1}/q}}.\label{eq:s:e25}
\end{align}

For $1<s<\infty$, let $t(s)$ be such that
\[
 \frac{n-\gamma_{i_0}+\gamma_{i_0+1}}{t(s)}  =\frac{n-\gamma_{i_0}+\gamma_{i_0+1}}{p'_{i_0}} -
  \frac{\gamma_{i_0+1}}{q'}  + \frac{\gamma_{i_0+1}}{s}.
\]
Take tow positive numbers $s_1$ and $s_2$ such that
$1<s_1<q'<s_2<\infty$ and $t(s_l)>1$, $l=1,2$.

For  any $s\in [s_1, s_2]$, define  the operator $S$   by
\[
  Sg (\bar y_{i_0}) = \int_{\bbR^{\gamma_{i_0+1}}} \frac{g( \bar x)d \bar x}
{(|\bar x| + |\bar y_{i_0}|)^{(n-\gamma_{i_0}+\gamma_{i_0+1})/p'_{i_0} + \gamma_{i_0+1}/q}},
\quad   g\in L^{s}(\bbR^{\gamma_{i_0+1}}).
\]
We see from H\"older's inequality that
\[
 | Sg(\bar y_{i_0})| \lesssim
  \frac{\|g\|_{L^s}}{|\bar y_{i_0}|^{ (n-\gamma_{i_0}+\gamma_{i_0+1})/t(s)}}.
\]
Hence $S$ is of weak type $(s, t(s))$ for $s=s_1$, $s_2$.
Since $t(q')= p'_{i_0}\ge q'$,
it follows from the interpolation theorem
that $S$ is of type  $(q', p'_{i_0})$.
By setting $g(\bar x)=\|h(\tilde x, \bar x)\|_{L_{\tilde x}^{q'}}$, we see from (\ref{eq:s:e25})
that (\ref{eq:w:e30}) is true.
This completes the proof.
\endproof

\subsubsection{ The Case   \texorpdfstring{$p_{i_{\nu}}=q$}{p i nu = q}}

The proof of the case $p_{i_{\nu}}=q$ is similar to the case
$p_{i_{\nu}} < q$, but with more technical details. Specifically, we have to show that
\begin{equation}\label{eq:s:e20}
  \left|
\int_{\bbR^{(m+1)n}}
 \frac {f(y_1,\ldots,y_m)h(x)dx dy }{W(x,y)^{\lambda}}
\right|
\lesssim \|f\|_{L^{\vec p}}\|h\|_{L^{q'}}.
\end{equation}

Since $\min\{p_{i_l}:\, 1\le l\le \nu\}<q$,
there is some $1\le \mu\le \nu-1$ such that
$p_{i_{\mu}}<q$ and $p_{i_{\mu+1}} = \ldots = p_{i_{\nu}}=q$.
As in the case  $p_{i_{\nu}}<q$,
we first consider the case  $i_0=m$. There are two subcases.

(C1)\,\, $0<\rank(D_m)<n$.

In this case, $i_{\nu}=m$, $q=p_m$ and the function $W(x,y)$ in (\ref{eq:s:e20})
is of the following form,
\begin{align*}
W(x,y)
&=\sum_{\substack{1\le l\le n\\
           j_l<i_{\mu} }} |x^{(l)} - y_{j_l}^{(k_l)}|
     +  \sum_{\substack{1\le l\le n\\
           j_l=i_{\mu} }}  |x^{(l)} - y_{j_l}^{(k_l)}|
     +  \sum_{\substack{1\le l\le n\\
           j_l>i_{\mu} }}  |x^{(l)} - y_{j_l}^{(k_l)}|\\
 &\qquad
   +  \sum_{\substack{1\le j\le m, 1\le k\le n\\
  (j,k)\ne (j_l,k_l), 1\le l\le n}} |y_j^{(k)}|,
\end{align*}
where $\{(j_l, k_l):\, 1\le l\le n\}$ is defined by (\ref{eq:jl kl}).

Using Young's inequality
when computing the integral with respect to  $x^{(l)}$ and $ y_{j_l}^{(k_l)}$
for
$1\le l\le n$ with $j_l < i_{\mu}$,
and using H\"older's inequality
for computing the integral with respect to $y_j^{(k)}$
for $j< i_{\mu}$ with $(j,k)\not\in \{(j_l,k_l):\, 1\le l\le n\}$, we get
  that for any $f\in L^{\vec p}$,
\begin{align}
&\left|
\int_{\bbR^{i_{\mu}n - \gamma_{i_{\mu}}}}
 \frac {f(y_1,\ldots,y_m)h(x)dx^{(1)}\ldots dx^{(n-\gamma_{i_{\mu}})}
  dy_1\ldots dy_{i_{\mu}-1}
 }{W(x,y)^{\lambda}}
\right|  \nonumber \\
&\lesssim
   \frac{
   \|f(\ldots, y_{i_{\mu}}, \ldots,y_m)
   \|_{L_{(y_1,\ldots,y_{i_{\mu}-1})}^{
      (p'_1,\ldots,p'_{i_{\mu}-1})}}
   \|h(\tilde x, \bar x, \bar {\bar x})
       \|_{L_{\tilde x}^{q'}}}
   { W_1(x,y)^{\lambda-n/p'_1-\ldots-n/p'_{i_{\mu}-1}
   -(n-\gamma_{i_{\mu}})/q}},\label{eq:s:e21}
\end{align}
where
\begin{align*}
 W_1(x,y)
&=   \sum_{\substack{1\le l\le n\\
           j_l=i_{\mu} }}  |x^{(l)} - y_{j_l}^{(k_l)}|
     +  \sum_{\substack{1\le l\le n\\
           j_l>i_{\mu} }}  |x^{(l)} - y_{j_l}^{(k_l)}|
    +  \sum_{\substack{i_{\mu}\le j\le m, 1\le k\le n\\
  (j,k)\ne (j_l,k_l), 1\le l\le n}} |y_j^{(k)}|, \\
\tilde x &=(x^{(1)}, \ldots, x^{(n-\gamma_{i_{\mu}})}), \\
\bar x &=
(x^{(n-\gamma_{i_{\mu}}+1)},\ldots,
x^{(n-\gamma_{i_{\mu+1}})}),\\
\bar{\bar x}&=
(x^{(n-\gamma_{i_{\mu+1}}+1)},\ldots,x^{(n)}).
\end{align*}
Denote
\[
  \bar y =(y_{i_{\mu+1}}^{(n+1-\gamma_{i_{\mu+1}}+\gamma_{i_{\mu+2}})},
\ldots,
y_{i_{\mu+1}}^{(n)},
\ldots,
y_{i_{\nu}}^{(n+1-\gamma_{i_{\nu}})},
\ldots,
y_{i_{\nu}}^{(n)}
).
\]
We have $\dim \bar y = \dim\bar{\bar x}$.
We rewrite $W_1$ as
\[
 W_1(x,y)
 =  |\bar{\bar x} - \bar y|
    +  \sum_{\substack{1\le l\le n\\
           j_l=i_{\mu} }}  |x^{(l)} - y_{j_l}^{(k_l)}|
     +   \sum_{\substack{i_{\mu}\le j\le m, 1\le k\le n\\
  (j,k)\ne (j_l,k_l), 1\le l\le n}} |y_j^{(k)}|.
\]
Set
$h_1(\bar x, \bar {\bar x})
=\|h(\tilde x, \bar x, \bar{\bar x})\|_{L_{\tilde x}^{q'}}$.
We see from (\ref{eq:s:e21}) that
\begin{align}
&\left|
\int_{\bbR^{(m+1)n}}
 \frac {f(y_1,\ldots,y_m)h(x)dx dy_1\ldots dy_m }{W(x,y)^{\lambda}}
\right|     \nonumber \\
&
\lesssim
  \int_{\bbR^{(m+1-i_{\mu})n+\gamma_{i_{\mu}}}}
   \frac{
   \|f(\ldots, y_{i_{\mu}}, \ldots,y_m)
   \|_{L_{(y_1,\ldots,y_{i_{\mu}-1})}^{
      (p'_1,\ldots,p'_{i_{\mu}-1})}}
    h_1(\bar x, \bar {\bar x})
    d\bar x d\bar {\bar x}
    dy_{i_{\mu}} \ldots dy_m}
   { W_1(x,y)^{\lambda-n/p'_1-\ldots-n/p'_{i_{\mu}-1}
   -(n-\gamma_{i_{\mu}})/q}}.
 \label{eq:s:e22}
\end{align}
Recall that for a radial decreasing function $\phi(x)\in L^1$, we have $ |\int_{\bbR^n}  g (x) \phi(y-x)dx|
\lesssim \|\phi\|_{L^1} Mg(y)$, where $M$ is the Hardy-Littlewood maximal function. Hence
\begin{align}
& \left| \int_{\bbR^{\gamma_{i_{\mu+1}}}}
  \frac{ h_1(\bar x, \bar {\bar x})
     d\bar {\bar x} }
  { W_1(x,y)^{\lambda-n/p'_1-\ldots-n/p'_{i_{\mu}-1}
   -(n-\gamma_{i_{\mu}})/q}} \right|
\lesssim
   \frac{M_2 h_1(\bar x,  \bar y)}
  {  W_2(x,y)^{\alpha}},     \label{eq:s:e23}
\end{align}
where
\begin{align*}
 W_2(x,y)
&=   \sum_{\substack{1\le l\le n\\
           j_l=i_{\mu} }}  |x^{(l)} - y_{j_l}^{(k_l)}|
    +  \sum_{\substack{i_{\mu}\le j\le m, 1\le k\le n\\
  (j,k)\ne (j_l,k_l), 1\le l\le n}} |y_j^{(k)}|,\\
 \alpha &=
\lambda-\frac{n}{p'_1} -\ldots-\frac{n}{p'_{i_{\mu}-1}}
   -\frac{n-\gamma_{i_{\mu}}}{q}- \gamma_{i_{\mu+1}},
\end{align*}
and $M_2 h_1$ stands for the maximal function $M$ acting on the second variable
of $h_1$, i.e.,
$M_2 h_1(\bar x,  \bar y) =(M_2 h_1(\bar x, \cdot))(\bar y)$.

For any $g(\bar x, \bar y)\in L_{\bar y}^{q'}(L_{\bar x}^s)$
with $1\le s \le p'_{i_{\mu}}$,
define the operator $S$ by
\begin{equation}\label{eq:Sg}
  Sg(y_{i_{\mu}},\ldots,y_m)
  = \int_{\bbR^{\gamma_{i_{\mu}} - \gamma_{i_{\mu+1}}}}
   \frac{ g(\bar x,  \bar y) d\bar x}
  { W_2(x,y)^{\alpha}}.
\end{equation}
Denote $\tilde y_{i_{\mu}}
= (y_{i_{\mu}}^{(n-r_{i_{\mu}} + r_{i_{\mu+1}}+1)},
\ldots, y_{i_{\mu}}^{(n)})$. We rewrite $W_2$ as
\[
 W_2(x,y)
 =   |\bar x - \tilde y_{i_{\mu}}|
    +  \sum_{\substack{i_{\mu}\le j\le m, 1\le k\le n\\
  (j,k)\ne (j_l,k_l), 1\le l\le n}} |y_j^{(k)}|.
\]
By Young's inequality, we have
\begin{align*}
&  \|Sg(y_{i_{\mu}},\ldots,y_m)
   \|_{L_{\tilde y_{i_{\mu}}}^{p'_{i_{\mu}}}} \\
&  \lesssim
  \frac{\|g(\bar x, \bar y)\|_{L_{\bar x}^s}}
  {(\sum_{\substack{i_{\mu}\le j\le m, 1\le k\le n\\
  (j,k)\ne (j_l,k_l), 1\le l\le n}} |y_j^{(k)}|)^{
  \alpha - (\gamma_{i_{\mu}} - \gamma_{i_{\mu+1}})/p'_{i_{\mu}}
  -
  (\gamma_{i_{\mu}} - \gamma_{i_{\mu+1}})/s'}
  }.
\end{align*}

Recall that $i_{\nu}=m$ and $p_{i_l} =q$ for $\mu+1\le l\le \nu$.
Computing the norm  with respect to $y_j^{(k)}$ directly,
we get
\[
 \|Sg(y_{i_{\mu}},\ldots,y_m)\|_{L_{(y_{i_{\mu}},\ldots,y_{m-1},y_m^1)
  }^{(p'_{i_{\mu}},\ldots,p'_{m-1},p'_m)}}
 \lesssim
 \frac{\|g\|_{L^{q'}(L^s)}}{|y_m^2|^{\beta}},
\]
where
\begin{align*}
  y_m^1 &= (y_m^{(1)},\ldots, y_m^{(n-\gamma_{i_{\mu}}+\gamma_{i_{\mu+1}}-\gamma_{i_{\nu}})},
  y_m^{(n+1-\gamma_{i_{\nu}})},
\ldots,
y_m^{(n)}
    ),
  \\
  y_m^2 &= (y_m^{(n-\gamma_{i_{\mu}}+\gamma_{i_{\mu+1}}-\gamma_{i_{\nu}}+1)},
  \ldots, y_m^{(n- \gamma_{i_{\nu }} )})
\end{align*}
and
\begin{align*}
\beta &= \alpha - \frac{\gamma_{i_{\mu}} - \gamma_{i_{\mu+1}}}{p'_{i_{\mu}}}
  -\frac{\gamma_{i_{\mu}} - \gamma_{i_{\mu+1}}}{s'}
  -\sum_{i=i_{\mu}}^{m-1} \frac{n-(\gamma_i - \gamma_{i+1})}{p'_i}
    \\
&\qquad      -\frac{n-\gamma_{i_{\mu}}+\gamma_{i_{\mu+1}}-\gamma_{i_{\nu}}}{p'_m}\\
&=\lambda-\frac{n}{p'_1} -\ldots-\frac{n}{p'_{i_{\mu}-1}}
   -\frac{n-\gamma_{i_{\mu}}}{q}- \gamma_{i_{\mu+1}}
      - \frac{\gamma_{i_{\mu}} - \gamma_{i_{\mu+1}}}{p'_{i_{\mu}}}
  -\frac{\gamma_{i_{\mu}} - \gamma_{i_{\mu+1}}}{s'}\\
&\qquad
  -\sum_{i=i_{\mu}}^{m-1} \frac{n }{p'_i}
  +\sum_{i=i_{\mu}}^{m-1} \frac{ \gamma_i - \gamma_{i+1} }{p'_i}
     -\frac{n-\gamma_{i_{\mu}}+\gamma_{i_{\mu+1}}-\gamma_{i_{\nu}}}{p'_m}\\
&= \frac{ \gamma_{i_{\mu}}}{q}- \gamma_{i_{\mu+1}}
      - \frac{\gamma_{i_{\mu}} - \gamma_{i_{\mu+1}}}{p'_{i_{\mu}}}
  -\frac{\gamma_{i_{\mu}} - \gamma_{i_{\mu+1}}}{s'}
  +\sum_{l=\mu}^{\nu-1} \frac{ \gamma_{i_l} - \gamma_{i_{l+1}} }{p'_{i_l}}\\
  &\qquad
     +\frac{ \gamma_{i_{\mu}}-\gamma_{i_{\mu+1}}+\gamma_{i_{\nu}}}{p'_m}\\
&=  \frac{\gamma_{i_{\mu}} - \gamma_{i_{\mu+1}}}{s}.
\end{align*}
Hence $S$ is bounded from $L^{q'}(L^s)$
to $L_{y_m^2}^{s,\infty}(
  L_{(y_{i_{\mu}},\ldots,y_{m-1},y_m^1)
  }^{(p'_{i_{\mu}},\ldots,p'_{m-1},p'_m)})$
  whenever $1\le s\le p'_{i_{\mu}}$.
  Since $1<q'=p'_m<p'_{i_{\mu}}$, by  the interpolation theorem (Lemma~\ref{Lm:interpolation}),
$S$ is bounded from $L^{q'}(L^{q'})$
to $L_{y_m^2}^{q'}(
  L_{(y_{i_{\mu}},\ldots,y_{m-1},y_m^1)
  }^{(p'_{i_{\mu}},\ldots,p'_{m-1},p'_m)})$. That is,
\[
 \|Sg(y_{i_{\mu}},\ldots,y_m)\|_{L_{(y_{i_{\mu}},\ldots,y_m)
  }^{(p'_{i_{\mu}},\ldots,p'_m)}}
  \lesssim
  \|g\|_{L^{q'}}.
\]
Set $g= M_2h_1$. We have
\begin{align*}
  \left\|\int_{\bbR^{\gamma_{i_{\mu}} - \gamma_{i_{\mu+1}}}}
   \frac{ M_2h_1(\bar x,  \bar y) d\bar x}
  { W_2(x,y)^{\alpha}}\right\|_{L_{(y_{i_{\mu}},\ldots,y_m)
  }^{(p'_{i_{\mu}},\ldots,p'_m)}}
  \lesssim  \|M_2 h_1\|_{L^{q'}}
 \approx  \| h_1\|_{L^{q'}}.
\end{align*}
It follows from (\ref{eq:s:e23})
that
\begin{align}
& \left\| \int_{\bbR^{\gamma_{i_{\mu}}}}
  \frac{h_1(\bar x, \bar{\bar x}) d\bar x d\bar{\bar x}}
  { W_1(x,y)^{\lambda-n/p'_1-\ldots-n/p'_{i_{\mu}-1}
   -(n-\gamma_{i_{\mu}})/q}} \right\|_{L_{(y_{i_{\mu}},\ldots,y_m)
  }^{(p'_{i_{\mu}},\ldots,p'_m)}}
 \lesssim   \| h_1\|_{L^{q'}}.
   \nonumber
\end{align}
By (\ref{eq:s:e22}), we get
\[
  \left|
\int_{\bbR^{(m+1)n}}
 \frac {f(y_1,\ldots,y_m)h(x)dx dy_1\ldots dy_m }{W(x,y)^{\lambda}}
\right|
\lesssim \|f\|_{L^{\vec p}}\|h\|_{L^{q'}}.
\]

(C2)\,\, $\rank(D_m)=0$.

In this case, $i_{\nu}\le m-1$.
(\ref{eq:s:e23}) is still true.
Define the operator $S$ as in (\ref{eq:Sg}).
We have
\[
 \|Sg(y_{i_{\mu}},\ldots,y_m)\|_{L_{(y_{i_{\mu}},\ldots,y_{m-1})
  }^{(p'_{i_{\mu}},\ldots,p'_{m-1})}}
 \lesssim
 \frac{\|g\|_{L^{q'}(L^s)}}{|y_m|^{
    n/t(s)}},
\]
where $t(s)$ satisfies
\[
\frac{n}{t(s)} =  \frac{n}{p'_m}+ \frac{\gamma_{i_{\mu}}-\gamma_{i_{\mu+1}}}{s}
 -  \frac{\gamma_{i_{\mu}}-\gamma_{i_{\mu+1}}}{q'}.
\]
Hence $S$ is bounded from $L^{q'}(L^s)$
to $L_{y_m}^{t(s),\infty}(
  L_{(y_{i_{\mu}},\ldots,y_{m-1})
  }^{(p'_{i_{\mu}},\ldots,p'_{m-1})})$
  whenever $1\le s\le p'_{i_{\mu}}$.
Observe that $t(q') = p'_m$.
We can choose positive numbers $s_1$ and $s_2$ such that
 $1<s_1<q'<s_2<p'_{i_{\mu}}$ and $t(s_l)>1$ for $l=1,2$.
 Since $p'_m\ge q'$,
using  the interpolation theorem (Lemma~\ref{Lm:interpolation}) again,
we get that
\[
 \|Sg(y_{i_{\mu}},\ldots,y_m)\|_{L_{(y_{i_{\mu}},\ldots,y_m)
  }^{(p'_{i_{\mu}},\ldots,p'_m)}}
  \lesssim
  \|g\|_{L^{q'}}.
\]
Now with the same arguments as that in Case (C1) we get
the conclusion as desired.

Next we consider the case   $i_0<m$.
Since $ p_{i_{\nu}}=q \ge p_{i_0}>1$,
we have $i_{\nu}\le  i_0$.
Consequently, $\gamma_{i_0+1}=\ldots = \gamma_m=0$.
Moreover, there is some $1\le \mu\le \nu-1$ such that
$p_{i_{\mu}}<q$ and $p_{i_{\mu+1}} = \ldots = p_{i_{\nu}}=q$.

By Lemmas~\ref{Lm:end points D}
and \ref{Lm:Dx-y},
it suffices to show that
\begin{equation}\label{eq:s:e26}
   \int_{\bbR^{(i_0+1)n}} \frac{f(y_1,\ldots,y_{i_0})h(x)dxdy_1\ldots dy_{i_0}}{W(x,y)^{\lambda}}
  \lesssim \|f\|_{L^{\vec{\tilde p}'}}\|h\|_{L^{q'}},
  \quad h\in L^{q'},
\end{equation}
where $\vec{\tilde p} =(p_1,\ldots, p_{i_0})$,
\begin{align*}
W(x,y)
&=\sum_{\substack{1\le l\le n\\
           j_l<i_{\mu} }} |x^{(l)} - y_{j_l}^{(k_l)}|
     +  \sum_{\substack{1\le l\le n\\
           j_l=i_{\mu} }}  |x^{(l)} - y_{j_l}^{(k_l)}|
     +  \sum_{\substack{1\le l\le n\\
           j_l>i_{\mu} }}  |x^{(l)} - y_{j_l}^{(k_l)}|\\
 &\qquad
   +  \sum_{\substack{1\le j\le i_0, 1\le k\le n\\
  (j,k)\ne (j_l,k_l), 1\le l\le n}} |y_j^{(k)}|,
\end{align*}
and
$j_l $ and $k_l$ are defined by (\ref{eq:jl kl}).

By setting $m=i_0$ in Case (C1) (for $\gamma_{i_0}>0$ )
or in Case (C2) ( for $\gamma_{i_0}=0$ ), we get the conclusion
as desired.

\section{Extension of The Multilinear Fractional Integrals to Linear Operators}

In this section, we give a proof of Theorem~\ref{thm:main}.
We begin with some preliminary results.

The following result is a generalization of Lemma~\ref{Lm:end points D}.

\begin{Lemma}\label{Lm:end points ii}
Suppose that $1\le p_i\le \infty$, $1\le i\le m+1$, $1\le q\le \infty$
and $0<\lambda<mn$ which meet (\ref{eq:ss:e0}).
Suppose that  $p_{i_0+1}=\ldots=p_{m+1}=1$ for some $1\le i_0\le m$.
Let
$D_i$ be $n\times n$ matrices, $2\le i\le m+1$. Denote $\tilde p=(p_1,\ldots, p_{i_0})$. Then the following two items are equivalent:

\begin{enumerate}
  \item there is a constant $C_{\lambda,\vec p, q, n}$ such that for any $h\in \!L^{q'}$ and almost all $(x_{i_0+1}$, $\ldots$, $x_{m+1})$ $\in\bbR^{(m+1-i_0)n}$,
   \[
     \left\| \frac{h(x_1)}{( \sum_{i=2}^{m+1}|D_ix_1-x_i| )^{\lambda}} \right\|_{
     L_{(x_1,\ldots,x_{i_0})}^{(p'_1,\ldots,p'_{i_0}) }}
     \le  C_{\lambda,\vec p, q, n}\|h\|_{L^{q'}}.
   \]

  \item  for any $h\in L^{q'}$,
   \[
     \left\| \frac{h(x_1)}{( \sum_{i=2}^{i_0}|D_ix_1-x_i|
     +\sum_{i=i_0+1}^{m+1}|D_i x_1|)^{\lambda}} \right\|_{L^{\vec {\tilde p}'}}
      \le  C_{\lambda,\vec p, q, n}\|h\|_{L^{q'}}.
   \]
\end{enumerate}
\end{Lemma}

\begin{proof}
Note that $p'_i=\infty$ for $i_0+1\le i\le m+1$.
So (i) is equivalent to
   \[
     \left\| \frac{h(x_1)}{( \sum_{i=2}^{m+1}|D_ix_1-x_i| )^{\lambda}} \right\|_{
     L^{\vec p'}}
     \le  C_{\lambda,\vec p, q, n}\|h\|_{L^{q'}}.
   \]
Or equivalently, for any
$f\in L^{\vec p}$ and $h\in L^{q'}$,
\[
     \left | \int_{\bbR^{(m+1)n}}\frac{f(x_1,\ldots,x_{m+1})h(x_1)dx_1\ldots dx_{m+1}}{( \sum_{i=2}^{m+1}|D_ix_1-x_i| )^{\lambda}} \right |
     \le  C_{\lambda,\vec p, q, n}\|f\|_{L^{\vec p}} \|h\|_{L^{q'}}.
\]
By setting
\[
  f(x_1,\ldots,x_{m+1})
  =\tilde f(x_1,\ldots,x_{i_0}) \prod_{i=i_0+1}^{m+1}
    \frac{1}{\delta^n} \chi^{}_{\{|x_i|\le \delta\}}(x_i)
\]
and letting $\delta\rightarrow 0$, we see from
Fatou's lemma that
\[
     \left| \int_{\bbR^{i_0n}}\frac{\tilde f(x_1,\ldots,x_{i_0}) h(x_1)dx_1\ldots
     d x_{i_0}}{( \sum_{i=2}^{i_0}|D_ix_1-x_i|
     +\sum_{i=i_0+1}^{m+1}|D_i x_1| )^{\lambda}} \right|
     \le  C_{\lambda,\vec p, q, n}\|\tilde f\|_{L^{\vec {\tilde p}}} \|h\|_{L^{q'}},
\]
where $\tilde p=(p_1,\ldots,p_{i_0})$.
Hence for any $\tilde f\in L^{\vec {\tilde p}}$
and $h\in L^{q'}$, we have
\[
     \left | \int_{\bbR^{ i_0 n}}\frac{\tilde f(x_1,\ldots,x_{i_0})h(x_1)dx_1\ldots dx_{i_0}}{( \sum_{i=2}^{i_0}|D_ix_1-x_i|
     +\sum_{i=i_0+1}^{m+1}|D_i x_1| )^{\lambda}} \right|
     \le  C_{\lambda,\vec p, q, n}\|\tilde f\|_{L^{\vec {\tilde p}}} \|h\|_{L^{q'}}.
\]
Consequently, (ii) is true.

Next we assume that (ii) is true.
If $D_{i_0+1}=\ldots=D_{m+1}=0$, then (i) is obvious.
Next we assume that not all of $D_{i_0+1}$, $\ldots$,
$D_{m+1}$ are zero matrices.

For
any $\tilde f\in L^{\vec {\tilde p}}$ and $h\in L^{q'}$, we have
\begin{align}
        \left| \int_{\bbR^{ i_0 n}} \frac{\tilde f(x_1,\ldots,x_{i_0})
 h(x_1) dx_1\ldots dx_{i_0}}{( \sum_{i=2}^{i_0}|D_ix_1-x_i|
     +\sum_{i=i_0+1}^{m+1}|D_i x_1|)^{\lambda}} \right |
     &\le C_{\lambda,\vec p, q, n}  \|\tilde f\|_{\vec {\tilde p}}\|h\|_{L^{q'}}. \label{eq:w:e28a}
\end{align}
Let
\[
  \tilde D = \begin{pmatrix}
  D_{i_0+1} \\
  \vdots \\
  D_m
  \end{pmatrix}
  \mathrm{\,\, and\,\,}
  y = \begin{pmatrix}
  x_{i_0+1}
  \\
  \vdots\\
  x_{m+1}
  \end{pmatrix}.
\]
Suppose that $\rank(\tilde D)=r$. Then  there are $(m-i_0)n\times (m-i_0)n$ invertible matrix $U$ and
$n\times n$ invertible matrix $V$ such that
\[
  U \tilde D V   =  \binom{I_r \quad }{\quad 0} .
\]
Note that $\sum_{i=i_0+1}^{m+1}|D_i x_1| \approx |\tilde D x_1| \approx |U\tilde Dx_1|$.
By a change of variable of the form $x_1\rightarrow Vx_1$ and replacing $f$ and $h$ by $f(V^{-1}\cdot, \ldots)$ and  $h(V^{-1}\cdot)$, respectively,
(\ref{eq:w:e28a}) turns out to be
\begin{equation}\label{eq:w:e29a}
       \left | \int_{\bbR^{i_0n}}\frac{f(x_1,\ldots,x_{i_0})
        h(x_1)dx_1\ldots dx_{i_0}}{( \sum_{i=2}^{i_0}|D_iVx_1-x_i|
     +\sum_{l=1}^r|x_1^{(l)}|)^{\lambda}} \right |
     \le C_{\lambda,\vec p, q, n} \|f\|_{L^{\vec{\tilde  p }}}\|h\|_{L^{q'}}.
\end{equation}

For  any $f\in L^{\vec {\tilde p}}$
 and $(x_{i_0+1},\ldots, x_{m+1})\in\bbR^{ (m +1-i_0)n  }$,
denote
\[
  \Delta_{f,h}:=
  \int_{\bbR^{i_0n}}\frac{|f(x_1,\ldots,x_{i_0})h(x_1)|
        dx_1\ldots dx_{i_0}}{( \sum_{i=2}^{m+1}|D_ix_1-x_i|)^{\lambda}}.
\]
Note that
\begin{equation}\label{eq:w:e36}
  \sum_{i=i_0+1}^{m+1} |D_ix_1-x_i|
  \approx  |\tilde Dx_1 - y|
  \approx | U\tilde Dx_1 - Uy|.
\end{equation}
By a change a variable of the form $x_1\rightarrow V(x_1+z)$,
where $z=(z^{(1)},\ldots,z^{(r)},0$, $\ldots,0)^*\in\bbR^n$
with $z^{(1)}, \ldots, z^{(r)}$ being the first $r$ components
of $Uy$, we have
\begin{align}
\Delta_{f,h} &\approx
  \int_{\bbR^{i_0n}}\frac{|f(x_1,\ldots,x_{i_0})h(x_1)|
        dx_1\ldots dx_{i_0}}{( \sum_{i=2}^{i_0}|D_ix_1-x_i|
        +| U\tilde Dx_1 - Uy|
        )^{\lambda}}
     \nonumber\\
&\approx
  \int_{\bbR^{i_0n}}\frac{|f(V(x_1+z),x_2,\ldots,x_{i_0})h(V(x_1+z))|
        dx_1\ldots dx_{i_0}}{( \sum_{i=2}^{i_0}|D_iVx_1-x_i+D_iVz|
        +|W(y)|
        +\sum_{l=1}^r|x_1^{(l)}|
        )^{\lambda}} \nonumber  \\
&\le
  \int_{\bbR^{i_0n}}\frac{|f(V(x_1+z),x_2,\ldots,x_{i_0})h(V(x_1+z))|
        dx_1\ldots dx_{i_0}}{( \sum_{i=2}^{i_0}|D_iVx_1-x_i+D_iVz|
        +\sum_{l=1}^r|x_1^{(l)}|
        )^{\lambda}}, \label{eq:w:e37}
\end{align}
where   $W(y)$ is a vector in $\bbR^{(m+1-i_0)n-r}$
consisting of the last $(m+1-i_0)n-r$ components of $Uy$.

By a change of variable of the form
$(x_2,\ldots,x_{i_0})
\rightarrow (x_2+ D_2Vz,\ldots, x_{i_0}+D_{i_0}Vz)$, we get
\begin{align}
&\Delta_{f,h} \nonumber \\
&\lesssim\!
  \int_{\bbR^{i_0n}}\hskip -8pt\frac{|f(V(x_1 +z),x_2+D_2Vz, \ldots,x_{i_0}+D_{i_0}Vz)
  h(V(x_1+z))|
        dx_1\ldots dx_{i_0}}{( \sum_{i=2}^{i_0}|D_iVx_1-x_i|
        +\sum_{l=1}^r|x^{(l)}|
        )^{\lambda}}.\label{eq:w:e38}
\end{align}
Note that the constants invoked in
(\ref{eq:w:e36}) -- (\ref{eq:w:e38})
are independent of
$(x_{i_0+1}$, $\ldots$, $x_{m+1})$.
By (\ref{eq:w:e29a}), we have
\begin{align*}
    \Delta_{f,h} &\le
      C_{\lambda,\vec p, q, n} \|f(V(\cdot+z),\cdot+D_2Vz,\ldots, \cdot+D_{i_0}Vz)\|_{L^{\vec{\tilde  p }}}\|h\|_{L^{q'}} \\
&  =
    C_{\lambda,\vec p, q, n} \|f\|_{L^{\vec{\tilde  p  }}}
 \|h\|_{L^{q'}}.
\end{align*}
Hence for any $(x_{i_0+1},\ldots, x_{m+1})\in\bbR^{ (m+1 -i_0)n  }$,
\[
  \left\| \frac{ h(x_1)
         }{( \sum_{i=2}^{m+1}|D_ix_1-x_i|)^{\lambda}}
         \right\|_{L^{\vec{\tilde  p }'}}
     \le C_{\lambda,\vec p, q, n}  \|h\|_{L^{q'}}.
\]
This completes the proof.
\end{proof}

To study   conditions on the matrix $A$ for which there exist $\vec p$, $q$ and $\lambda$ such that
$T_{\lambda}$ is bounded from $L^{\vec p}$ to $L^q$, we need the following two lemmas, which gives the structure of
$A^{-1}$.

\begin{Lemma}\label{Lm:La}
Suppose that  $A$ in an invertible $(n_1+n_2)\times (n_1+n_2)$ matrix  such that
\[
A = \begin{pmatrix}
A_1 & A_2 \\
A_3 & A_4
\end{pmatrix},
\qquad
\tilde A :=A^{-1} = \begin{pmatrix}
\tilde A_1 & \tilde A_2 \\
\tilde A_3 & \tilde A_4
\end{pmatrix},
\]
where $A_1$ is an $n_1\times n_2$ matrix,
$A_2$ is an $n_1\times n_1$ matrix,
$A_3$ is an $n_2\times n_2$ matrix,
$A_4$ is an $n_2\times n_1$ matrix,
$\tilde A_1$ is an $n_2\times n_1$ matrix,
$\tilde A_2$ is an $n_2\times n_2$ matrix,
$\tilde A_3$ is an $n_1\times n_1$ matrix,
and
$\tilde A_4$ is an $n_1\times n_2$ matrix.
We have
\begin{enumerate}
\item $A_3$ is invertible if and only if $\tilde A_3$ is invertible.

\item if $\tilde A_3$ is invertible, then $\tilde A_2 - \tilde A_1 \tilde A_3^{-1}\tilde A_4$ is invertible and
\[(\tilde A_2 - \tilde A_1 \tilde A_3^{-1} \tilde A_4 )^{-1}\tilde A_1 \tilde A_3^{-1}=-A_4.
 \]

\end{enumerate}

As a consequence, for $n_1=n_2$, $A_3$ and $A_4$ are invertible if and only if $\tilde A_1$ and $\tilde A_3$
are invertible.
\end{Lemma}

\begin{proof}
 First, we assume that $A_3$ is invertible.
Note that
\[
\begin{pmatrix}
0  & I_{n_2} \\
I_{n_1} & 0
\end{pmatrix}
\begin{pmatrix}
I_{n_1} & -A_1A_3^{-1} \\
0 & I_{n_2}
\end{pmatrix}
A
=\begin{pmatrix}
A_3 & A_4 \\
0  & A_2 -A_1 A_3^{-1}A_4
\end{pmatrix}.
\]
Since $A$ is invertible,
so is $A_2 -A_1 A_3^{-1}A_4$.
Moreover, we have
\[
  A^{-1} = \begin{pmatrix}
  *  & * \\
  (A_2 - A_1 A_3^{-1} A_4)^{-1} & - (A_2 - A_1 A_3^{-1} A_4)^{-1}A_1 A_3^{-1}
  \end{pmatrix}.
\]
Hence $\tilde A_3 = (A_2 - A_1 A_3^{-1} A_4)^{-1} $ is invertible.

On the other hand, if $\tilde A_3$ is invertible, similar arguments show that
$\tilde A_2 - \tilde A_1 \tilde A_3^{-1} \tilde A_4$ is invertible, $A_3 = (\tilde A_2 - \tilde A_1 \tilde A_3^{-1} \tilde A_4)^{-1} $
and
\[
  A_4 = - (\tilde A_2 - \tilde A_1 \tilde A_3^{-1} \tilde A_4)^{-1}\tilde A_1 \tilde A_3^{-1}.
\]
This proves (i) and (ii).

For the case  $n_1=n_2$, if $A_3$ is  invertible, then we see from the above arguments that
$\tilde A_3$ is invertible and $A_4 = - (\tilde A_2 - \tilde A_1 \tilde A_3^{-1} \tilde A_4)^{-1}\tilde A_1 \tilde A_3^{-1}$.
Hence $A_4$ is invertible if and only if $\tilde A_1$ is  invertible. This completes the proof.
\end{proof}

\begin{Lemma}
\label{Lm:A inverse}
Let $A $ be an
$(m+1)n\times (m+1)n$ invertible matrix,
$A_{m+1,m+1}$ be the submatrix consisting of
the last $n$ rows and the last $n$ columns of $A$,
and $B$ be the submatrix consisting of the first
$mn$ rows and the first $mn$ columns of $A^{-1}$.
Then we have
$\rank(B)=mn$ if and only if
$\rank(A_{m+1,m+1})=n$.
\end{Lemma}

\begin{proof}
Set $ P = A \begin{pmatrix}
 & I_{mn} \\
 I_n &
\end{pmatrix}$. Then we have
\[
  P = \begin{pmatrix}
  * & * \\
  A_{m+1,m+1} & *
  \end{pmatrix}
  \quad
  \mathrm{and}
  \quad
  P^{-1} =
   \begin{pmatrix}
 & I_n \\
 I_{mn} &
\end{pmatrix}
  A^{-1}
 =  \begin{pmatrix}
  * & * \\
  B & *
  \end{pmatrix}.
\]
Now the conclusion follows from Lemma~\ref{Lm:La}.
\end{proof}

\begin{Lemma}\label{Lm:q}
Suppose that $\vec p = (p_1,\ldots,p_{m+1})$ with $1\le p_i\le \infty$, $1\le i\le m+1$ and  $0<q\le \infty$.
Let $T_{\lambda}$ be defined by (\ref{eq:Tf}).
If $T_{\lambda}$ is bounded from $L^{\vec p}$ to $L^q$, we have
\begin{enumerate}
\item The matrix $A$ is invertible.
\item $q\ge 1$.
\end{enumerate}
\end{Lemma}

\begin{proof}
(i)\,\, Assume on the contrary that $\rank(A)<(m+1)n$.
Then the Lebesgue measure of
 $A \bbR^{(m+1)n}$ is zero. Set $f(x)=1
 $ for $x\in A \bbR^{(m+1)n}$ and $0$ for others. Then
we have $T_{\lambda}f(x_{m+1})=\infty$ while $\|f\|_{L^{\vec p}}=0$,
which contradicts with the boundedness of $T_{\lambda}$.

(ii)
Since $A$ is invertible, by Laplace's theorem,
there exist
$1\le i_1<\ldots < i_{mn}\le (m+1)n$
such that
the $mn\times mn$ submatrix consisting of
the first $mn$ columns and the $i_l$-th rows, $1\le l\le mn$,
is invertible.
Denote $\{1,\ldots, (m+1)n\}\setminus \{i_l:\, 1\le l\le mn\}$
by $\{i_{mn+1}, \ldots, i_{(m+1)n}\}$, where $i_{mn+1}<\ldots<i_{(m+1)n}$.

Let $A_1 = A_{(i_{mn+1},\ldots,i_{(m+1)n})}^{(1,\ldots,mn)}$
be the submatrix consisting of
the $i_{mn+1}$-th, $\ldots$, the $i_{(m+1)n}$-th rows
and the first $mn$ columns of $A$.
Let
$A_2 = A_{(i_{mn+1},\ldots,i_{(m+1)n})}^{(mn+1,\ldots,(m+1)n)}$,
$A_3 = A_{(i_1,\ldots,i_{mn})}^{(1,\ldots,mn)}$
and
$A_4 = A_{(i_1,\ldots,i_{mn})}^{(mn+1,\ldots,(m+1)n)}$
be defined similarly.

For $z=(z_1,\ldots, z_{(m+1)n})\in\bbR^{(m+1)n}$, Denote
\begin{equation}\label{eq:ea:a2}
f(z_1,\ldots, z_{(m+1)n}) = g(z_{i_{mn+1}}, \ldots, z_{i_{(m+1)n}},
z_{i_1}, \ldots, z_{i_{mn}}).
\end{equation}
Set $\tilde x = (x_1,\ldots,x_m)$.
Denote the $i$-th entry of $Ax$ by $(Ax)_i$.
We have
\begin{align*}
 f(Ax) &= f((Ax)_1, \ldots, (Ax)_{(m+1)n}) \\
    &= g((Ax)_{i_{mn+1}},\ldots,(Ax)_{i_{(m+1)n}},
      (Ax)_{i_1},\ldots, (Ax)_{i_{mn}}) \\
 &=
 g(A_1 \tilde x + A_2 x_{m+1}, A_3 \tilde x + A_4x_{m+1} ).
\end{align*}
Hence
\[
  T_{\lambda}f(x_{m+1}) = \int_{\bbR^{mn}}
     \frac{g(A_1 \tilde x + A_2 x_{m+1}, A_3 \tilde x + A_4x_{m+1} )
      d\tilde x}{|\tilde x|^{\lambda}}.
\]
Therefore,
\[
\left\|
\int_{\bbR^{mn}}
     \frac{g(A_1 \tilde x + A_2 x_{m+1}, A_3 \tilde x + A_4x_{m+1} )
      d\tilde x}{|\tilde x|^{\lambda}}
      \right\|_{L_{x_{m+1}}^q}
      \lesssim \|f\|_{L^{\vec p}}.
\]

By a change of variable of the form
$\tilde x \rightarrow A_3^{-1}(\tilde x -  A_4 x_{m+1})$, we get
\begin{align}
   \left\|\int_{\bbR^{mn}}
       \frac{g(A_1 A_3^{-1}\tilde x  + (A_2-A_1 A_3^{-1} A_4)x_{m+1}, \tilde x)   }{|\tilde x -  A_4 x_{m+1}|^{\lambda}}
     d\tilde x\right\|_{L_{x_{m+1}}^q}
  &  \lesssim \|f\|_{L^{\vec p}}  .
    \label{eq:ea:a8}
\end{align}

There are two cases.

(a)\,\,  The submatrix consisting of the last $n$ rows
and the first $mn$ columns of $A$ is a  zero matrix.

In this case, $i_l=l$ for all $1\le l\le (m+1)n$, $A_1=0$,
$A_2$ is invertible
 and $g(x_{m+1},\tilde x) = f(\tilde x, x_{m+1})$.
By (\ref{eq:ea:a8}),
\begin{align}
   \left\|\int_{\bbR^{mn}}
       \frac{f(\tilde x, A_2 x_{m+1})   }{|\tilde x -  A_4 x_{m+1}|^{\lambda}}
     d\tilde x\right\|_{L_{x_{m+1}}^q}
  &  \lesssim \|f\|_{L^{\vec p}}  ,
  \qquad \forall f\in L^{\vec p}.
\nonumber
\end{align}
By a change of variable of the form
$\tilde x \rightarrow \tilde x +A_4 x_{m+1}$, we get
\begin{align}
   \left\|\int_{\bbR^{mn}}
       \frac{f(\tilde x+    A_4 x_{m+1}, A_2 x_{m+1})   }{|\tilde x|^{\lambda}}
     d\tilde x\right\|_{L_{x_{m+1}}^q}
  &  \lesssim \|f\|_{L^{\vec p}}.
\nonumber
\end{align}
Since $A_2$ is invertible, we have
\begin{align}
   \left\|\int_{\bbR^{mn}}
       \frac{f(\tilde x+    A_4 A_2^{-1}y,  y)   }{|\tilde x|^{\lambda}}
     d\tilde x\right\|_{L_y^q}
  &  \lesssim \|f\|_{L^{\vec p}}.
\nonumber
\end{align}
Replacing $f(\tilde x, y)$ by $f(\tilde x - A_4 A_2^{-1}y, y)$, we get
\begin{align*}
   \left\|\int_{\bbR^{mn}}
       \frac{f( \tilde x ,y   )d\tilde x}{|\tilde x |^{\lambda}}
  \right\|_{L_y^q}
  &  \lesssim \|f\|_{L^{\vec p}}.
\end{align*}
If $q<1$, by setting $f(\tilde x, y)= f_1(\tilde x) f_2(y)$
with $f_2\in L^{p_{m+1}}\setminus L^q$, we get a contradiction.

(b)\,\, The submatrix consisting of the last $n$ rows
and the first $mn$ columns of $A$ is not zero.

In this case, we can choose the index $i_{mn}$ such that  $i_{mn}>mn$.
If $0<q<1$,   we see from  (\ref{eq:ea:a8}) and
Minkowski's inequality that for positive functions $f$,
\begin{align}
   \int_{\bbR^{mn}}\left\|
       \frac{g(A_1 A_3^{-1}\tilde x  + (A_2-A_1 A_3^{-1} A_4)x_{m+1}, \tilde x)   }{|\tilde x -  A_4 x_{m+1}|^{\lambda}}\right\|_{L_{x_{m+1}}^q}
     d\tilde x
  &  \lesssim \|f\|_{L^{\vec p}}  .
   \nonumber
\end{align}
By Lemma~\ref{Lm:La},
$A_2-A_1 A_3^{-1} A_4$ is invertible.
By a change of variable of the form
\[
  x_{m+1}\rightarrow (A_2-A_1 A_3^{-1} A_4)^{-1}(y
     -A_1 A_3^{-1}\tilde x),
\]
we get
\begin{align}
   \int_{\bbR^{mn}}\left\|
       \frac{g(y, \tilde x   )}{|P \tilde x -  Dy |^{\lambda}}
       \right\|_{L_y^q}
     d\tilde x
  &  \lesssim \|f\|_{L^{\vec p}} ,\nonumber
\end{align}
where
$P = I -A_4(A_2-A_1 A_3^{-1} A_4)^{-1} A_1 A_3^{-1}$
and
$D = A_4(A_2-A_1 A_3^{-1} A_4)^{-1}$.

Since $|P \tilde x -  Dy | \le |\tilde x| +|y|$, we have
\begin{align*}
  \int_{\bbR^{mn}}
       \left\| \frac{g(y, \tilde x   )}{(|\tilde x| +|y|)^{\lambda}}\right\|_{L_y^q}
     d\tilde x
  &  \lesssim \|f\|_{L^{\vec p}}.
\end{align*}
Recall that $g$ is defined by (\ref{eq:ea:a2}).
We have
\begin{equation}\label{eq:ea:a9}
  \int_{\bbR^{mn}}
       \left\| \frac{f(z)}{|z|^{\lambda}}\right\|_{L_{(z_{i_{mn+1}},
       \ldots, z_{i_{(m+1)n}})}^q}
     d z_{i_1}\ldots dz_{i_{mn}}
    \lesssim \|f\|_{L^{\vec p}}.
\end{equation}
Let 
\[
  \vec s=  (p_1,\ldots,p_1,\ldots,p_{m+1},\ldots,p_{m+1}),
\]
where each $p_i$ has $n$ copies, and 
\begin{align*}
 \vec u = (s_{i_1},\ldots,s_{i_{mn}}),
\qquad 
 \vec v = (s_{i_{mn+1}},\ldots,s_{i_{(m+1)n}}).
\end{align*} 
Set
\begin{equation}\label{eq:f}
  f(z_1,\ldots,z_{(m+1)n})
    =
    \frac{|z_{i_{mn}}|^{\alpha-\beta} h(z_{i_1},\ldots,z_{i_{mn}})}
    {(\sum_{l=1}^n |z_{i_{mn+l}}|+| z_{i_{mn}}|)^{\alpha}},
\end{equation}
where $\alpha>n/q$, $h\in L^{\vec u}$   and $\beta = 1/v_1+\ldots +1/v_n$.

We conclude that
 $\|f\|_{L^{\vec p}} \approx \|h\|_{L^{\vec u}}$.

Since $i_1<\ldots<i_{mn}$ and $i_{mn+1}<\ldots<i_{(m+1)n}$,
we have
$i_1=1$ or $i_{mn+1}=1$.
If $i_1=1$, we have
\[
  \|f(z_1,\ldots,z_{(m+1)n})\|_{L_{z_1}^{s_1}}
   \approx     \frac{|z_{i_{mn}}|^{\alpha-\beta} \|h(z_{i_1},\ldots,z_{i_{mn}})\|_{L_{z_{i_1}}^{u_1}}}
    {(\sum_{l=1}^n |z_{i_{mn+l}}|+| z_{i_{mn}}|)^{\alpha}}.
\]
And if $i_{mn+1}=1$, we have
\[
  \|f(z_1,\ldots,z_{(m+1)n})\|_{L_{z_1}^{s_1}}
  \approx \frac{|z_{i_{mn}}|^{\alpha-\beta} |h(z_{i_1},\ldots,z_{i_{mn}})|}
    {(\sum_{l=2}^n |z_{i_{mn+l}}|+| z_{i_{mn}}|)^{\alpha-1/v_1}}.
\]
And we compute the norm with respect to
$z_2$, $z_3$, $\ldots$ successively.
Denote $\kappa=i_{mn}$.
If $i_{(m+1)n}<\kappa$, then $i_{mn}=(m+1)n$ and 
we get
\begin{align*}
  \|f(z_1,\ldots,z_{(m+1)n})\|_{L_{(s_1,\ldots,s_{\kappa-1})
  }^{(s_1,\ldots,s_{\kappa-1})}}
&\approx \frac{|z_{i_{mn}}|^{\alpha-\beta} \|h(z_{i_1},\ldots,z_{i_{mn}})\|_{
    L_{(z_{i_1},\ldots,z_{i_{mn-1}})}^{
    (u_{i_1},\ldots,u_{i_{mn-1}})}}}
    {| z_{i_{mn}}|^{\alpha-1/v_1-\ldots-1/v_n}} \\
& =
 \|h(z_{i_1},\ldots,z_{i_{mn}})\|_{
    L_{(z_{i_1},\ldots,z_{i_{mn-1}})}^{
    (u_{i_1},\ldots,u_{i_{mn-1}})}}
    .
\end{align*}
It follows that $\|f\|_{L^{\vec p}} \approx \|h\|_{L^{\vec u}}$. 

If $i_{(m+1)n}>\kappa$, set $l_0=\min\{l:\, 1\le l\le n, i_{mn+l}>\kappa\}$. Since $i_{mn}>mn$, we have
$(z_{\kappa},\ldots,z_{(m+1)n})
 = (z_{\kappa}, z_{i_{mn+l_0}},\ldots,z_{i_{mn+n}})$
 and 
$u_{mn}=v_l = p_{m+1}$ when $l_0\le l\le n$. It follows that
\[
  \|f\|_{L^{\vec p}}
   =   \|f(z_1,\ldots,z_{(m+1)n})\|_{L_{(z_1,\ldots,z_{\kappa-1},
   z_{\kappa+1},\ldots,z_{(m+1)n},z_{\kappa})
  }^{(s_1,\ldots,s_{\kappa-1},
  s_{\kappa+1},\ldots,s_{(m+1)n},s_{\kappa})}}.
\]
Therefore, we also have $\|f\|_{L^{\vec p}} \approx \|h\|_{L^{\vec u}}$.
 
Note that for any $h\in L^{\vec u}$, the function $f$
defined in (\ref{eq:f}) satisfies
\begin{align*}
 & \int_{\bbR^{mn}}
       \left\| \frac{f(z)}{|z|^{\lambda}}\right\|_{L_{(z_{i_{mn+1}},
       \ldots, z_{(m+1)n})}^q}
     d z_{i_1}\ldots dz_{i_{mn}}\\
&\ge
  \int_{\bbR^{mn}}
       \left\|
        \frac{|z_{i_{mn}}|^{\alpha-\beta} h(z_{i_1},\ldots,z_{i_{mn}})}
    {|z|^{\lambda+\alpha}}
       \right\|_{L_{(z_{i_{mn+1}},
       \ldots, z_{(m+1)n})}^q}
     d z_{i_1}\ldots dz_{i_{mn}}\\
\\
&\approx
 \int_{\bbR^{mn}}
\frac{|z_{i_{mn}}|^{\alpha-\beta}| h(z_{i_1},\ldots,z_{i_{mn}})|}
{ (\sum_{l=1}^{mn} |z_{i_l}|)^{\lambda+\alpha-n/q}}
dz_{i_1}\ldots dz_{i_{mn}}.
\end{align*}
It follows from (\ref{eq:ea:a9}) that for any $h\in L^{\vec u}$,
\[
 \int_{\bbR^{mn}}
\frac{|z_{i_{mn}}|^{\alpha-\beta}
|h(z_{i_1},\ldots,z_{i_{mn}})|}
{ (\sum_{l=1}^{mn} |z_{i_l}|)^{\lambda+\alpha-n/q}}
dz_{i_1}\ldots dz_{i_{mn}}
     \lesssim \|h\|_{L^{\vec u}},
\]
which is impossible since $ |z_{i_{mn}}|^{\alpha-\beta} /
(\sum_{l=1}^{mn} |z_{i_l}|)^{\lambda+\alpha-n/q}
\not\in L^{\vec u\ \!\!'}$.
Hence $q\ge 1$.
\end{proof}

The following result is a variant of
Lemma~\ref{Lm:Dx-y}, which can be proved with almost the same arguments
and we omit the details.

\begin{Lemma}\label{Lm:ADx-y}
Suppose that  $\vec p = (p_1, \ldots, p_{m+1})$ with
$1\le p_i\le \infty$, $1\le i\le m$, $1\le q\le \infty$ and $0<\lambda<mn$.
Let $r_k$ and $k_l$  be defined as in
Theorem~\ref{thm:main} such that $r_2=n$.
Then
\[
\left\|\frac{ h( x_1 ) }
    {( \sum_{i=2}^{m+1} |A_{i,m+1}x_1 - x_i|)^{\lambda}}\right\|_{L^{\vec p'}}
 \lesssim \|h\|_{L^{q'}}
    \quad \forall h\in L^{q'}
\]
if and only if
\[
\left\|\frac{ h( x_1 ) }
    {W(x_1,\ldots,x_{m+1})^{\lambda}}\right\|_{L^{\vec p'}}
 \lesssim \|h\|_{L^{q'}}
    \quad \forall h\in L^{q'} ,
\]
where
\[
  W(x_1,\ldots,x_{m+1}) = \sum_{l=1}^n |x_1^{(l)} - x_{i_l}^{(j_l)}|
  + \sum_{\substack{2\le i\le m+1, 1\le j\le n\\
  (i,j)\ne (i_l,j_l), 1\le l\le n}} |x_i^{(j)}|
\]
and
\begin{equation}\label{eq:il jl}
 \{(i_l,j_l):\, 1\le l\le n\}
  =\bigcup_{1\le s\le \nu}\{(k_s, t):\, n+1-(r_{k_s}-r_{k_{s+1}})\le t\le n\}
\end{equation}
satisfying $i_l\le i_{l+1}$ and if
$i_l=i_{l+1}$, then $j_l<j_{l+1}$.
\end{Lemma}

 \subsection{Proof of Theorem~\ref{thm:main}: The Necessity}

In this subsection, we give a proof for the necessity part of Theorem~\ref{thm:main}.
We split the proof into several parts.

(P1) We show that $ p_{k_i}<\infty$, $1\le i\le \nu$.

Denote $A=(B_0 \ B_1)$, where $B_0$ and $B_1$
are $(m+1)n\times mn$ and $(m+1)n\times n$ matrices, respectively.
Since $A$ is invertible,  $\rank(B_1)=n$.
Hence there is an $n\times n$ invertible submatrix of $B_1$,
for which $r_k-r_{k+1}$ rows come from $A_{k,m+1}$,
$1\le k\le m+1$.
Consequently, there is some invertible  $(m+1)n\times (m+1)n$ matrix $U$
of the form
\[
   U = \begin{pmatrix}
     * & ?  & ? & \ldots   &  ?  \\
     0 & *  & ? & \ldots   &  ?   \\
     0 & 0  & *  & \ldots   &  ? \\
    & \ldots & & \ldots\\
    0& 0 & 0&   \ldots  &*
   \end{pmatrix},
\]
where $*$ stands for an $n\times n$ invertible matrix,
 such that
$UB_1$ has exactly $n$ non-zero rows.
More precisely, there are exactly
$r_k-r_{k+1}$ non-zero rows among the $(kn-n+1)$-th, $\ldots$, the
$kn$-th rows of $UB_1$.
Since $\|f\|_{L^{\vec p}} \approx \|f(U\cdot)\|_{L^{\vec p}}$, we have $\|T_{\lambda}(f(U\cdot))\|_{L^q}
\lesssim \|f\|_{L^{\vec p}}$. That is,
\begin{equation}\label{eq:a:e3}
  \left\| \int_{\bbR^{mn}}
    \frac{f(UAx)}{(|x_1|+\ldots+|x_m|)^{\lambda}} dx_1\ldots dx_m\right\|_{L_{x_{m+1}}^q}
    \lesssim \|f\|_{L^{\vec p}}.
\end{equation}
Assume that $p_{k_i}=\infty$ for some $1\le i\le \nu$.
Set $\vec{\tilde p}=(p_1,\ldots,p_{k_i-1},p_{k_i+1}$, $\ldots,p_{m+1})$
and
$f(x)=\tilde  f(x_1,\ldots,x_{k_i-1},x_{k_i+1},\ldots,x_{m+1})$
with $\tilde f\in L^{\vec {\tilde p}}$. Then we have
$f\in L^{\vec p}$.
Let $\tilde A$ be the submatrix consisting of
the first $(k_i-1)n$ and the last $(m-k_i+1)n$ rows of $UA$.
We see from (\ref{eq:a:e3}) that
\[
  \left\| \int_{\bbR^{mn}}
    \frac{\tilde f(\tilde Ax)}{(|x_1|+\ldots+|x_m|)^{\lambda}} dx_1\ldots dx_m\right\|_{L_{x_{m+1}}^q}
    \lesssim \|\tilde f\|_{L^{\vec {\tilde p}}}.
\]
Note that the rank of the last $n$ columns of $\tilde A$
is less than $n$.
Hence there is some $n\times n$ invertible matrix $V$ such that
the last column of
\[
\bar A:= \tilde A \begin{pmatrix}
    I_{mn}& \\
       & V
 \end{pmatrix}
\]
is $\vec 0$.
By a change of variable of the form
$x_{m+1}\rightarrow Vx_{m+1}$, we get
\[
  \left\| \int_{\bbR^{mn}}
    \frac{\tilde f(\bar A x)}{(|x_1|+\ldots+|x_m|)^{\lambda}} dx_1\ldots dx_m\right\|_{L_{x_{m+1}}^q}
    \lesssim \|\tilde f\|_{L^{\vec {\tilde p}}}.
\]
But $ \tilde f(\bar Ax) $ is independent of $x_{m+1}^{(n)}$. Therefore,
\[
  \left\| \int_{\bbR^{mn}}
    \frac{\tilde f(\bar A x)}{(|x_1|+\ldots+|x_m|)^{\lambda}} dx_1\ldots dx_m\right\|_{L_{x_{m+1}}^q}
=\infty,
\]
which is a contradiction.

(P2)  We show that $q\ge p_{k_i}$ for all $1\le i\le \nu$.

By Lemma~\ref{Lm:q}, $q\ge 1$.
For $1\le i\le m+1$,
define the  matrix $B_i$ by
\[
  B_i = \begin{pmatrix}
  A_{i,m+1}\\ \vdots \\ A_{m+1,m+1}
\end{pmatrix}.
\]
Since $\rank(B_{k_i}) >\rank(B_{k_i+1})$ for $1\le i\le \nu$,
there is some $z\in \bbR^n$ such that    $B_{k_i}z \ne 0$ while $B_{k_i+1}z = 0$.

For $a>0$, set $f_a(x) = f(x + a B_1 z)$. We have
\begin{align*}
\|T_{\lambda}f + T_{\lambda}f_a\|_{L^q}
&\le \|T_{\lambda}\|_{L^{\vec p}\rightarrow L^q}
  \|f + f_a\|_{L^{\vec p}} \\
&=\|T_{\lambda}\|_{L^{\vec p}\rightarrow L^q}
  \|f + f(\cdot + aB_1z)\|_{L^{\vec p}} \\
&\rightarrow 2^{1/p_{k_i}} \|T_{\lambda}\|_{L^{\vec p}\rightarrow L^q}
  \|f\|_{L^{\vec p}},
 \quad a\rightarrow\infty,
\end{align*}
where we use Lemma~\ref{Lm: translation} in the last step.
On the other hand, denote the submatrix consisting of the first $mn$ columns of $A$
by $B_0$. We have
\begin{align*}
T_{\lambda}f(x_{m+1}) &+ T_{\lambda}f_a(x_{m+1})\\
&= \int_{\bbR^{mn}}\! \frac{f(B_0 y+B_1 x_{m+1}) + f(B_0y+B_1x_{m+1} + aB_1z)}{|y|^{\lambda}} dy \\
&= T_{\lambda}f(x_{m+1}) + T_{\lambda}f(x_{m+1}+az).
\end{align*}
Hence
\[
  \lim_{a\rightarrow\infty} \|T_{\lambda}f  + T_{\lambda}f_a\|_{L^q}
  = 2^{1/q} \|T_{\lambda}f\|_{L^q}.
\]
Therefore,
\[
  2^{1/q} \|T_{\lambda}f\|_{L^q}\le 2^{1/p_{k_i}} \|T_{\lambda}\|_{L^{\vec p}\rightarrow L^q}
  \|f\|_{L^{\vec p}}.
\]
It follows that $q\ge p_{k_i}$, $1\le i\le \nu$.

(P3) We show that
$T_{\lambda}$ is bounded if and only if
\begin{align}
\left\|\frac{ h( x_1 ) }
    {( \sum_{i=2}^{m+1} |A_{i,m+1}x_1 - x_i|)^{\lambda}}\right\|_{L^{\vec p'}}
 \lesssim \|h\|_{L^{q'}}. \label{eq:w:e43}
\end{align}

By the duality, $T_{\lambda}$ is bounded from $L^{\vec p}$
to $L^q$ if and only if
for any $f\in L^{\vec p}$ and $h\in L^{q'}$,
\[
  \int_{\bbR^{(m+1)n}} \frac{f(Ax) h(x_{m+1})}{(|x_1|+\ldots+|x_m|)^{\lambda}} dx_1\ldots dx_{m+1}
  \lesssim  \|f\|_{L^{\vec p}} \|h\|_{L^{q'}}.
\]
By a change of variable of the form $x \rightarrow A^{-1}x$, we get
\begin{equation}\label{eq:fh}
  \int_{\bbR^{(m+1)n}} \frac{f(x) h((A^{-1}x)_{m+1})}
    {(|(A^{-1}x)_1|+\ldots+|(A^{-1}x)_m|)^{\lambda}} dx_1\ldots dx_{m+1}
  \lesssim  \|f\|_{L^{\vec p}} \|h\|_{L^{q'}},
\end{equation}
where $(A^{-1}x)_i$ stands for the $n$-dimensional vector consisting of the
$(in-n+1)$-th, $\ldots$, the $in$-th entries of
$A^{-1}x$.
Note that (\ref{eq:fh})  is equivalent to
\begin{equation}\label{eq:h}
  \left\|\frac{ h((A^{-1}x)_{m+1})}
    {(|(A^{-1}x)_1|+\ldots+|(A^{-1}x)_m|)^{\lambda}} \right\|_{L^{\vec p'}} \lesssim \|h\|_{L^{q'}},
    \qquad \forall h\in L^{q'}.
\end{equation}

Let $y=(y_1^*,\ldots, y_m^*)^*$ be the $mn$-dimensional vector consisting of the last $mn$ coordinates of $x$,
i.e., $y_i = x_{i+1}$ for $1\le i\le m$.
Denote    $A^{-1} = \begin{pmatrix}
\tilde A_1 & \tilde A_2 \\
\tilde A_3 & \tilde A_4
\end{pmatrix}$,
where $\tilde A_3$ is an $n\times n$ matrix.
By Lemma~\ref{Lm:La}, $\tilde A_3$ is invertible. We have
\begin{align*}
& \left\|\frac{  h((A^{-1}x)_{m+1}) }
    {(|(A^{-1}x)_1|+\ldots+|(A^{-1}x)_m|)^{\lambda }}\right\|_{L_{x_1}^{p'_1}} \\
&=\left\| \frac{  h(\tilde A_3 x_1 + \tilde A_4 y) }
    { |\tilde A_1 x_1 + \tilde A_2 y|^{\lambda }}\right\|_{L_{x_1}^{p'_1}} \\
&\approx \left\|\frac{ h( x_1 ) }
    {|\tilde A_1 \tilde A_3^{-1}x_1 + ( \tilde A_2 - \tilde A_1\tilde A_3^{-1}\tilde A_4) y|^{\lambda }}\right\|_{L_{x_1}^{p'_1}}\\
&\approx \left\|\frac{  h( x_1 )  }
    { |( \tilde A_2 - \tilde A_1\tilde A_3^{-1}\tilde A_4)^{-1}\tilde A_1 \tilde A_3^{-1}x_1 +  y|^{\lambda }}\right\|_{L_{x_1}^{p'_1}} \\
&=  \left\|\frac{  h( x_1 )  }
    {( \sum_{i=2}^{m+1} |A_{i,m+1}x_1 -     y_{i-1}|)^{\lambda }} \right\|_{L_{x_1}^{p'_1}}
    \mbox{\hskip 4em (using Lemma~\ref{Lm:La})}\\
&=  \left\|\frac{  h( x_1 )  }
    {( \sum_{i=2}^{m+1} |A_{i,m+1}x_1 - x_i|)^{\lambda }}
     \right\|_{L_{x_1}^{p'_1}}.
\end{align*}
Hence $T_{\lambda}$ is bounded if and only if (\ref{eq:w:e43}) is true.

(P4) We show that
$q<p_1$.

We see from (P3) and Lemma~\ref{Lm:ADx-y} that
\begin{equation}\label{eq:h a}
\left\|\frac{ h( x_1 ) }
    {W(x_1,\ldots,x_{m+1})^{\lambda}}\right\|_{L^{\vec p'}}
 \lesssim \|h\|_{L^{q'}}
    \quad \forall h\in L^{q'} ,
\end{equation}
where
\[
  W(x_1,\ldots,x_{m+1}) = \sum_{l=1}^n |x_1^{(l)} - x_{i_l}^{(j_l)}|
  + \sum_{\substack{2\le i\le m+1, 1\le j\le n\\
  (i,j)\ne (i_l,j_l), 1\le l\le n}} |x_i^{(j)}|
\]
and
$\{(i_l,j_l):\, 1\le l\le n\}$ is defined by (\ref{eq:il jl}).

There are two cases.

(A1) $p_1>1$.

Set $h(x_1)  =  1/|x_1|^{\alpha}$ for $|x_1|<1$ and $0$ for others,
where $\alpha>0$ is a constant to be determined later.
We have
\begin{align*}
I_1
&:= \int_{\bbR^n} \frac{|h(x_1)|^{p'_1}dx_1}{W(x_1,\ldots,x_{m+1})^{\lambda p'_1}} \\
&\gtrsim
 \int_{|x_1|<1} \frac{dx_1}
    {|x_1|^{\alpha p'_1}( 1+|x_2| +\ldots + |x_{m+1}|   )^{\lambda p'_1} }.
\end{align*}
If $q>p_1$, then $q'<p'_1$. Hence there is some $\alpha>0$ such that $\alpha p'_1 >n$ while $\alpha q'<n$.
Consequently, $h\in L^{q'}$
and  $I_1 = \infty$, which contradicts with (\ref{eq:h a}).

If $q=p_1$, then we have $ \lambda = n/p'_2 + \ldots + n/p'_{m+1}$.
Since $\lambda>0$, there is some $2\le i\le m+1$
such that $p_i>1$.
There are two subcases.

(A1)(a)\,\,  $p_{m+1}>1$.

Suppose that
$\alpha q'<n$.  Then
\[
  I_1 \gtrsim  \frac{1}{( 1+ |x_2|+\ldots+|x_{m+1}| )^{\lambda p'_1}}.
\]
It follows that
\begin{align*}
 I_h &:= \left\| \ldots\left\|\int_{\bbR^n} \frac{|h(x_1)|^{p'_1}dx_1}{W(x_1,\ldots,x_{m+1})^{\lambda p'_1}}
   \right\|_{L_{x_2}^{p'_2/p'_1}}
    \ldots\right\|_{L_{x_{m+1}}^{p'_{m+1}/p'_1}}  \\
   &\gtrsim   \Big\|\frac{1}{( 1+ |x_2|+\ldots+|x_{m+1}| )^{\lambda}} \Big\|_{L^{p'_{m+1}}( \ldots (L^{p'_2}))}^{p'_1}.
\end{align*}
Denote $a = 1+   |x_3|+\ldots+|x_{m+1}|$.
A simple computation shows that
\begin{align*}
& \hskip -10mm\int_{\bbR^n}\frac{1}{( 1+ |x_2|+\ldots+|x_{m+1}| )^{\lambda p'_2}} dx_2   \\
 & \approx  \int_0^{\infty} \frac{t^{n-1}}{(   t+ a )^{\lambda p'_2}} dt  \\
&\ge
  \int_a^{2a} \frac{t^{n-1}}{(   t+ a )^{\lambda p'_2}} dt  \\
&\gtrsim
   \frac{1}{( 1+ |x_3|+\ldots+|x_{m+1}| )^{\lambda p'_2 -n}}.
\end{align*}
Hence
\begin{align*}
 \left\|\frac{1}{( 1+ |x_2|+\ldots+|x_{m+1}| )^{\lambda}}
\right\|_{L_{x_2}^{p'_2}}
&\gtrsim  \frac{1}{( 1+ |x_3|+\ldots+|x_{m+1}| )^{\lambda - n/ p'_2}}.
\end{align*}
Note that the above inequality is true for all $1\le p_2\le \infty$.
Computing the $L^{p'_l}$ norm with respect to $x_l$ successively, $3\le l\le m+1$, we get
\begin{align*}
 I_h & \gtrsim \bigg(\int_{\bbR^n} \frac{1}{(1+|x_{m+1}|)^{n}} dx_{m+1} \bigg)^{p'_1/p'_{m+1}} = \infty,
\end{align*}
which contradicts with (\ref{eq:h a}).

(A1)(b)\,\,  $p_{m+1}=1$

In this case, there is some $2\le k_0\le m$ such that $p_{k_0}>1$
while $p_{k_0+1}=\ldots = p_{m+1}=1$.

By Lemma~\ref{Lm:end points ii}, we have
\[
\left\|\frac{ h( x_1 ) }
    {W(x_1,\ldots,x_{k_0})^{\lambda}}\right\|_{L^{\vec {\tilde p}'}}
 \lesssim \|h\|_{L^{q'}}
    \quad \forall h\in L^{q'} ,
\]
where $\vec{\tilde p} = (p_1, \ldots, p_{k_0})$ and
\[
  W(x_1,\ldots,x_{k_0}) = \sum_{\substack{1\le l\le n \\
      i_l\le k_0}}^n |x_1^{(l)} - x_{i_l}^{(j_l)}|
+\sum_{\substack{1\le l\le n \\
      i_l> k_0}}^n |x_1^{(l)} |
  + \sum_{\substack{2\le i\le k_0, 1\le j\le n\\
  (i,j)\ne (i_l,j_l), 1\le l\le n}} |x_i^{(j)}|.
\]
When $|x_1|\le 1$, we have
\[
  W(x_1,\ldots,x_{k_0}) \lesssim 1+ \sum_{i=2}^{k_0}|x_i|.
\]
By replacing $m+1$ with $k_0$ in Case (A1)(a), we get a contradiction.

(A2)  $p_1=1$.

Use  notations introduced in Case (A1).
Set $h(x_1)  = \chi^{}_{\{|x_1|\le 1\}}(x_1)  /|x_1|^{\alpha}$.
If $q>1$, then we can find some $\alpha>0$ such that
$\alpha q'<n$. Consequently,
\begin{align*}
&\left\|  \frac{ |h(x_1)| }
    {W(x_1,\ldots,x_{m+1})^{\lambda }}\right\|_{L_{x_1}^{p'_1}}
 \gtrsim
 \left\| \frac{1}
    {( 1+|x_2| +\ldots + |x_{m+1}|   )^{\lambda }  |x_1|^{\alpha} }
    \right\|_{L_{x_1}^{p'_1}}
   =\infty,
\end{align*}
which contradicts with (\ref{eq:h a}).

If $q=1$, by setting $h\equiv 1$, we  get
\begin{align*}
  \left\|  \frac{h(x_1)}
    {W(x_1,\ldots,x_{m+1})^{\lambda }}\right\|_{L^{\vec p'}}
\gtrsim
\left\|  \frac{1}
    {(\sum_{i=1}^{m+1}|x_i|)^{\lambda }}\right\|_{L^{\vec p'}}
   =\infty.
\end{align*}
Again, we get a contradiction.

(P5) We show that the rank of the $mn\times n$ matrix $D:=(A_{i,m+1})_{2\le i\le m+1}$ is $n$, i.e., $r_2=n$.

Assume on the contrary that $\rank(D)<n$. Then there is some $n\times n$ invertible matrix $V$ such that the first column of $DV$ is zero. Consequently,
$\sum_{i=2}^{m+1}| A_{i,m+1} Vx_1$ $- x_i| $ is a function of
$x_1^{(2)}$, $\ldots$, $x_1^{(n)}$.

Replacing $h$ by $h(V^{-1}\cdot)$ in (\ref{eq:w:e43}), we get
\[
 \left\|\frac{ h(V^{-1} x_1 ) }
    {( \sum_{i=2}^{m+1} |A_{i,m+1}x_1 - x_i|)^{\lambda}}\right\|_{L^{\vec p'}}
  \lesssim \|h\|_{L^{q'}}.
\]
By a change of variable of the form $x_1\rightarrow V x_1$, we get
\begin{align}
&
\left\|\frac{ h(   x_1 ) }
    {( \sum_{i=2}^{m+1} |A_{i,m+1}V x_1 - x_i|)^{\lambda}}\right\|_{L^{\vec p'}}
  \lesssim \|h\|_{L^{q'}}.  \label{eq:s:e50}
\end{align}
Take some $g\in L^{q'}\setminus L^{p'_1}(\bbR)$, $h_1\in L^{q'}(\bbR^{n-1})$
and let $h = g\otimes h_1$. Then we have $h\in L^{q'}$.
Since
$\sum_{i=2}^{m+1}| A_{i,m+1} Vx_1 - x_i| $ is a function of
$x_1^{(2)}$, $\ldots$, $x_1^{(n)}$, we have
\[
  \left\|\frac{ h(   x_1 ) }
    {( \sum_{i=2}^{m+1} |A_{i,m+1}V x_1 - x_i|)^{\lambda}}\right\|_{L^{\vec p'}}
=  \left\|\frac{ g(x_1^{(1)}) h_1(x_1^{(2)},\ldots, x_1^{(n)}) }
    {( \sum_{i=2}^{m+1} |A_{i,m+1}V x_1 - x_i|)^{\lambda}}\right\|_{L^{\vec p'}}
=\infty,
\]
which  contradicts with  (\ref{eq:s:e50}).

 (P6) We show that there is some $2\le i
\le m+1$ such that $p_i>1$.

Assume on the contrary that   $p_i=1$ for $2\le i\le m+1$. We see from
Lemma~\ref{Lm:end points ii} that  (\ref{eq:w:e43})
is equivalent to
\[
  \Big\| \frac{h(x_1)}{(\sum_{i=2}^{m+1}
   |A_{i,m+1} x_1|)^{\lambda}} \Big\|_{L^{p'_1}}
  \lesssim \|h\|_{L^{q'}}.
\]
Since  $r_2=n$, we have
\[
  \sum_{i=2}^{m+1}
   |A_{i,m+1} x_1|\approx |x_1|.
\]
Therefore,
\begin{equation}\label{eq:w:e42}
  \Big\| \frac{h(x_1)}{|x_1|^{\lambda}} \Big\|_{L^{p'_1}}
  \lesssim \|h\|_{L^{q'}}.
\end{equation}
Moreover, we see from (P4) that
  $q<p_1 $.
Set
\[
  h(x_1) = \frac{1}{ |x_1|^{n/q'}
(\log 1/|x_1|)^{(1+\varepsilon)/q'} }
\cdot  \chi^{}_{\{|x_1|<1/2\}}(x_1),
\]
where $\varepsilon$ is a positive constant such that
$(1+\varepsilon)p'_1/q'<1$.
Then we have $h\in L^{q'}$ and
$   \|  {h(x_1)}/{|x_1|^{\lambda}} \|_{L^{p'_1}}=\infty$,
which contradicts with (\ref{eq:w:e42}).

(P7) We show that  $q\ge p_{k_0}$.

By Lemma~\ref{Lm:ADx-y},
\begin{equation}\label{eq:w:e41}
  \left\|\frac{ h(x_1)}
    {W(x_1,\ldots,x_{k_0})^{\lambda}} \right\|_{L^{\vec {\tilde p}'}} \lesssim \|h\|_{L^{q'}},
    \quad \forall h\in L^{q'},
\end{equation}
where $\vec{ \tilde p} =(p_1,\ldots, p_{k_0})$ satisfies that
\[
  \frac{1}{p_1}+\ldots + \frac{1}{p_{k_0}}
   = \frac{1}{q} + \frac{(k_0-1)n-\lambda}{n}
\]
and
\[
  W(x_1,\ldots,x_{k_0}) = \sum_{\substack{1\le l\le n \\
   i_l\le k_0}} |x_1^{(l)} - x_{i_l}^{(j_l)}|
   +
\sum_{\substack{1\le l\le n \\
   i_l> k_0}} |x_1^{(l)} |
  + \sum_{\substack{2\le i\le k_0, 1\le j\le n\\
  (i,j)\ne (i_l,j_l), 1\le l\le n}} |x_i^{(j)}|
\]
with $\{(i_l,j_l):\, 1\le l\le n\}$ being defined by (\ref{eq:il jl}).

If $q=1$, by setting $h\equiv 1$, we get
\[
   \left\|\frac{ h(x_1)}
    {W(x_1,\ldots,x_{k_0})^{\lambda}} \right\|_{L^{\vec {\tilde p}'}}=\infty,
\]
which contradicts with (\ref{eq:w:e41}).
Hence $q>1$.

Set $h(x_1) =  \frac{\chi^{}_{\{|x_1|<1/2\}}(x_1)}{|x_1|^{n/q'} (\log(1/|x_1|))^{(1+\varepsilon)/q'}}$, where $\varepsilon>0$. Then we have $h\in L^{q'}$.
Moreover, for $a:=\sum_{i=2}^{k_0}|x_i| <1/2^2$
and $|x_1|<a$,
we have
\[
  W(x_1,\ldots,x_{k_0}) \le 2a.
\]
Hence
\begin{align*}
 & \left\|\frac{ h(x_1)}
    {W(x_1,\ldots,x_{k_0})^{\lambda}} \right\|_{L_{x_1}^{p'_1}} \\
&\gtrsim \left(
  \int_{a^2 <|x_1|<a}
  \frac{dx_1}{(\sum_{i=2}^{k_0 }|x_i|)^{\lambda p'_1+np'_1/q'}
  (\log(1/|x_1|))^{(1+\varepsilon)p'_1/q'}
  }
\right)^{1/p'_1} \\
&\gtrsim \frac{1}{(\sum_{i=2}^{k_0 }|x_i|)^{\lambda+n/q'-n/ p'_1}
  (\log(1/\sum_{i=2}^{k_0 }|x_i|))^{(1+\varepsilon)/q'}}.
\end{align*}
Similarly arguments show that for $\sum_{i=3}^{k_0 }|x_i|<1/2^3$,
\begin{align*}
  &\left\|\frac{ h(x_1)}
    {W(x_1,\ldots,x_{k_0})^{\lambda}} \right\|_{L_{(x_1,x_2)}^{(p'_1,p'_2)}}\\
&\gtrsim \frac{1}{(\sum_{i=3}^{k_0 }|x_i|)^{\lambda+n/q'
-n/ p'_1-n/ p'_2}
  (\log(1/\sum_{i=3}^{k_0 }|x_i|))^{(1+\varepsilon)/q'}}.
\end{align*}
Note that the above inequality is also true for $p'_2=\infty$.
By induction, it is easy to see
that for $|x_{k_0 }| <1/2^{k_0 }$,
\begin{align*}
 & \hskip -2em \left\|\frac{ h(x_1)}
    {W(x_1,\ldots,x_{k_0})^{\lambda}} \right\|_{L_{(x_1,\ldots,x_{k_0-1})}^{(p'_1,
    \ldots,p'_{k_0-1})}}\\
 &\gtrsim
   \frac{1}{ |x_{k_0}|^{\lambda+n/q'-\sum_{i=1}^{k_0-1}n/p'_i}
  (\log(1/ |x_{k_0}|))^{(1+\varepsilon)/q'}}\\
 &=
   \frac{1}{ |x_{k_0}|^{  n/p'_{k_0}}
  (\log(1/ |x_{k_0}|))^{(1+\varepsilon)/q'}}.
\end{align*}
If $q<p_{k_0}$, then we can choose $\varepsilon>0$ small
enough such that
$ (1+\varepsilon)p'_{k_0}/q'<1$. Consequently,
\[
   \left\|\frac{ h(x_1)}
    {W(x_1,\ldots,x_{k_0})^{\lambda}} \right\|_{L^{\vec {\tilde p}'}}
=\infty,
\]
which contradicts with (\ref{eq:w:e41}).
Hence $ q\ge p_{k_0}  $.

(P8) We show that $\min\{p_{k_l}:\, 1\le l\le \nu\}<q$.

Assume on the contrary that  $\min\{p_{k_l}:\, 1\le l\le \nu\}\ge q$.
Since  $\max \{p_{k_l}:\, 1\le l\le \nu\}\le q$, we have
$q=p_{k_1}=\ldots = p_{k_{\nu}}$.

By Lemma~\ref{Lm:ADx-y}, we have
\begin{align}
\left\|\frac{ h( x_1 ) }
    {W(x_1,\ldots,x_{m+1})^{\lambda}}\right\|_{L^{\vec p'}}
 \lesssim \|h\|_{L^{q'}}
    \quad \forall h\in L^{q'} ,
  \label{eq:s:e41}
\end{align}
where
\begin{align*}
  W(x_1,\ldots,x_{m+1})
  &= \sum_{l=1}^n |x_1^{(l)} - x_{i_l}^{(j_l)}|
  + w(x_2,\ldots,x_{m+1}), \\
w(x_2,\ldots,x_{m+1})
&=\sum_{\substack{2\le i\le m+1, 1\le j\le n\\
  (i,j)\ne (i_l,j_l), 1\le l\le n}} |x_i^{(j)}|,
\end{align*}
and $\{(i_l,j_l):\, 1\le l\le n\}$ is defined by (\ref{eq:il jl}).

 Let
 $
h(x_1)  =
\chi_{\{ |x_1^{(l)}|\le 1, 1\le l\le n\}}(x_1)$, $0<\delta<1/(n+2)$,
 and
\[
  E = \{(x_2,\ldots,x_{m+1}):\,  w(x_2, \ldots, x_{m+1})<\delta , |x_{i_l}^{(j_l)}|\le \delta,
1\le l\le n  \}.
\]
Similarly to (\ref{eq:s:h}), we get
\begin{align*}
\left(
\int_{R^n} \frac{|h(x_1)|^{p'_1}dx_1 }{W(x_1,\ldots,x_{m+1})^{
   \lambda p'_1}}\right)^{1/p'_1}
\gtrsim \frac{\chi^{}_{E}(x_2,\ldots,x_{m+1})}{  w(x_2,\ldots,x_{m+1})^{\lambda-n/p'_1}}.
\end{align*}
Since $\lambda = \sum_{i=1}^{m+1}n/p'_1 - n/q'$
and $p_{k_1}=\ldots = p_{k_{\nu}}=q$, we have
\begin{align*}
\lambda -\frac{n}{p'_1}
&=  \sum_{i=2}^{m+1} \frac{n}{p'_i}  - \frac{n}{q'}\\
&=  \sum_{i=2}^{m+1} \frac{n}{p'_i}  - \sum_{i=2}^{m+1} \frac{r_i-r_{i+1}}{q'}\\
&=  \sum_{i=2}^{m+1} \frac{n}{p'_i}  - \sum_{l=1}^{\nu} \frac{r_{k_l}-r_{k_{l+1}}}{p'_{k_l}}\\
&= \sum_{i=2}^{m+1}\frac{n- ( r_i - r_{i+1})}{p'_i},
\end{align*}
where we use the fact that $r_i-r_{i+1}=0$ if $i\not\in\{k_l:\, 1\le l\le \nu\}$.
Note that
\[
  w(x_2,\ldots, x_{m+1}) = \sum_{i=2}^{m+1} \sum_{j=1}^{n- ( r_i - r_{i+1})} |x_i^{(j)}|.
\]
We have
\[
 \left\|\frac{ h( x_1 ) }
    {W(x_1,\ldots,x_{m+1})^{\lambda}}\right\|_{L^{\vec p'}}
    =\infty,
\]
which contradicts with (\ref{eq:s:e41}).
This completes the proof of the necessity.

\subsection{Proof of Theorem~\ref{thm:main}:  The Sufficiency}

We see from (P3) in the proof of the necessity of Theorem~\ref{thm:main}
that $T_{\lambda}$ is bounded if and only if
\begin{align}
\left\|\frac{ h( x_1 ) }
    {( \sum_{i=2}^{m+1} |A_{i,m+1}x_1 - x_i|)^{\lambda}}\right\|_{L^{\vec p'}}
 \lesssim \|h\|_{L^{q'}}. \label{eq:w:e43a}
\end{align}

If $p_1 \ge p_i$ for all $2\le i\le m+1$,
then (\ref{eq:w:e43a}) is equivalent to
\begin{align}
\left\|\int_{\bbR^n} \frac{|h( x_1 )|^{p'_1} dx_1 }
    {( \sum_{i=2}^{m+1} |A_{i,m+1}x_1 - x_i|)^{\lambda p'_1}}\
    \right\|_{L_{(x_2,\ldots,x_{m+1})}^{(s'_2,\ldots,s'_{m+1})}}
 \lesssim \||h|^{p'_1}\|_{L^{q'/p'_1}}, \label{eq:w:e43b}
\end{align}
where $s'_i = p'_i/p'_1$.
By duality, (\ref{eq:w:e43b}) is equivalent to
\[
\left\|\int_{\bbR^{mn}} \frac{ f(x_2, \ldots,x_{m+1})dx_2\ldots dx_{m+1} }
    {( \sum_{i=2}^{m+1} |A_{i,m+1}x_1 - x_i|)^{\lambda p'_1}}\
    \right\|_{L_{x_1}^{(q'/p'_1)'}}
 \lesssim \|f\|_{L^{\vec s}},
\]
where $\vec s = (s_2, \ldots, s_{m+1})$.
Note that
\[
  \frac{1}{s_2} + \ldots + \frac{1}{s_{m+1}} = \frac{1}{(q'/p'_1)'} + \frac{mn-\lambda p'_1}{n}.
\]
Now the conclusion follows from Theorem~\ref{thm:factional:D}.

Next we assume that there is some $2\le i\le m+2$ such that $p_i>p_1$.
We begin with a simple lemma.

\begin{Lemma}\label{Lm:norm a}
Let $R$ be a positive constant, $\vec p=(p_1,\ldots,p_m)$ with $1\le p_i\le \infty$,
and $\alpha> \sum_{i=1}^m n_i/p_i$, where $n_i$ are positive integers.
Then we have
\[
  \left\|\frac{\chi^{}_{\{\sum_{i=1}^m |x_i|\ge R\}}(x_1,\ldots,x_m)}{(\sum_{i=1}^m|x_i|)^{\alpha}}
  \right\|_{L^{\vec p}}
  \lesssim R^{  n_1/p_1+\ldots+n_m/p_m - \alpha},
\]
where $x_i\in\bbR^{n_i}$.
\end{Lemma}

\begin{proof}
If $\sum_{i=2}^m |x_i|<R$, we have
\begin{align*}
&\left\|\frac{\chi^{}_{\{\sum_{i=1}^m |x_i|\ge R\}}(x_1,\ldots,x_m)}{(\sum_{i=1}^m|x_i|)^{\alpha}}
  \right\|_{L_{x_1}^{p_1}}^{p_1}
   \\
& = \int_{|x_1|\ge R-\sum_{i=2}^m |x_i|}
    \frac{\chi^{}_{\{\sum_{i=2}^m |x_i|< R\}}(x_2,\ldots,x_m)dx_1}
     {(\sum_{i=1}^m|x_i|)^{\alpha p_1}}\\
&\lesssim
     \frac{\chi^{}_{\{\sum_{i=2}^m |x_i|< R\}}(x_2,\ldots,x_m) }
     {R^{\alpha p_1-n_1}}.
\end{align*}
And if $\sum_{i=2}^m|x_i|\ge R$, we have
\begin{align*}
&\left\|\frac{\chi^{}_{\{\sum_{i=1}^m |x_i|\ge R\}}(x_1,\ldots,x_m)}{(\sum_{i=1}^m|x_i|)^{\alpha}}
  \right\|_{L_{x_1}^{p_1}}^{p_1}
   \\
& = \int_{R^{n_1}}
    \frac{\chi^{}_{\{\sum_{i=2}^m |x_i|\ge  R\}}(x_2,\ldots,x_m)dx_1}
     {(\sum_{i=1}^m|x_i|)^{\alpha p_1}}\\
&\lesssim
     \frac{\chi^{}_{\{\sum_{i=2}^m |x_i|\ge  R\}}(x_2,\ldots,x_m) }
     {(\sum_{i=2}^m|x_i|)^{\alpha p_1-n_1}}.
\end{align*}
Hence
\begin{align*}
I_1 &:= \left\|\frac{\chi^{}_{\{\sum_{i=1}^m |x_i|\ge R\}}(x_1,\ldots,x_m)}{(\sum_{i=1}^m|x_i|)^{\alpha}}
  \right\|_{L_{x_1}^{p_1}}
      \\
  &\lesssim
      \frac{\chi^{}_{\{\sum_{i=2}^m |x_i|< R\}}(x_2,\ldots,x_m) }
     {R^{\alpha -n_1/p_1}}
   +
     \frac{\chi^{}_{\{\sum_{i=2}^m |x_i|\ge  R\}}(x_2,\ldots,x_m) }
     {(\sum_{i=2}^m|x_i|)^{\alpha -n_1/p_1}}.
\end{align*}
Note that the above inequality is true for all $1\le p_1\le \infty$.
Since
\[
  \left\|\frac{\chi^{}_{\{\sum_{i=2}^m |x_i|< R\}}(x_2,\ldots,x_m) }
     {R^{\alpha -n_1/p_1}}
     \right\|_{L_{(x_2,\ldots,x_m)}^{(p_2,\ldots,p_m)}}
 \lesssim R^{n_1/p_1+\ldots+n_m/p_m-\alpha},
\]
we have
\[
  \|I_1\|_{L_{(x_2,\ldots,x_m)}^{(p_2,\ldots,p_m)}}
 \lesssim
 R^{n_1/p_1+\ldots+n_m/p_m-\alpha}
 + \left\|\frac{\chi^{}_{\{\sum_{i=2}^m |x_i|\ge  R\}}(x_2,\ldots,x_m) }
     {(\sum_{i=2}^m|x_i|)^{\alpha -n_1/p_1}}
     \right\|_{L_{(x_2,\ldots,x_m)}^{(p_2,\ldots,p_m)}}\!\!\!.
\]
By induction, we get
$
\|I_1\|_{L_{(x_2,\ldots,x_m)}^{(p_2,\ldots,p_m)}}
 \lesssim
 R^{n_1/p_1+\ldots+n_m/p_m-\alpha}.
$
\end{proof}

The following estimate is used in the proof of Theorem~\ref{thm:main}.

\begin{Lemma}\label{Lm:h}
Suppose that $1<p_{m+1}<q<p_1\le \infty$
and $1\le p_i\le \infty$ for $2\le i\le m$.
Let $n_1$, $\ldots$, $n_m$ be positive integers
and
\begin{equation}\label{eq:w:e49}
  \alpha = \sum_{i=1}^m \frac{n_i}{p'_i} + \frac{n_1}{p'_{m+1}} - \frac{n_1}{q'}>0.
\end{equation}
Then for any $h\in L^{q'}$, we have
\begin{equation}\label{eq:w:e48}
  \left\|\frac{ h(x_1)}
    {(|x_1-x_{m+1}|+  \sum_{i=2}^m |x_i|)^{\alpha}} \right\|_{L^{\vec p'}} \lesssim \|h\|_{L^{q'}},
    \quad \forall h\in L^{q'},
\end{equation}
where $\|\cdot\|_{L^{\vec p'}}
= \Big\| \| \cdot\|_{L_{x_1}^{p'_1}}\!\ldots \Big\|_{L_{x_{m+1}}^{p'_{m+1}}}$,
$x_1,x_{m+1}\in R^{n_1}$
and  $x_i\in \bbR^{n_i}$ for $2\le i\le m$.
\end{Lemma}

\begin{proof}
Let $R>0$ be a constant to be determined later.
We have
\begin{align}
&  \int_{\bbR^{n_1}}\frac{ |h(x_1)|^{p'_1} dx_1}
    {(|x_1-x_{m+1}|+  \sum_{i=2}^m |x_i|)^{\alpha p'_1}}
     \nonumber \\
  &=
  \bigg(\int_{|x_1-x_{m+1}|<R}
  +
    \int_{|x_1-x_{m+1}|\ge R}\bigg)
    \frac{ |h(x_1)|^{p'_1} dx_1}
    {(|x_1-x_{m+1}|+  \sum_{i=2}^m |x_i|)^{\alpha p'_1}}\nonumber
\\
&= I_1 + I_2. \label{eq:w:e6}
\end{align}
For fixed $x_2$, $\ldots$, $x_m$,
$I_1$ can be considered as the convolution
of $|h(x_1)|^{p'_1}$ and $ \chi^{}_{\{|x_1|<R\}}\!(x_1)/(|x_1|$ $+\sum_{i=2}^m|x_i|)^{\alpha p'_1}$
and  .
Hence
\begin{align*}
I_1 &\lesssim M|h|^{p'_1}(x_{m+1}) \int_{|x_1|<R}
    \frac{ dx_1}
    {(|x_1|+  \sum_{i=2}^m |x_i|)^{\alpha p'_1}}\\
&= M|h|^{p'_1}(x_{m+1}) \bigg(\int_{|x_1|<R}
    \frac{\chi^{}_{\{\sum_{i=2}^m |x_i|<R\}}(x_2,\ldots,x_m) dx_1}
    {(|x_1|+  \sum_{i=2}^m |x_i|)^{\alpha p'_1}}\\
&\qquad
    +
    \int_{|x_1|< R}
    \frac{\chi^{}_{\{\sum_{i=2}^m |x_i|\ge R\}}(x_2,\ldots,x_m) dx_1}
    {(|x_1|+  \sum_{i=2}^m |x_i|)^{\alpha p'_1}}
    \bigg) \\
&:= M|h|^{p'_1}(x_{m+1})(I_{11} + I_{12})    ,
\end{align*}
where $M$ is the Hardy-Littlewood maximal function.
Observe that $(a+b)^{\alpha} \le C_{\alpha} (a^{\alpha}+ b^{\alpha})$ for any $a,b,\alpha>0$. We have
\begin{align*}
\|I_1^{1/p'_1}\|_{L_{(x_2,\ldots,x_m)}^{(p'_2,\ldots,p'_m)}}
& \lesssim
\left(M|h|^{p'_1}(x_{m+1}) \right)^{1/p'_1} \\
&\qquad \times
\left(\|I_{11}^{1/p'_1}\|_{L_{(x_2,\ldots,x_m)}^{(p'_2,\ldots,p'_m)}}
+
\|I_{12}^{1/p'_1}\|_{L_{(x_2,\ldots,x_m)}^{(p'_2,\ldots,p'_m)}}
\right).
\end{align*}

First, we estimate
$\|I_{11}^{1/p'_1}\|_{L_{(x_2,\ldots,x_m)}^{(p'_2,\ldots,p'_m)}}$.
Since $ q>p_{m+1}$, by (\ref{eq:w:e49}), we have
\begin{equation}\label{eq:s:e2}
  \alpha        <\frac{n_1}{p'_1} + \ldots + \frac{n_m}{p'_{m}}.
\end{equation}

There are two cases.

 (A1) For any $1\le k\le m-1$,
$\alpha \ne \sum_{i=1}^k n_i/p'_i$.

If  $\alpha p'_1>n_1$, we have
\begin{align}
I_{11} & =  \int_{|x_1|<R}
    \frac{\chi^{}_{\{\sum_{i=2}^m |x_i|<R\}}(x_2,\ldots,x_m) dx_1}
    {(|x_1|+  \sum_{i=2}^m |x_i|)^{\alpha p'_1}} \nonumber \\
&\lesssim
\frac{\chi^{}_{\{\sum_{i=2}^m |x_i|<R\}}(x_2,\ldots,x_m)}
    {(\sum_{i=2}^m |x_i|)^{\alpha p'_1-n_1}}. \label{eq:w:e39}
\end{align}
By (\ref{eq:s:e2}), there is some $1\le k\le m-1$ such that
\[
  \frac{n_1}{p'_1} + \ldots + \frac{n_k}{p'_{k}}
<
\alpha < \frac{n_1}{p'_1} + \ldots + \frac{n_{k+1}}{p'_{k+1}}.
\]
By (\ref{eq:w:e39}), we have for $k\ge 2$,
\begin{align*}
\|I_{11}^{1/p'_1}\|_{L_{x_2}^{p'_2}}
 & \lesssim
\frac{\chi^{}_{\{\sum_{i=3}^m |x_i|<R\}}(x_3,\ldots,x_m)}
    {(\sum_{i=3}^m |x_i|)^{\alpha  -n_1/p'_1-n_2/p'_2}}.
\end{align*}
Similar arguments show that
\begin{align*}
\|I_{11}^{1/p'_1}\|_{L_{(x_2,\ldots,x_{k})}^{(p'_2,\ldots,p'_{k})}}
 & \lesssim
\frac{\chi^{}_{\{\sum_{i=k+1}^m |x_i|<R\}}(x_{k+1},\ldots,x_m)}
    {(\sum_{i={k+1}}^m |x_i|)^{\alpha  -n_1/p'_1-\ldots -n_k/p'_k}}
\end{align*}
and
\begin{align*}
\|I_{11}^{1/p'_1}\|_{L_{(x_2,\ldots,x_{k+1})}^{(p'_2,\ldots,p'_{k+1})}}
 & \lesssim
 R^{   n_1/p'_1+\ldots +n_{k+1}/p'_{k+1}-\alpha }
 \chi^{}_{\{\sum_{i=k+2}^m |x_i|<R\}}(x_{k+2},\ldots,x_m).
\end{align*}
Hence
\begin{align}
\|I_{11}^{1/p'_1}\|_{L_{(x_2,\ldots,x_m)}^{(p'_2,\ldots,p'_m)}}
 & \lesssim
 R^{   n_1/p'_1+\ldots +n_m/p'_m-\alpha }
=R^{n_1/q' - n_1/p'_{m+1}}. \label{eq:w:e40}
\end{align}

 If  $\alpha p'_1<  n_1$, we have
\begin{align}
I_{11} & =  \int_{|x_1|<R}
    \frac{\chi^{}_{\{\sum_{i=2}^m |x_i|<R\}}(x_2,\ldots,x_m) dx_1}
    {(|x_1|+  \sum_{i=2}^m |x_i|)^{\alpha p'_1}} \nonumber \\
&\lesssim
 R^{n_1-\alpha p'_1}\chi^{}_{\{\sum_{i=2}^m |x_i|<R\}}(x_2,\ldots,x_m). \nonumber
\end{align}
Hence (\ref{eq:w:e40}) is also true.

(A2) There is some $1\le k\le m-1$ such that
$\sum_{i=1}^{k-1} n_i/p'_i< \alpha = \sum_{i=1}^k n_i/p'_i$.

In this case, $p'_k<\infty$.
As in Case (A1), we have
\begin{equation}\label{eq:s:e3}
\|I_{11}^{1/p'_1}\|_{L_{(x_2,\ldots,x_k)}^{(p'_2,\ldots,p'_k)}}
  \lesssim
  \left(\log  \frac{R}{\sum_{i=k+1}^m |x_i|}\right)^{1/p'_k}
  \chi^{}_{\{\sum_{i=k+1}^m |x_i|<R\}}(x_{k+1},\ldots,x_m).
\end{equation}
Denote $a=\sum_{i=k+2}^m |x_i|$. If $p'_{k+1}<\infty$, we have
\begin{align*}
&\|I_{11}^{1/p'_1}\|_{L_{(x_2,\ldots,x_{k+1})}^{(p'_2,\ldots,p'_{k+1})}}^{p'_{k+1}}
   \\
&\lesssim
  \int_{|x_{k+1}|\le R -a }
  \left(\log \frac{R}{|x_{k+1}|+a} \right)^{p'_{k+1}/p'_k} dx_{k+1}
  \cdot \chi^{}_{\{a<R\}}(x_{k+2},\ldots,x_m) \\
&=
  \int_0^{R -a }
  \left(\log \frac{R}{t+a}\right)^{p'_{k+1}/p'_k} t^{n_{k+1}-1}dt
  \cdot  \chi^{}_{\{a<R\}}(x_{k+2},\ldots,x_m) \\
&\le
  R^{n_{k+1}-1}\int_0^{R -a }
  \left(\log \frac{R}{t+a}\right)^{p'_{k+1}/p'_k} dt
  \cdot  \chi^{}_{\{a<R\}}(x_{k+2},\ldots,x_m) \\
&=  R^{n_{k+1}}\int_1^{R/a}
  \left(\log t \right)^{p'_{k+1}/p'_k} \frac{dt}{t^2}
  \cdot  \chi^{}_{\{a<R\}}(x_{k+2},\ldots,x_m).
\end{align*}
Hence
\[
 \|I_{11}^{1/p'_1}\|_{L_{(x_2,\ldots,x_{k+1})}^{(p'_2,\ldots,p'_{k+1})}}
 \lesssim R^{n_{k+1}/p'_{k+1}} \chi^{}_{\{\sum_{i=k+2}^m |x_i|<R\}}(x_{k+2},\ldots,x_m).
\]
Therefore,
(\ref{eq:w:e40}) is also true.

If $p'_{k+1}=\infty$, we have
\[
\|I_{11}^{1/p'_1}\|_{L_{(x_2,\ldots,x_{k+1})}^{(p'_2,\ldots,p'_{k+1})}}
  \lesssim
  \left(\log \frac{R}{\sum_{i=k+2}^m |x_i|}\right)^{1/p'_k}
  \chi^{}_{\{\sum_{i=k+2}^m |x_i|<R\}}(x_{k+2},\ldots,x_m).
\]
By (\ref{eq:s:e2}), there is some $k+1\le i\le m$ such that $p'_i<\infty$.
Suppose that $p'_{k+1} = \ldots = p'_{l-1}=\infty$ while $p'_l<\infty$.
Then we have
\[
\|I_{11}^{1/p'_1}\|_{L_{(x_2,\ldots,x_{l-1})}^{(p'_2,\ldots,p'_{l-1})}}
  \lesssim
  \left(\log \frac{R}{\sum_{i=l}^m |x_i|}\right)^{1/p'_k}
  \chi^{}_{\{\sum_{i=l}^m |x_i|<R\}}(x_l,\ldots,x_m).
\]
Hence
\[
 \|I_{11}^{1/p'_1}\|_{L_{(x_2,\ldots,x_l)}^{(p'_2,\ldots,p'_l)}}
 \lesssim R^{n_l/p'_l} \chi^{}_{\{\sum_{i=l+1}^m |x_i|<R\}}(x_{l+1},\ldots,x_m).
\]
Again, we get
(\ref{eq:w:e40}) as desired.

Next we estimate
$\|I_{12}^{1/p'_1}\|_{L_{(x_2,\ldots,x_m)}^{(p'_2,\ldots,p'_m)}}$.
Since $\sum_{i=2}^m |x_i|\ge R $ and $|x_1|<R$, we have
$|x_1|+  \sum_{i=2}^m |x_i| \approx
\sum_{i=2}^m |x_i|$. Hence
\begin{align*}
I_{12}
&=\int_{|x_1|< R}
    \frac{\chi^{}_{\{\sum_{i=2}^m |x_i|\ge R\}}(x_2,\ldots,x_m) dx_1}
    {(|x_1|+  \sum_{i=2}^m |x_i|)^{\alpha p'_1}}\\
&\approx \frac{R^{n_1}  }
    {(  \sum_{i=2}^m |x_i|)^{\alpha p'_1}}
    \chi^{}_{\{\sum_{i=2}^m |x_i|\ge R\}}(x_2,\ldots,x_m).
\end{align*}

By the hypothesis, $q<p_1$. Hence
\[
  \alpha - \frac{n_2}{p'_2}-\ldots -\frac{n_m}{p'_m}
   =  \frac{n_1}{p'_{m+1}}
   + \frac{n_1}{p'_1}
    - \frac{n_1}{q'} >0.
\]
By Lemma~\ref{Lm:norm a}, we get
\begin{align*}
\|I_{12}^{1/p'_1}\|_{L_{(x_2,\ldots,x_m)}^{(p'_2,\ldots,p'_m)}}
&\lesssim  R^{n_1/p'_1+\ldots+n_m/p'_m-\alpha}
=R^{n_1/q' - n_1/p'_{m+1}}.
\end{align*}
Hence
\begin{align}
\|I_1^{1/p'_1}\|_{L_{(x_2,\ldots,x_m)}^{(p'_2,\ldots,p'_m)}}
& \lesssim
\left(M|h|^{p'_1}(x_{m+1}) \right)^{1/p'_1}R^{n_1/q' - n_1/p'_{m+1}}. \label{eq:w:e7}
\end{align}

Next we consider $I_2$. By H\"older's inequality, we have
\begin{align*}
I_2&=
\int_{|x_1-x_{m+1}|\ge R}
    \frac{ |h(x_1)|^{p'_1} dx_1}
    {(|x_1-x_{m+1}|+  \sum_{i=2}^m |x_i|)^{\alpha p'_1}}\\
&\lesssim
\frac{\|h\|_{q'}^{p'_1}}{(R+\sum_{i=2}^m|x_i|)^{\alpha p'_1 - n_1 + n_1p'_1/q'}}.
\end{align*}
Hence
\begin{align}
\||I_2|^{1/p'_1}\|_{L_{(x_2,\ldots,x_m)}^{(p'_2,\ldots,p'_m)}}
&\lesssim \|h\|_{q'} \left\|\frac{1}
{(R+\sum_{i=2}^m|x_i|)^{n_2/p'_2 + \ldots + n_1/p'_{m+1}}}
\right\|_{L_{(x_2,\ldots,x_m)}^{(p'_2,\ldots,p'_m)}}
 \nonumber\\
&\approx \|h\|_{q'} R^{-n_1/p'_{m+1}}. \label{eq:w:e8}
\end{align}
Set
\[
  R = \left(\frac{\|h\|_{q'}}{\left(M|h|^{p'_1}(x_{m+1})\right)^{1/p'_1}}
     \right)^{q'/n_1}.
\]
We see from (\ref{eq:w:e6}),
(\ref{eq:w:e7}) and (\ref{eq:w:e8}) that
\begin{align*}
 &\hskip -5mm \left\|\frac{ h(x_1)}
    {(|x_1-x_{m+1}|+  \sum_{i=2}^m |x_i|)^{\alpha}} \right\|_{L_{(x_1,\ldots,x_m)}^{(p'_1,\ldots,p'_m)}}
    \\
& \lesssim
 \|h\|_{q'}^{1-q'/p'_{m+1}}
 \cdot \left(M|h|^{p'_1}(x_{m+1})\right)^{q'/(p'_1p'_{m+1})}.
\end{align*}
Hence
\begin{align*}
   \left\|\frac{ h(x_1)}
    {(|x_1-x_{m+1}|+  \sum_{i=2}^m |x_i|)^{\alpha}} \right\|_{L^{\vec p'}}
 & \lesssim
    \|h\|_{q'}^{1 -  q'/p'_{m+1}}
 \| M|h|^{p'_1}(x_{m+1})^{q'/p'_1}\|_{L^1}^{1/p'_{m+1}}\\
 & \lesssim\|h\|_{q'}.
\end{align*}
This completes the proof.
\end{proof}

To prove the main result, we also need Minkowski's inequality of the following form.

\begin{Lemma}\label{Lm:minkowski}
Suppose that $\vec p = (p_1,\ldots, p_m)$
and $p_i\ge p_{i+1}$ for some $i$. Then we have
\[
  \|f\|_{L_{x_m}^{p_m}(\ldots L_{x_i}^{p_i}(L_{x_{i+1}}^{p_{i+1}}(\ldots)))}
  \le \|f\|_{L^{\vec p}}.
\]
\end{Lemma}

\begin{proof}
We prove the lemma only for $m=2$. Other cases can be proved similarly.

Suppose that $p_1\ge p_2$. By Minkowski's inequality, we have
\begin{align*}
\left\| \int_{\bbR^n} |f(x_1,x_2)|^{p_2} dx_2
  \right\|_{L_{x_1}^{p_1/p_2}}
&\le
 \int_{\bbR^n} \Big(\int_{\bbR^n} |f(x_1,x_2)|^{p_1}
 dx_1\Big)^{p_2/p_1}dx_2 .
\end{align*}
Hence
\[
  \|f\|_{L_{x_1}^{p_1}(L_{x_2}^{p_2})}
 \le
  \|f\|_{L_{x_2}^{p_2}(L_{x_1}^{p_1})}
  =\|f\|_{L^{\vec p}}.
\]
\end{proof}

We split the proof for the sufficiency into two parts:
one is for the case $p_{k_{\nu}}<q$,
and the other one is for the case $p_{k_{\nu}}=q$.
We point out that the proof for the  second case is more technical,
where
both the maximal function and the theory
of interpolation spaces and
interpolation theorems for mixed-norm Lebesgue spaces are involved.

\subsubsection{The Case   \texorpdfstring{$p_{k_{\nu}}<q$}{p k nu < q}}

First, we consider the case   $k_0=m+1$. That is,
$p_{m+1}>1$. There are three subcases.

(A1) \, $\rank(A_{m+1,m+1})=n$.

In this case, $\nu=1$, $i_{\nu}=m+1$ and    $p_{m+1}<q<p_1$.
By Lemma~\ref{Lm:ADx-y}, (\ref{eq:w:e43a}) is equivalent to
\[
  \left\|\frac{ h( x_1 ) }
    {W(x_1, \ldots,x_{m+1})^{\lambda}}\right\|_{L^{\vec p'}}
 \lesssim \|h\|_{L^{q'}},
\]
where
\[
  W(x_1, \ldots,x_{m+1}) = |x_1-x_{m+1}|+  \sum_{i=2}^m |x_i|.
\]
Now  we see from Lemma~\ref{Lm:h} that
the above inequality is true since $1<p_{m+1}< q< p_1\le \infty$.

(A2)\,\,  $0<\rank(A_{m+1,m+1})<n$.

In this case, we also have $k_{\nu}=m+1$   and $p_{m+1}<q<p_1$.

By Lemma~\ref{Lm:ADx-y},
(\ref{eq:w:e43a}) is equivalent to
\begin{align}
\left\|\frac{ h( x_1 ) }
    { W(x_1,\ldots,x_{m+1})^{\lambda}}\right\|_{L^{\vec p'}}
 \lesssim  \|h\|_{L^{q'}}, \label{eq:w:e52a}
\end{align}
where
\[
  W(x_1,\ldots,x_{m+1}) = \sum_{l=1}^n |x_1^{(l)} - x_{i_l}^{(j_l)}|
  + \sum_{\substack{2\le i\le m+1, 1\le j\le n\\
  (i,j)\ne (i_l,j_l), 1\le l\le n}} |x_i^{(j)}|
\]
and $\{(i_l,j_l):\, 1\le l\le n\}$ is defined by (\ref{eq:il jl}).

Let $\{\tau_2,\ldots,\tau_m\}$ be a rearrangement of
$\{ 2,\ldots,m\}$ such that
$\tau_2<\ldots < \tau_{\iota_0}$, $\tau_{\iota_0+1}<\ldots<\tau_m$,
$p_{\tau_i}\le p_1$ for $2\le i\le \iota_0$
and $p_{\tau_i}> p_1$ for $\iota_0+1\le i\le m$.
Then $p'_{\tau_i} < p'_1 \le  p'_{\tau_l}$ for $l\le   \iota_0< i$.
Since $q\ge p_{k_l}$ for $1\le l\le \nu$,
we have $\{k_1,\ldots, k_{\nu-1}\}\subset \{\tau_2, \ldots, \tau_{\iota_0}\}$.

By Lemma~\ref{Lm:minkowski}, we have
\begin{align*}
\left\|\cdot \right\|_{L^{\vec p'}}
 \le
 \left\|\cdot\right\|_{
    L_{(x_1, x_{\tau_2},\ldots,x_{\tau_m},x_{m+1})}^{
    (p'_1,  p'_{\tau_2},\ldots,p'_{\tau_m},  p'_{m+1})}
    }.
\end{align*}

Since $p'_1 \le p'_{\tau_i}$ for $2\le i\le \iota_0$, we  have
\begin{equation}\label{eq:si}
  s_i:= \Big(\frac{p'_{\tau_i}}{p'_1}\Big)' \in [1,\infty].
\end{equation}
Set $\vec s = (s_{\tau_2}, \ldots,s_{\tau_{\iota_0}})$.
We have
\begin{align}
& \left\|\frac{ h( x_1 ) }
    {W(x_1,\ldots,x_{m+1})^{\lambda}}\right\|_{
    L_{(x_1,x_{\tau_2}, \ldots,x_{\tau_{\iota_0}})}^{
    (p'_1,p'_{\tau_2}, \ldots,p'_{\tau_{\iota_0}})}
    }   \nonumber\\
&=\left\|
\int_{\bbR^n}\frac{ |h( x_1 )|^{p'_1} dx_1 }
    {W(x_1,\ldots,x_{m+1})^{\lambda p'_1}}
\right\|_{
    L_{(x_{\tau_2},\ldots,x_{\tau_{\iota_0}})}^{
    (s'_{\tau_2}, \ldots,s'_{\tau_{\iota_0}})}
    }^{1/p'_1} \nonumber   \\
&=\left(\sup_{\|f\|_{L^{\vec s}}=1}
\int_{\bbR^{\iota_0 n}}\frac{f(x_{\tau_2}, \ldots,x_{\tau_{\iota_0}}) |h( x_1 )|^{p'_1} dx_1 dx_{\tau_2}  \ldots dx_{\tau_{\iota_0}}}
    {W(x_1,\ldots,x_{m+1})^{\lambda p'_1}}\right)^{1/p'_1}.
  \label{eq:s:e60}
\end{align}

Note that $r_{\tau_2}=n$.
If $\tau_2 \in \{k_l:\, 1\le l\le \nu\}$,
then $ s'_{\tau_2} = p'_{\tau_2}/p'_1  \ge  q'/p'_1$.
We estimate the integration $dx_1^{(n-r_{\tau_2}+1)}dx_{\tau_2}^{(n-r_{\tau_2} + r_{\tau_2+1}+1)}
 \ldots dx_1^{(n-r_{\tau_2+1})}
  dx_{\tau_2}^{(n)}
$
with Young's inequality
and estimate the integration
$dx_{\tau_2}^{(1)} \ldots dx_{\tau_2}^{(n-r_{\tau_2} + r_{\tau_2+1})}$
with
H\"older's inequality.
We get
\begin{align*}
&
\int_{\bbR^{n+r_{\tau_2} - r_{\tau_2+1}}}
 \frac{f(x_{\tau_2}, \ldots,x_{\tau_{\iota_0}}) |h( x_1 )|^{p'_1}
 dx_1^{(n-r_{\tau_2}+1)}\ldots dx_1^{(n-r_{\tau_2+1})} dx_{\tau_2}}
    {W(x_1,\ldots,x_{m+1})^{\lambda p'_1}} \\
& \lesssim
  \frac{\|f(x_{\tau_2}, \ldots,x_{\tau_{\iota_0}})\|_{L_{x_{\tau_2}}^{s_{\tau_2}}}
  \|h( x_1 )\|_{L_{(x_1^{(n-r_{\tau_2}+1)},\ldots, x_1^{(n-r_{\tau_2+1})} )}^{q'}}^{p'_1}}
    {W_1(x_1,\ldots,x_{m+1})^{p'_1(\lambda -   n/p'_{\tau_2} -
     (r_{\tau_2} - r_{\tau_2+1})(1/p'_1
     - 1/q'))}},
\end{align*}
where
\[
  W_1(x_1,\ldots,x_{m+1}) = \sum_{l=n-r_{\tau_2+1}+1}^n |x_1^{(l)} - x_{i_l}^{(j_l)}|
  + \sum_{\substack{2\le i\le m+1, 1\le j\le n\\
    i\ne \tau_2\\
  (i,j)\ne (i_l,j_l), 1\le l\le n}} |x_i^{(j)}|.
\]

If $\tau_2\not\in  \{k_l:\, 1\le l\le \nu\}$,
then $r_{\tau_2} = r_{\tau_2+1}$. Using H\"older's inequality
when estimating the integration $dx_{\tau_2}$, we get
\begin{align*}
&
\int_{\bbR^{n}}
 \frac{f(x_{\tau_2}, \ldots,x_{\tau_{\iota_0}}) |h( x_1 )|^{p'_1}
 dx_{\tau_2}}
    {W(x_1,\ldots,x_{m+1})^{\lambda p'_1}} \\
& \lesssim
  \frac{\|f(x_{\tau_2}, \ldots,x_{\tau_{\iota_0}})\|_{L_{x_{\tau_2}}^{s_{\tau_2}}}
  |h( x_1 ) |^{p'_1}}
    {W_1(x_1,\ldots,x_{m+1})^{p'_1(\lambda -   n/p'_{\tau_2} -
     (r_{\tau_2} - r_{\tau_2+1})(1/p'_1
     - 1/q'))}}.
\end{align*}

For $3\le i\le \iota_0$,
 integrating with respect to
$ dx_1^{(n-r_{\tau_i}+1)}d x_{\tau_i}^{(n-r_{\tau_i} + r_{\tau_i+1}+1)}$,
$
 \ldots$, $  d x_1^{(n-r_{\tau_i+1})} dx_{\tau_i}^{(n)}$
and $d x_{\tau_i}^{(1)} \ldots dx_{\tau_i}^{(n-r_{\tau_i} + r_{\tau_i+1} )}$
using Young's  and H\"older's inequalities, respectively, we get

\begin{align*}
&\int_{\bbR^{\iota_0 n}}\frac{f(x_{\tau_2}, \ldots,x_{\tau_{\iota_0}}) |h( x_1 )|^{p'_1} dx_1 dx_{\tau_2}  \ldots dx_{\tau_{\iota_0}}}
    {W(x_1,\ldots,x_{m+1})^{\lambda p'_1}} \\
&\lesssim
\|f\|_{L^{\vec s}}\cdot    \int_{\bbR^{r_{k_{\nu}}}} \frac{\|h(x_1)\|_{L_{(x_1^{(1)},\ldots,x_1^{(n-r_{k_{\nu}})})
}^{q'}}^{p'_1} dx_1^{(n-r_{k_{\nu}}+1)}\ldots dx_1^{(n)}}
{W_2(x_1,\ldots,x_{m+1})^{\alpha p'_1}},
\end{align*}
where
\[
 W_2(x_1,\ldots,x_{m+1})  =
 \sum_{l=n-r_{k_{\nu}}+1}^n |x_1^{(l)} - x_{m+1}^{(l)}|
  + \sum_{i=\iota_0+1}^m |x_{\tau_i}| +
   \sum_{j=1}^{n-r_{k_{\nu}}} |x_{m+1}^{(j)}|
\]
and
\begin{align*}
\alpha &= \lambda - \sum_{i=2}^{\iota_0} \frac{n}{p'_{\tau_i}}
   - \frac{n-r_{k_{\nu}}}{p'_1}
   +  \frac{n-r_{k_{\nu}}}{q'} \\
 & =  \frac{ r_{k_{\nu}}}{p'_1} + \sum_{i= \iota_0+1 }^m \frac{n}{p'_{\tau_i}}
   + \frac{n}{p'_{m+1}}
  - \frac{ r_{k_{\nu}}}{q'}.
\end{align*}
Set $\tilde x = (x_1^{(1)},\ldots,x_1^{(n-r_{k_{\nu}})})$
and $\bar x_1 = (x_1^{(n-r_{k_{\nu}}+1)},\ldots, x_1^{(n)})$.
It follows from (\ref{eq:s:e60}) that
\begin{align}
&\left\|\frac{ h( x_1 ) }
    {W(x_1,\ldots,x_{m+1})^{\lambda}}\right\|_{
    L_{(x_1,x_{\tau_2}, \ldots,x_{\tau_{\iota_0}})}^{
    (p'_1,p'_{\tau_2}, \ldots,p'_{\tau_{\iota_0}})}
    }
 \lesssim   \Bigg\|
 \frac{\|h(\tilde x_1, \bar x_1)\|_{L_{\tilde x_1}^{q'}}}
{W_2(x_1,\ldots,x_{m+1})^{\alpha }}
 \Bigg\|_{L_{\bar x_1}^{p'_1}} .\label{eq:w:e47}
\end{align}
Since $p_{m+1}>1$, by setting
\begin{align*}
y_1 &= \bar x_1,  \\
y_i &= x_{\tau_{\iota_0+i-1}}, \qquad 2 \le i\le m-\iota_0+1 \\
y_{m-\iota_0} &=  (x_{m+1}^{( 1)},\ldots, x_{m+1}^{(n-r_{k_{\nu}})}),\\
y_{m-\iota_0+1} &= (x_{m+1}^{(n-r_{k_{\nu}}+1)},\ldots, x_{m+1}^{(n)}),
 \\
(n_1,\ldots,  n_{m-\iota_0+1})& =(r_{k_{\nu}},n,\ldots, n,
      n-r_{k_{\nu}}, r_{k_{\nu}}), \\
\vec u &= (p_1,p_{\tau_{\iota_0+1}},\ldots, \tau_m, p_{m+1},p_{m+1}),
\end{align*}
we see from  Lemma~\ref{Lm:h} that
for $g\in L^{q'}(\bbR^{r_{k_{\nu}}})$,
\begin{align}
\left\|\frac{g(y_1)}{(|y_1- y_{m-\iota_0+1}| + \sum_{i=2}^{m-\iota_0}|y_i| )^{\alpha}}
\right\|_{L^{\vec u'}}
\lesssim \|g\|_{L^{q'}}. \label{eq:s:c1}
\end{align}
Take $g(\bar x_1)=\|h(\tilde x_1, \bar x_1)\|_{L_{\tilde x_1}^{q'}}$.
We get
\begin{align*}
\Bigg\|
 \frac{\|h(\tilde x_1, \bar x_1)\|_{L_{\tilde x_1}^{q'}}}
{W_2(x_1,\ldots,x_{m+1})^{\alpha }}
 \Bigg\|_{L_{(\bar x_1, x_{\tau_{\iota_0+1}},\ldots, x_{\tau_m},x_{m+1})}^{
 (p'_1, p'_{\tau_{\iota_0+1}},\ldots, p'_{\tau_m},p'_{m+1})}}
\lesssim \|h\|_{L^{q'}}.
\end{align*}
It follows from (\ref{eq:w:e47}) that
\begin{align}
& \left\|\left\|\frac{ h( x_1 ) }
    { W(x_1,\ldots,x_{m+1})^{\lambda}}\right\|_{
    L_{(x_1,x_{\tau_2}, \ldots,x_{\tau_{\iota_0}})}^{
    (p'_1,p'_{\tau_2}, \ldots,p'_{\tau_{\iota_0}})}
    }\right\|_{L_{(  x_{\tau_{\iota_0+1}},\ldots, x_{\tau_m},x_{m+1})}^{
 ( p'_{\tau_{\iota_0+1}},\ldots, p'_{\tau_m},p'_{m+1})}}  \lesssim \|h\|_{L^{q'}}.\nonumber
\end{align}
Hence (\ref{eq:w:e52a}) is true.

(A3)\,\, $\rank(A_{m+1,m+1})=0$.

In this case, we have
$k_{\nu}\le m$. Therefore,
$p_{k_{\nu}}<q$ and  $p_{m+1}\le q<p_1$.

As in Case (A2),
there is a rearrangement
 $\{\tau_2,\ldots,\tau_{k_{\nu}}\}$    of $\{ 2,\ldots,k_{\nu}\}$
such that
$\tau_2<\ldots < \tau_{\iota_0}$, $\tau_{\iota_0+1}<\ldots<\tau_{k_{\nu}}$,
$p_{\tau_i}\le p_1$ for $2\le i\le \iota_0$
and $p_{\tau_i}> p_1$ for $\iota_0+1\le i\le m$.
Then $\tau_{\iota_0}=k_{\nu}$
and $p'_{\tau_i} < p'_1 \le  p'_{\tau_l}$ for $l\le   \iota_0< i\le k_{\nu}$.
Let $\tau_i = i$ for $k_{\nu}+1\le i\le m$.
By Lemma~\ref{Lm:minkowski},
\begin{equation}\label{eq:s:e28}
\left\|\cdot \right\|_{L^{\vec p'}}
 \le
 \left\|\cdot\right\|_{
    L_{(x_1, x_{\tau_2},\ldots,x_{\tau_m},x_{m+1})}^{
    (p'_1,  p'_{\tau_2},\ldots,p'_{\tau_m},  p'_{m+1})}
    }.
\end{equation}

By Lemma~\ref{Lm:ADx-y}, $T_{\lambda}$ is bounded if and only if
\begin{align}
\left\|\frac{ h( x_1 ) }
    { W(x_1,\ldots,x_{m+1})^{\lambda}}\right\|_{L^{\vec p'}}
 \lesssim  \|h\|_{L^{q'}}, \label{eq:w:e52}
\end{align}
where
\[
  W(x_1,\ldots,x_{m+1}) = \sum_{l=1}^n |x_1^{(l)} - x_{i_l}^{(j_l)}|
  + \sum_{\substack{2\le i\le m, 1\le j\le n\\
  (i,j)\ne (i_l,j_l), 1\le l\le n}} |x_i^{(j)}|
  +|x_{m+1}|
\]
and $\{(i_l,j_l):\, 1\le l\le n\}$ is defined by (\ref{eq:il jl}).

Let $s_i$ be defined by (\ref{eq:si}) and $\vec s = (s_{\tau_2},\ldots, s_{\tau_{\iota_0}})$. We have
\begin{align}
& \left\|\frac{ h( x_1 ) }
    { W(x_1,\ldots,x_{m+1})^{\lambda}}\right\|_{
    L_{(x_1,x_{\tau_2},\ldots, x_{\tau_{\iota_0}})}^{
    (p'_1,p'_{\tau_2},\ldots, p'_{\tau_{\iota_0}})}
    } \nonumber\\
&=\left(\sup_{\|f\|_{L^{\vec s}}=1}
\int_{\bbR^{\iota_0n}}\frac{f(x_{\tau_2},\ldots, x_{\tau_{\iota_0}}) |h( x_1 )|^{p'_1} dx_1 dx_{\tau_2} \ldots dx_{\tau_{\iota_0}}  }
    { W(x_1,\ldots,x_{m+1})^{\lambda p'_1}}\right)^{1/p'_1}.
     \label{eq:s:e30}
\end{align}

Using Young's inequality  when computing the integration
 $dx_1^{(l)} dx_{i_l}^{(j_l)}$ for $1\le l\le n$ with
$i_l \le  k_{\nu-1}$
and
H\"older's inequality for others terms, respectively,
we have for $h\in L^{q'}$,
\begin{align}
& \sup_{\|f\|_{L^{\vec s}}=1}
\int_{\bbR^{\iota_0n}}\frac{f(x_{\tau_2},\ldots, x_{\tau_{\iota_0}}) |h( x_1 )|^{p'_1} dx_1 dx_{\tau_2} \ldots dx_{\tau_{\iota_0}}  }
    { W(x_1,\ldots,x_{m+1})^{\lambda p'_1}}
     \nonumber
    \\
&\lesssim
\sup_{\|f\|_{L^{\vec s}}=1}
\int_{\bbR^{r_{k_{\nu}}+n}}\frac{\|f(\ldots, x_{\tau_{\iota_0}}) \|_{L_{(x_{\tau_2},\ldots, x_{\tau_{\iota_0-1}})}^{
(s_{\tau_2},\ldots, s_{\tau_{\iota_0-1}})}}
\|h( \tilde x_1, \bar x_1 )\|_{L_{\tilde x_1}^{q'}}^{p'_1}
d\bar x_1  d x_{\tau_{\iota_0}}  }
    { W_1(x_1,\ldots,x_{m+1})^{\alpha  p'_1}}, \label{eq:s:e29}
\end{align}
where
\begin{align*}
  W_1(x_1,\ldots,x_{m+1})
  &= |\bar x_1 - \bar x_{\tau_{\iota_0}}| + |\tilde x_{\tau_{\iota_0}}|
  + \sum_{  i=\iota_0+1}^{m } |x_{\tau_i} | + |x_{m+1}|, \\
  \tilde x_1 &= (x_1^{(1)}, \ldots, x_1^{(n-r_{k_{\nu}})}) ,\qquad
     \bar x_1 =  (x_1^{(n-r_{k_{\nu}}+1)}, \ldots, x_1^{(n)}) , \\
  \tilde x_{\tau_{\iota_0}} &= (x_{\tau_{\iota_0}}^{(1)}, \ldots, x_{\tau_{\iota_0}}^{(n-r_{k_{\nu}})}) ,\qquad
     \bar x_{\tau_{\iota_0}} =  (x_{\tau_{\iota_0}}^{(n-r_{k_{\nu}}+1)}, \ldots, x_{\tau_{\iota_0}}^{(n)}) , \\
 \alpha &= \lambda - \sum_{i= 2 }^{ \iota_0-1 } \frac{n}{p'_{\tau_i }}
    - \frac{n-r_{k_{\nu}}}{p'_1} +
 \frac{n-r_{k_{\nu}}}{q'}.
\end{align*}
Combining (\ref{eq:s:e30}) and (\ref{eq:s:e29}), we get
\begin{align}
 \left\|\frac{ h( x_1 ) }
    { W(x_1,\ldots,x_{m+1})^{\lambda}}\right\|_{
    L_{(x_1,x_{\tau_2},\ldots, x_{\tau_{\iota_0}})}^{
    (p'_1,p'_{\tau_2},\ldots, p'_{\tau_{\iota_0}})}
    }
&\lesssim
  \left\|\int_{\bbR^{r_{k_{\nu}}}}\frac{ \|h( \tilde x_1, \bar x_1 )\|_{L_{\tilde x_1}^{q'}}^{p'_1} d\bar x_1 }
    { W_1(x_1,\ldots,x_{m+1})^{\alpha p'_1}}\right\|_{
    L_{ x_{\tau_{\iota_0}}}^{s'_{\tau_{\iota_0}}}}^{1/p'_1} \nonumber \\
&=
  \left\| \frac{ \|h( \tilde x_1, \bar x_1 )\|_{L_{\tilde x_1}^{q'}}   }
    { W_1(x_1,\ldots,x_{m+1})^{\alpha }}\right\|_{
    L_{ (\bar x_1, x_{\tau_{\iota_0}})}^{(p'_1, p'_{\tau_{\iota_0}})}}. \nonumber
\end{align}
It follows from (\ref{eq:s:e28}) that
\begin{equation}\label{eq:s:e31}
 \left\|\frac{ h( x_1 ) }
    { W(x_1,\ldots,x_{m+1})^{\lambda}}\right\|_{L^{\vec p'}}
\lesssim
  \left\| \frac{ \|h( \tilde x_1, \bar x_1 )\|_{L_{\tilde x_1}^{q'}}   }
    { W_1(x_1,\ldots,x_{m+1})^{\alpha }}\right\|_{
    L_{ (\bar x_1, x_{\tau_{\iota_0}},\ldots,x_{\tau_m},x_{m+1})}^{(p'_1, p'_{\tau_{\iota_0}},\ldots,p'_{\tau_m},p'_{m+1})}}.
\end{equation}

Choose positive numbers $u_1$ and $u_2$ such that
$p'_1< u_1 <q'<u_2 \le p'_{k_{\nu}}$
and $v(u) := n/(n/p'_{m+1} -r_{k_{\nu}}/q' +r_{k_{\nu}}/u)>1$
for $u=u_1$, $u_2$.

For $g\in L^u$ with $u_1\le u\le u_2$, define the operator $S$ by
\[
  Sg(x_{m+1}) = \left\| \frac{ g(\bar x_1 ) }
    { W_1(x_1,\ldots,x_{m+1})^{\alpha }}\right\|_{
    L_{ (\bar x_1, x_{\tau_{\iota_0}},\ldots,x_{\tau_m})}^{(p'_1, p'_{\tau_{\iota_0}},\ldots,p'_{\tau_m})}} .
\]
Recall that $\tau_{\iota_0}=k_{\nu}$.
Using   Young's inequality when computing
the $L^{s'_{\tau_{\iota_0}}}$ norm with respect to $\bar x_{\tau_{\iota_0}}$
and then computing the  $L^{s'_{\tau_{\iota_0}}}$ norm with respect to $\tilde  x_{\tau_{\iota_0}}$
directly, we get
 \begin{align*}
&  \left\|\int_{\bbR^{r_{k_{\nu}}}}\frac{ |g( \bar x_1 )|^{p'_1} d\bar x_1 }
    { W_1(x_1,\ldots,x_{m+1})^{\alpha p'_1}}\right\|_{
    L_{ x_{\tau_{\iota_0}}}^{s'_{\tau_{\iota_0}}}}^{1/p'_1}
 \\
& \lesssim \frac{\|g\|_{L^u}}{ (
  \sum_{  i=\iota_0+1}^{m } |x_{\tau_i} | + |x_{m+1}|
      )^{\alpha -
    n/p'_{\tau_{\iota_0}} -r_{k_{\nu}}/p'_1 + r_{k_{\nu}}/u  }}.
 \end{align*}
That is,
\[
   \left\| \frac{ g(\bar x_1)}
    { W_1(x_1,\ldots,x_{m+1})^{\alpha }}\!\right\|_{
    L_{ (\bar x_1, x_{\tau_{\iota_0}})}^{(p'_1, p'_{\tau_{\iota_0}})}}
 \!\! \lesssim \!\!\frac{\|g\|_{L^u}}{ (
 \sum_{  i=\iota_0+1}^{m } \!|x_{\tau_i} | \!+\! |x_{m+1}| )^{\alpha -
    n/p'_{\tau_{\iota_0}}\!\! -r_{k_{\nu}}/p'_1 + r_{k_{\nu}}/u  }}.
\]
Now Compute the $L_{x_{\tau_i}}^{p'_{\tau_i}}$ norms successively  for $ \iota_0+1\le i\le m$. We have
\begin{align*}
Sg(x_{m+1})
&\lesssim
\frac{\|g\|_{L^{u}}}
{|x_{m+1}|^{\alpha - n/p'_{\tau_{\iota_0}} -r_{k_{\nu}}/p'_1 + r_{k_{\nu}}/u
-n/p'_{\tau_{\iota_0+1}} - \ldots -n/p'_{\tau_{\iota_m}}}}\\
&= \frac{\|g\|_{L^{u}}}
{|x_{m+1}|^{n/p'_{m+1} -r_{k_{\nu}}/q' +r_{k_{\nu}}/u }}
= \frac{\|g\|_{L^{u}}}
{|x_{m+1}|^{n/v(u)}}.
\end{align*}
Hence $S$  is of weak type $(u, v(u))$
whenever $u_1\le u\le u_2$.
Since  $q\ge p_{m+1}$ and $v(q') = p'_{m+1}$, we see from the interpolation theorem
that
$S$ is of type $(q', p'_{m+1})$.

Take $g (\bar x_1) = \|h( \tilde x_1, \bar x_1 )\|_{L_{\tilde x_1}^{q'}}$.
We have
\[
 \left\|
  \left\| \frac{ \|h( \tilde x_1, \bar x_1 )\|_{L_{\tilde x_1}^{q'}}   }
    { W_1(x_1,\ldots,x_{m+1})^{\alpha }}\right\|_{
    L_{ (\bar x_1, x_{\tau_{\iota_0}},\ldots,x_{\tau_m})}^{(p'_1, p'_{\tau_{\iota_0}},\ldots,p'_{\tau_m})}}
  \right\|_{L_{x_{m+1}}^{p'_{m+1}}}
    \lesssim
     \|h\|_{L^{q'}}.
\]
Now we see from (\ref{eq:s:e31}) that (\ref{eq:w:e52}) is true.

Next we consider the case   $k_0<m+1$, i.e.,
$p_{m+1}=1$.
By Lemmas~\ref{Lm:end points ii} and \ref{Lm:ADx-y}, $T_{\lambda}$ is bounded if and only if
\begin{align}
\left\|\frac{ h( x_1 ) }
    { W(x_1,\ldots,x_{k_0})^{\lambda}}\right\|_{L^{\vec {\tilde p}'}}
 \lesssim  \|h\|_{L^{q'}}, \label{eq:s:e32}
\end{align}
where $\vec {\tilde p} = (p_1, \ldots, p_{k_0})$,
\[
  W(x_1,\ldots,x_{k_0}) = \sum_{l=1}^{n-r_{k_0+1}} |x_1^{(l)} - x_{i_l}^{(j_l)}|
   +
 \sum_{l=n-r_{k_0+1}+1}^n |x_1^{(l)}  |
  + \sum_{\substack{2\le i\le k_0, 1\le j\le n\\
  (i,j)\ne (i_l,j_l), 1\le l\le n}}\!\! |x_i^{(j)}|
\]
and $\{(i_l,j_l):\, 1\le l\le n\}$ is defined by (\ref{eq:il jl}).

There are four subcases.

(B1)\,\, $r_{k_0} = r_{k_0+1}=0$.

In this case, $k_l<k_0$ for $1\le l\le \nu$
and $W(x,y)$
is of the following form,
\[
  W(x_1,\ldots,x_{m+1}) = \sum_{l=1}^n |x_1^{(l)} - x_{i_l}^{(j_l)}|
  + \sum_{\substack{2\le i\le k_0-1, 1\le j\le n\\
  (i,j)\ne (i_l,j_l), 1\le l\le n}} |x_i^{(j)}|
  +|x_{k_0}|
\]
By Setting $m=k_0-1$ in Case (A3), we get (\ref{eq:s:e32}).

(B2)\,\, $r_{k_0} = r_{k_0+1}>0$.

In this case, $k_{\nu}\ge k_0+1$, $p_{k_{\nu}}=1$ and $W(x,y)$ is of the following form,
\begin{align*}
  W(x_1,\ldots,x_{k_0})
  &= \sum_{l=1}^{n-r_{k_0}} |x_1^{(l)} - x_{i_l}^{(j_l)}|
   +
 \sum_{l=n-r_{k_0}+1}^n |x_1^{(l)}  |
   \\
  &\qquad  + \sum_{\substack{2\le i\le k_0-1, 1\le j\le n\\
  (i,j)\ne (i_l,j_l), 1\le l\le n}} |x_i^{(j)}|
  +|x_{k_0}|.
\end{align*}
There is a rearrangement
 $\{\tau_2,\ldots,\tau_{k_0-1}\}$    of
$\{ 2,\ldots,k_0-1\}$ such that
$\tau_2<\ldots < \tau_{\iota_0}$, $\tau_{\iota_0+1}<\ldots<\tau_{k_0-1}$,
$p_{\tau_i}\le p_1$ for $2\le i\le \iota_0$
and $p_{\tau_i}> p_1$ for $\iota_0+1\le i\le k_0-1$.
Then $p'_{\tau_i} < p'_1 \le  p'_{\tau_l}$ for $l\le \iota_0< i$.
By Lemma~\ref{Lm:minkowski}, we have
\begin{align}
\left\|\cdot \right\|_{L^{\vec {\tilde p}'}}
 \le
 \left\|\cdot\right\|_{
    L_{(x_1,x_{\tau_2},\ldots,x_{\tau_{k_0-1}}, x_{k_0})}^{
    (p'_1,p'_{\tau_2},\ldots,p'_{\tau_{k_0-1}}, p'_{k_0})}
    }.   \label{eq:s:e33}
\end{align}
Since $  p_{k_l}\le q<p_1$ for $1\le l\le \nu$,
we have $k_l\in  \{\tau_2,\ldots, \tau_{\iota_0} \}$ if $k_l<k_0$.

Let $s_i$ be defined by (\ref{eq:si}) and $\vec s = (s_{\tau_2},\ldots, s_{\tau_{\iota_0}})$. We have
\begin{align}
& \left\|\frac{ h( x_1 ) }
    { W(x_1,\ldots,x_{k_0})^{\lambda}}\right\|_{
    L_{(x_1,x_{\tau_2},\ldots, x_{\tau_{\iota_0}})}^{
    (p'_1,p'_{\tau_2},\ldots, p'_{\tau_{\iota_0}})}
    }  \nonumber \\
&=\left(\sup_{\|f\|_{L^{\vec s}}=1}
\int_{\bbR^{ \iota_0 n}}\frac{f(x_{\tau_2},\ldots, x_{\tau_{\iota_0}}) |h( x_1 )|^{p'_1} dx_1 dx_{\tau_2} \ldots dx_{\tau_{\iota_0}}  }
    { W(x_1,\ldots,x_{k_0})^{\lambda p'_1}}\right)^{1/p'_1}.
   \label{eq:s:e34}
\end{align}
 Using Young's inequality  when computing the integration
 $dx_1^{(l)} dx_{i_l}^{(j_l)}$ for $1\le l\le n-r_{k_0}$
and
H\"older's inequality for others terms,
we have for $h\in L^{q'}$,
\begin{align}
& \sup_{\|f\|_{L^{\vec s}}=1}
\int_{\bbR^{ \iota_0 n}}\frac{f(x_{\tau_2},\ldots, x_{\tau_{\iota_0}}) |h( x_1 )|^{p'_1} dx_1 dx_{\tau_2} \ldots dx_{\tau_{\iota_0}}  }
    { W(x_1,\ldots,x_{k_0})^{\lambda p'_1}} \nonumber
    \\
&\lesssim
 \int_{\bbR^{r_{k_0}}}\frac{
\|h( \tilde x_1, \bar x_1 )\|_{L_{\tilde x_1}^{q'}}^{p'_1}
d\bar x_1  }
    { W_1(x_1,\ldots,x_{k_0})^{\alpha  p'_1}},  \label{eq:s:e35}
\end{align}
where
\begin{align*}
  \tilde x_1 &= (x_1^{(1)}, \ldots, x_1^{(n-r_{k_0})}),\qquad
     \bar x_1 =  (x_1^{(n-r_{k_0}+1)}, \ldots, x_1^{(n)}), \\
 \alpha &= \lambda - \sum_{i=2}^{\iota_0} \frac{n}{p'_{\tau_i }}
    - \frac{n-r_{k_0}}{p'_1} +
 \frac{n-r_{k_0}}{q'},\\
  W_1(x_1,\ldots,x_{k_0})
  &=  |\bar x_1   |
  + \sum_{ i=\iota_0+1}^{k_0-1}  |x_{\tau_i} | + |x_{k_0}|.
\end{align*}
Combining (\ref{eq:s:e34}) and (\ref{eq:s:e35}), we get
\begin{align}
& \left\|\frac{ h( x_1 ) }
    { W(x_1,\ldots,x_{k_0})^{\lambda}}\right\|_{
    L_{(x_1,x_{\tau_2},\ldots, x_{\tau_{\iota_0}})}^{
    (p'_1,p'_{\tau_2},\ldots, p'_{\tau_{\iota_0}})}
    }  \lesssim
  \left\|\frac{ \|h( \tilde x_1, \bar x_1 )\|_{L_{\tilde x_1}^{q'}}  }
    { W_1(x_1,\ldots,x_{k_0})^{\alpha  }}\right\|_{L_{\bar x_1}^{p'_1}}. \label{eq:s:e36}
\end{align}

Take two positive numbers $u_1$ and $u_2$ such
that $p'_1<u_1 < q'<u_2 $
and
$ v(u) := n/(n/p'_{k_0} -r_{k_0}/q' +r_{k_0}/u) >1$
for $u=u_1$, $u_2$.

For $g\in L^u$ with $u_1\le u\le u_2$, define the operator $S$ by
\[
  Sg(x_{k_0}) =   \left\|\frac{ g( \bar x_1 )  }
    { W_1(x_1,\ldots,x_{k_0})^{\alpha  }}
    \right\|_{L_{(\bar x_1,x_{\tau_{\iota_0+1}},\ldots, x_{\tau_{k_0-1}})}^{
     (p'_1,p'_{\tau_{\iota_0+1}},\ldots, p'_{\tau_{k_0-1}})}}.
\]
By H\"older's inequality,  we have
\[
  \left(\int_{\bbR^{r_{k_0}}}\frac{ |g( \bar x_1 )|^{p'_1} d\bar x_1 }
    { W_1(x_1,\ldots,x_{k_0})^{\alpha p'_1}}\right)^{1/p'_1}
  \lesssim \frac{\|g\|_{L^u}}{ (
 \sum_{ i=\iota_0+1}^{k_0-1}  |x_{\tau_i} | + |x_{k_0}| )^{\alpha
     -r_{k_0}/p'_1 + r_{k_0}/u  }}.
\]
It follows that
\begin{align*}
Sg(x_{k_0})
&\lesssim
\frac{\|g\|_{L^{u}}}
{|x_{k_0}|^{\alpha -  r_{k_0}/p'_1 + r_{k_0}/u
-n/p'_{\tau_{\iota_0+1}} - \ldots -n/p'_{\tau_{k_0-1}}}}\\
&= \frac{\|g\|_{L^{u}}}
{|x_{k_0}|^{n/p'_{k_0} -r_{k_0}/q' +r_{k_0}/u }}
= \frac{\|g\|_{L^{u}}}
{|x_{k_0}|^{n/v(u) }}.
\end{align*}
Hence $S$  is of weak type $(u, v(u))$
whenever  $u_1\le u\le u_2$.
Since   $v(q') = p'_{k_0}$ and $ p_{k_0}\le q $,
 we see from the interpolation theorem
that
$S$ is of type $(q', p'_{k_0})$.

Take $g (\bar x_1)= \|h( \tilde x_1, \bar x_1 )\|_{L_{\tilde x_1}^{q'}}$.
By  (\ref{eq:s:e36}), we have
\[
 \left\| \left\|\frac{ h( x_1 ) }
    { W(x_1,\ldots,x_{k_0})^{\lambda}}\right\|_{
    L_{(x_1,x_{\tau_2},\ldots, x_{\tau_{\iota_0}})}^{
    (p'_1,p'_{\tau_2},\ldots, p'_{\tau_{\iota_0}})}} \right\|_{L_{(x_{\tau_{\iota_0+1}},\ldots, x_{\tau_{k_0-1}},  x_{k_0})}^{
     (p'_{\tau_{\iota_0+1}},\ldots, p'_{\tau_{k_0-1}},p'_{k_0})}}
     \lesssim
  \|h\|_{L^{q'}}.
\]
Now we see from (\ref{eq:s:e33}) that (\ref{eq:s:e32}) is true.

(B3)\,\, $r_{k_0} > r_{k_0+1}=0$.

In this case, $k_{\nu}=k_0$
and $W(x,y)$
is of the following form,
\[
  W(x_1,\ldots,x_{k_0}) = \sum_{\substack{1\le l\le n \\
  i_l < k_0}} |x_1^{(l)} - x_{i_l}^{(j_l)}|
  +\sum_{\substack{1\le l\le n \\
  i_l = k_0}}  |x_1^{(l)} - x_{k_0}^{(j_l)}|
  + \sum_{\substack{2\le i\le k_0, 1\le j\le n\\
  (i,j)\ne (i_l,j_l), 1\le l\le n}} |x_i^{(j)}|.
\]
By Setting $m=k_0-1$ in Case (A1) (for $r_{k_0}=n$) or in Case (A2) (for $0<r_{k_0}<n$), we get (\ref{eq:s:e32}).

(B4)\,\, $r_{k_0} > r_{k_0+1}>0$.

In this case, $k_{\nu}>k_0$,
$p_{k_{\nu}}=1$
and $W(x,y)$
is of the following form,
\begin{align*}
W(x_1,\ldots,x_{k_0})
&=  \sum_{\substack{1\le l\le n \\
  i_l \le  k_0}} |x_1^{(l)} - x_{i_l}^{(j_l)}|
   + \sum_{l=n-r_{k_0+1}+1}^n  |x_1^{(l)} |
  + \!\! \sum_{\substack{2\le i\le k_0, 1\le j\le n\\
  (i,j)\ne (i_l,j_l), 1\le l\le n}} \!\!\!|x_i^{(j)}|
    \\
&=   \sum_{\substack{1\le l\le n \\
  i_l < k_0}} |x_1^{(l)} - x_{i_l}^{(j_l)}|
  +\!\!\sum_{\substack{1\le l\le n \\
  i_l = k_0}} \! |x_1^{(l)} - x_{k_0}^{(j_l)}|
  +\!\! \sum_{\substack{2\le i\le k_0-1, 1\le j\le n\\
  (i,j)\ne (i_l,j_l), 1\le l\le n}}\!\! |x_i^{(j)}|
   \\
& \qquad       +\sum_{l=n-r_{k_0+1}+1}^n  |x_1^{(l)} |
    + \sum_{l=1}^{n-r_{k_0} + r_{k_0+1}} |x_{k_0}^{(l)}|.
\end{align*}

Let the rearrangement
 $\{\tau_2,\ldots,\tau_{k_0-1}\}$    of
$\{ 2,\ldots,k_0-1\}$
be defined as in Case (B2).
Set
\[
  \tilde x_{k_0} = (x_{k_0}^{(n-r_{k_0}+r_{k_0+1}+1)},\ldots,x_{k_0}^{(n)} ),
  \quad
  \bar x_{k_0} = (x_{k_0}^{(1)}, \ldots, x_{k_0}^{(n-r_{k_0}+r_{k_0+1})}  ).
\]
Since $p_{k_0}\le q$, we have $p'_{k_0}\ge q' > p'_1 > p'_{\tau_i}$ for $i\ge \iota_0$. By Lemma~\ref{Lm:minkowski}, we have
\begin{align}
\left\|\cdot \right\|_{L^{\vec {\tilde p}'}}
 \le
 \left\|\cdot\right\|_{
    L_{(x_1,x_{\tau_2},\ldots,x_{\tau_{\iota_0}}, \tilde x_{k_0},x_{\tau_{\iota_0+1}},
     \ldots,  x_{\tau_{k_0-1}}, \bar x_{k_0})}^{
    (p'_1,p'_{\tau_2},\ldots,p'_{\tau_{\iota_0}}, p'_{k_0},p'_{\tau_{\iota_0+1}},
     \ldots, p'_{\tau_{k_0-1}}, p'_{k_0})}
    }.     \label{eq:s:e37}
\end{align}
Note that $  p_{k_l}\le q<p_1$ for $1\le l\le \nu$.
We have $k_l\in  \{\tau_2, \ldots,\tau_{\iota_0} \}$ if $k_l<k_0$.

Let $s_i$ be defined by (\ref{eq:si}) and $\vec s = (s_{\tau_2},\ldots, s_{\tau_{\iota_0}},s_{k_0})$. We have
\begin{align}
& \left\|\frac{ h( x_1 ) }
    { W(x_1,\ldots,x_{k_0})^{\lambda}}\right\|_{
    L_{(x_1,x_{\tau_2},\ldots, x_{\tau_{\iota_0}}, \tilde x_{k_0})}^{
    (p'_1,p'_{\tau_2},\ldots, p'_{\tau_{\iota_0}},p'_{k_0})}
    }^{p'_1} \nonumber\\
&= \sup_{\|f\|_{L^{\vec s}}=1}
\int_{\bbR^{ \iota_0 n + r_{k_0} - r_{k_0+1}}}\frac{f(x_{\tau_2},\ldots, x_{\tau_{\iota_0}}, \tilde x_{k_0}) |h( x_1 )|^{p'_1} dx_1 dx_{\tau_2} \ldots  dx_{\tau_{\iota_0}} d\tilde x_{k_0} }
    { W(x_1,\ldots,x_{k_0})^{\lambda p'_1}}  \nonumber \\
& \lesssim
  \int_{\bbR^{ r_{k_0+1}}}\frac{ \|  h( \tilde x_1, \bar x_1 )\|_{L_{\tilde x_1}^{q'}}^{p'_1} d\bar x_1 }
    { W_1(x_1,\ldots,x_{k_0})^{\alpha p'_1}},
     \label{eq:s:e40}
\end{align}
where we use  Young's inequality  when computing the integration
 $dx_1^{(l)} dx_{i_l}^{(j_l)}$ for $1\le l\le n$ with $i_l\le k_0$
and
H\"older's inequality for others terms
in the last step and
\begin{align*}
  \tilde x_1 &= (x_1^{(1)}, \ldots, x_1^{(n-r_{k_0+1})}),\\
      \bar x_1 &=  ( x_1^{(n-r_{k_0+1}+1)},  \ldots, x_1^{(n)}), \\
   \alpha &= \lambda - \sum_{i= 2  }^{ \iota_0 } \frac{n}{p'_{\tau_i }}
    - \frac{r_{k_0}-r_{k_0+1}}{p'_{k_0}}
    - \frac{n- r_{k_0+1}}{p'_1} +
 \frac{n- r_{k_0+1}}{q'},\\
  W_1(x_1,\ldots,x_{k_0})
  &= |\bar x_1|
      + \sum_{i=\iota_0+1 }^{k_0-1} |x_{\tau_i} |+    |\bar x_{k_0} |.
\end{align*}
We rewrite (\ref{eq:s:e40}) as
\begin{align}
  \left\|\frac{ h( x_1 ) }
    { W(x_1,\ldots,x_{k_0})^{\lambda}}\right\|_{
    L_{(x_1,x_{\tau_2},\ldots, x_{\tau_{\iota_0}}, \tilde x_{k_0})}^{
    (p'_1,p'_{\tau_2},\ldots, p'_{\tau_{\iota_0}},p'_{k_0})}
    }
& \lesssim
  \left\|\frac{ \|  h( \tilde x_1, \bar x_1 )\|_{L_{\tilde x_1}^{q'}}  }
    { W_1(x_1,\ldots,x_{k_0})^{\alpha }}\right\|_{L_{\bar x_1}^{p'_1}}. \label{eq:s:e40a}
\end{align}

For $u\ge 1$, define the function $v(u)$ by
\[
\frac{n - (r_{k_0}- r_{k_0+1})}{v(u)}
=  \frac{n - (r_{k_0}- r_{k_0+1})}{p'_{k_0}}
   - \frac{r_{k_0+1}}{q'} + \frac{r_{k_0+1}}{u}.
\]
Take two positive numbers $u_1$ and $u_2$
such that $p'_1 < u_1<q'<u_2$ and  $v(u)>1$ for $u=u_1, u_2$.

For $g\in L^u(\bbR^{r_{k_0+1}})$ with $u_1\le u\le u_2$, define the operator $S$
by
\[
  Sg(\bar x_{k_0}) = \left\|\frac{ g(\bar x_1) }
    { W_1(x_1,\ldots,x_{k_0})^{\alpha }}\right\|_{L_{(\bar x_1, x_{\tau_{\iota_0+1}},
    \ldots,x_{\tau_{k_0-1}} )}^{(p'_1, p'_{\tau_{\iota_0+1}},
    \ldots,p'_{\tau_{k_0-1}})}}.
\]
Computing the norm with respect to $\bar x_1$ by H\"older's inequality
and computing the norms with respect to $ x_{\tau_{\iota_0+1}},
    \ldots,x_{\tau_{k_0-1}}$ directly, we get
\begin{align*}
Sg(\bar x_{k_0})
\lesssim \frac{\|g\|_{L^u}}{|\bar x_{k_0}|^{\beta}},
\end{align*}
where
\begin{align*}
\beta &= \alpha - \frac{r_{k_0+1}}{p'_1} + \frac{r_{k_0+1}}{u} - \sum_{i=\iota_0+1}^{k_0-1} \frac{n}{p'_{\tau_i}}  \\
&= \frac{n - (r_{k_0}- r_{k_0+1})}{p'_{k_0}}
   - \frac{r_{k_0+1}}{q'} + \frac{r_{k_0+1}}{u}
  = \frac{n - (r_{k_0}- r_{k_0+1})}{v(u)}.
\end{align*}
Hence $S$ is of weak type $(u, v(u))$ for $u=u_1, u_2$.
Since $u_1<q'<u_2$ and $v(q') = p'_{k_0} \ge q'$,
by the interpolation theorem,
  $S$ is of type $(q', p'_{k_0})$.
That is,
\[
\left\|\frac{ g(\bar x_1) }
    { W_1(x_1,\ldots,x_{k_0})^{\alpha }}\right\|_{L_{(\bar x_1, x_{\tau_{\iota_0+1}},
    \ldots,x_{\tau_{k_0-1}}, \bar x_{k_0})}^{(p'_1, p'_{\tau_{\iota_0+1}},
    \ldots,p'_{\tau_{k_0-1}}, p'_{k_0})}}
 = \|Sg\|_{L_{\bar x_{k_0}}^{p'_{k_0}}}
 \lesssim \|g\|_{L^{q'}}.
\]
By setting $g(\bar x_1) = \|  h( \tilde x_1, \bar x_1 )\|_{L_{\tilde x_1}^{q'}}  $,
we get
\[
   \left\| \frac{\| h( \tilde x_1, \bar x_1 )\|_{L_{\tilde x_1}^{q'}}} {  W(x_1, \ldots, x_{k_0})^{\alpha}}\right\|_{L_{(\bar x_1, x_{\tau_{\iota_0+1}},
    \ldots, x_{\tau_{k_0-1}},\bar x_{k_0})}^{(p'_1, p'_{\tau_{\iota_0+1}},
    \ldots, p'_{\tau_{k_0-1}},p'_{k_0})}}
  \lesssim \|h\|_{L^{q'}}.
\]
Now we see from (\ref{eq:s:e37}) and (\ref{eq:s:e40a}) that (\ref{eq:s:e32}) is true.

\subsubsection{The Case
 \texorpdfstring{$p_{k_{\nu}}=q$}{p k nu = q}}

Since $\min\{p_{k_l}:\, 1\le l\le \nu\}<q$, there is some $1\le \mu\le \nu-1$ such that
$p_{k_{\mu}} <p_{k_{\mu+1}} = \ldots = p_{k_{\nu}} = q$.

First, we consider the case   $k_0={m+1}$. There are two subcases.

(C1)\,\, $0<\rank(A_{m+1,m+1})<n$.

In this case, $k_{\nu}=m+1$.

Let the rearrangement
 $\{\tau_2,\ldots,\tau_m\}$    of
$\{ 2,\ldots,m\}$ be defined as in Case (A2)(b).
Then
we have $\{k_1,\ldots, k_{\nu-1}\}\subset \{\tau_2,\ldots, \tau_{\iota_0}\}$
and (\ref{eq:s:e28}) is true.

By Lemma~\ref{Lm:ADx-y}, $T_{\lambda}$ is bounded if and only if
\begin{align}
\left\|\frac{ h( x_1 ) }
    { W(x_1,\ldots,x_{m+1})^{\lambda}}\right\|_{L^{\vec p'}}
 \lesssim  \|h\|_{L^{q'}}, \label{eq:s:e43}
\end{align}
where
\[
  W(x_1,\ldots,x_{m+1}) = \sum_{l=1}^n |x_1^{(l)} - x_{i_l}^{(j_l)}|
  + \sum_{\substack{2\le i\le m+1, 1\le j\le n\\
  (i,j)\ne (i_l,j_l), 1\le l\le n}} |x_i^{(j)}|
\]
and $\{(i_l,j_l):\, 1\le l\le n\}$ is defined by (\ref{eq:il jl}).

Let $s_i$ be defined by (\ref{eq:si}) and $\vec s = (s_{\tau_2},\ldots, s_{\tau_{\iota_0}})$.
 We have
\begin{align}
& \left\|\frac{ h( x_1 ) }
    { W(x_1,\ldots,x_{m+1})^{\lambda}}\right\|_{
    L_{(x_1,x_{\tau_2},\ldots, x_{\tau_{\iota_0}})}^{
    (p'_1,p'_{\tau_2},\ldots, p'_{\tau_{\iota_0}})}
    } \nonumber\\
&=\left(\sup_{\|f\|_{L^{\vec s}}=1}
\int_{\bbR^{ \iota_0 n}}\frac{f(x_{\tau_2},\ldots, x_{\tau_{\iota_0}}) |h( x_1 )|^{p'_1} dx_1 dx_{\tau_2} \ldots  dx_{\tau_{\iota_0}}  }
    { W(x_1,\ldots,x_{m+1})^{\lambda p'_1}}\right)^{1/p'_1}.
    \label{eq:s:e45}
\end{align}

Assume that $\tau_d = k_{\mu}$ for some  $2\le d\le  \iota_0 $.
Using Young's inequality  when computing the integration
 $dx_1^{(l)} dx_{i_l}^{(j_l)}$ for $1\le l\le n$ with
$i_l \le  k_{\mu-1}$
and
H\"older's inequality for others terms, respectively,
we have for $h\in L^{q'}$,
\begin{align}
&
\int_{\bbR^{ \iota_0 n}}\frac{f(x_{\tau_2},\ldots, x_{\tau_{\iota_0}}) |h( x_1 )|^{p'_1} dx_1 dx_{\tau_2} \ldots  dx_{\tau_{\iota_0}}  }
    { W(x_1,\ldots,x_{m+1})^{\lambda p'_1}}
    \nonumber
    \\
&\lesssim
 \int_{\bbR^{r_{k_{\mu}}+(\iota_0-d+1)n}} \frac{\|f(x_{\tau_2},\ldots, x_{\tau_{\iota_0}}\!)     \|_{L_{(x_{\tau_2},\ldots, x_{\tau_{d-1}})}^{
(s_{\tau_2},\ldots, s_{\tau_{d-1}})}}\!
  }
    { W_1(x_1,\ldots,x_{m+1})^{\alpha  p'_1}} \nonumber \\
&\qquad \qquad \times \|h( \tilde x_1, \bar x_1, \bar{\bar x}_1 )\|_{L_{\tilde x_1}^{q'}}^{p'_1}
 d\bar x_1d\bar {\bar x}_1 dx_{\tau_d}\ldots d x_{\tau_{\iota_0}}
     , \label{eq:s:e46}
\end{align}
where
\begin{align*}
  \tilde x_1 &= (x_1^{(1)}, \ldots, x_1^{(n-r_{k_{\mu}})}),\\
     \bar x_1 &=  (x_1^{(n-r_{k_{\mu}}+1)}, \ldots, x_1^{(n-r_{k_{\mu+1}})}), \\
     \bar {\bar x}_1 &=  (x_1^{(n-r_{k_{\mu+1}}+1)}, \ldots, x_1^{(n)}), \\
 \alpha &= \lambda - \sum_{i= 2 }^{ d-1 } \frac{n}{p'_{\tau_i }}
    - \frac{n-r_{k_{\mu}}}{p'_1} +
 \frac{n-r_{k_{\mu}}}{q'},\\
  W_1(x_1,\ldots,x_{m+1})
  &= \sum_{\substack{1\le l\le n \\  i_l > k_{\mu}}} |x_1^{(l)} - x_{i_l}^{(j_l)}|
  + \sum_{\substack{1\le l\le n \\  i_l = k_{\mu}}} |x_1^{(l)} - x_{i_l}^{(j_l)}|\\
  &\qquad
  + \sum_{\substack{2\le i\le m+1, 1\le j\le n\\
   i\not\in\{\tau_2,\ldots,\tau_{d-1}\}\\
  (i,j)\ne (i_l,j_l), 1\le l\le n}} |x_i^{(j)}|.
\end{align*}
Denote
\begin{align*}
  y&=(x_{k_{\mu+1}}^{(n-r_{k_{\mu+1}} + r_{k_{\mu+2}}+1)},\ldots,
   x_{k_{\mu+1}}^{(n)}, \ldots, x_{k_{\nu}}^{(n-r_{k_{\nu}}+1)},\ldots,x_{k_{\nu}}^{(n)})\in \bbR^{r_{k_{\mu+1}}}.
\end{align*}
We rewrite $W_1$ as
\[
 W_1(x_1,\ldots,x_{m+1})
   = |\bar{\bar x}_1-y|
  + \sum_{\substack{1\le l\le n \\  i_l = k_{\mu}}} |x_1^{(l)} - x_{i_l}^{(j_l)}|
  + \sum_{\substack{2\le i\le m+1, 1\le j\le n\\
   i\not\in\{\tau_2,\ldots,\tau_{d-1}\}\\
  (i,j)\ne (i_l,j_l), 1\le l\le n}} |x_i^{(j)}|.
\]

Set $h_1(\bar x_1, \bar{\bar x}_1 )
=\|h( \tilde x_1, \bar x_1, \bar{\bar x}_1 )\|_{L_{\tilde x_1}^{q'}}^{p'_1}$.
It follows from (\ref{eq:s:e45}) and (\ref{eq:s:e46}) that
\begin{align}
& \left\|\frac{ h( x_1 ) }
    { W(x_1,\ldots,x_{m+1})^{\lambda}}\right\|_{
    L_{(x_1,x_{\tau_{2}},\ldots, x_{\tau_{\iota_0}})}^{
    (p'_1,p'_{\tau_{2}},\ldots, p'_{\tau_{\iota_0}})}
    } \nonumber\\
&
 \lesssim\left\|
\int_{\bbR^{r_{k_{\mu}}}}
\frac{h_1(\bar x_1, \bar{\bar x}_1 )
d\bar x_1d\bar {\bar x}_1 }
    { W_1(x_1,\ldots,x_{m+1})^{\alpha  p'_1}} \right\|_{L_{(x_{\tau_d},\ldots,x_{\tau_{\iota_0}})}^{(s'_{\tau_d},\ldots,
      s'_{\tau_{\iota_0}})}}^{1/p'_1}. \label{eq:s:e47}
\end{align}
Note that the integration with respect to $\bar{\bar x}$ in the above inequality
can be considered as the convolution of $h_1(\bar x_1, \cdot  )$ and
$|\cdot| + \sum_{\substack{1\le l\le n \\     i_l =k_{\mu}}} |x_1^{(l)} - x_{i_l}^{(j_l)}|
  + \sum_{\substack{2\le i\le m+1, 1\le j\le n\\
   i\not\in\{\tau_2,\ldots,\tau_{d-1}\}\\
  (i,j)\ne (i_l,j_l), 1\le l\le n}} |x_i^{(j)}|$.
Hence
\[
 \int_{\bbR^{r_{k_{\mu+1}}}}
\frac{
  h_1(\bar x_1, \bar{\bar x}_1)
 d\bar {\bar x}_1 }
    { W_1(x_1,\ldots,x_{m+1})^{\alpha  p'_1}}
    \lesssim \frac{M_2 h_1(\bar x_1, y)  }{ W_2(x_1,\ldots,x_{m+1})^{\alpha  p'_1 - r_{k_{\mu+1}}}},
\]
where $M_2$ means the maximal function with respect to the second variable,
i.e.,
\[
   M_2 h_1(\bar x_1, y)
    = (M_2 h_1(\bar x_1, \cdot))(y),
\]
and
\begin{align*}
 W_2(x_1,\ldots,x_{m+1})
  &=\sum_{\substack{1\le l\le n \\    i_l =k_{\mu} }}  |x_1^{(l)} - x_{i_l}^{(j_l)}|
  + \sum_{\substack{2\le i\le m+1, 1\le j\le n\\
   i\not\in\{\tau_2,\ldots,\tau_{d-1}\}\\
  (i,j)\ne (i_l,j_l), 1\le l\le n}} |x_i^{(j)}|.
\end{align*}
By (\ref{eq:s:e47}), we have
\begin{align}
& \left\|\frac{ h( x_1 ) }
    { W(x_1,\ldots,x_{m+1})^{\lambda}}\right\|_{
    L_{(x_1,x_{\tau_2},\ldots, x_{\tau_{\iota_0}})}^{
    (p'_1,p'_{\tau_2},\ldots, p'_{\tau_{\iota_0}})}
    } \nonumber \\
& \lesssim\left\|
\int_{\bbR^{r_{k_{\mu}} - r_{k_{\mu+1}}}}
\frac{M_2 h_1(\bar x_1, y)
d\bar x_1  }
    { W_2(x_1,\ldots,x_{m+1})^{\alpha  p'_1-r_{k_{\mu+1}}}} \right\|_{L_{(x_{\tau_d},\ldots,x_{\tau_{\iota_0}})}^{(s'_{\tau_d},\ldots,
      s'_{\tau_{\iota_0}})}}^{1/p'_1} \nonumber \\
& =\left\|
 \frac{(M_2 h_1(\bar x_1, y))^{1/p'_1}
  }
    { W_2(x_1,\ldots,x_{m+1})^{\alpha -r_{k_{\mu+1}}/ p'_1}} \right\|_{L_{(\bar x_1, x_{\tau_d},\ldots,x_{\tau_{\iota_0}})}^{(p'_1, p'_{\tau_d},\ldots,
      p'_{\tau_{\iota_0}})}}.  \label{eq:s:e48}
\end{align}
It follows from (\ref{eq:s:e28}) that
\begin{align}
& \left\|\frac{ h( x_1 ) }
    { W(x_1,\ldots,x_{m+1})^{\lambda}}\right\|_{L^{\vec p'}}
\!\! \lesssim
   \! \left\|
 \frac{(M_2 h_1(\bar x_1, y))^{1/p'_1}
  }
    { W_2(x_1,\ldots,x_{m+1})^{\alpha -r_{k_{\mu+1}}/ p'_1}} \right\|_{L_{(\bar x_1, x_{\tau_d},\ldots,x_{\tau_m}, x_{m+1})}^{(p'_1, p'_{\tau_d}, \ldots, p'_{\tau_m}, p'_{m+1})}}\!. \label{eq:s:e49}
\end{align}

Take two numbers $u_1, u_2$ such that
$p'_1 < u_1<q'<u_2\le p'_{k_{\mu}}$.
For $u_1\le u \le u_2$ and $g(\bar x_1, y)\in L_y^{q'}(L_{\bar x_1}^u)$,
define the operator $S$ by
\[
  Sg(\bar x_1, x_{\tau_d},\ldots,x_{\tau_m},
    x_{m+1} )
  = \frac{g(\bar x_1, y)}{ W_2(x_1,\ldots,x_{m+1})^{\alpha -r_{k_{\mu+1}}/ p'_1}}.
\]
Recall that $k_{\nu}=m+1$ and $\tau_d = k_{\mu}$. $W_2(x,y)$ is of the following form,
\[
W_2(x_1,\ldots,x_{m+1})
  = |\bar x_1 - \bar x_{\tau_d}|
   + W_3(x_2,\ldots,x_{m+1}),
\]
where
\begin{align*}
 W_3(x_2,\ldots,x_{m+1})&= \sum_{\substack{2\le i\le m, 1\le j\le n\\
   i\not\in\{\tau_2,\ldots,\tau_{d-1}\}\\
  (i,j)\ne (i_l,j_l), 1\le l\le n}} |x_i^{(j)}|
  + \sum_{j=1}^{n-r_{k_{\nu}}} |x_{m+1}^{(j)}|,\\
\bar x_{\tau_d} &= (x_{\tau_d}^{(n-r_{\tau_d} + r_{\tau_d+1}+1)},
\ldots, x_{\tau_d}^{(n)}), \\
\tilde x_{\tau_d} &= ( x_{\tau_d}^{(1)}, \ldots, x_{\tau_d}^{(n-r_{\tau_d} + r_{\tau_d+1} )}).
\end{align*}
Denote
\begin{align*}
 \tilde x_{m+1} &= (x_{m+1}^{(r_{k_{\mu}}-r_{k_{\mu+1}}+1)},\ldots,  x_{m+1}^{(n)}),\\
 \bar x_{m+1} &= (x_{m+1}^{(1)},\ldots, x_{m+1}^{(r_{k_{\mu}}-r_{k_{\mu+1}})}).
\end{align*}
Since $n-r_{k_{\nu}} \ge r_{k_{\mu}}-r_{k_{\mu+1}}$,
all   entries of $\bar x_{m+1}$ are contained in $\{ x_{m+1}^{(1)}$, $\ldots$, $x_{m+1}^{(n-r_{k_{\nu}})}\}$.
We see from Young's inequality that
\begin{align*}
\|Sg\|_{L_{(\bar x_1, \bar x_{\tau_d})}^{(p'_1, p'_{\tau_d})}}
   &  =
   \left\|
\int_{\bbR^{r_{k_{\mu}} - r_{k_{\mu+1}}}}
\frac{ g(\bar x_1, y)^{p'_1}
d\bar x_1  }
    { W_2(x_1,\ldots,x_{m+1})^{\alpha  p'_1-r_{k_{\mu+1}}}} \right\|_{L_{\bar x_{\tau_d}}^{s'_{\tau_d}}}^{1/p'_1} \\
&\lesssim
   \frac{\|g(\bar x_1, y)\|_{L_{\bar x_1}^{u}}}
   {W_3(x_2,\ldots,x_{m+1})^{\alpha - r_{k_{\mu}}/p'_1
   -(r_{k_{\mu}} - r_{k_{\mu+1}})/p'_{\tau_d}
   +(r_{k_{\mu}} - r_{k_{\mu+1}})/u
   }}.
\end{align*}
Recall that $p_{k_{\mu+1}} = \ldots = p_{k_{\nu}}=q$.
Computing the norm with respect to
$\tilde x_{\tau_d}$, $x_{\tau_{d+1}}$, $\ldots$, $x_{\tau_m}$,
  $ \tilde  x_{m+1}$ directly,
we get
\begin{align*}
\|Sg\|_{L_{(\bar x_1, \bar x_{\tau_d},\tilde x_{\tau_d},x_{\tau_{d+1}},\ldots, x_{\tau_m},
   \tilde  x_{m+1})}^{(p'_1, p'_{\tau_d},p'_{\tau_d},p'_{\tau_{d+1}},\ldots, p'_{\tau_m},
   p'_{m+1})}}
\lesssim
 \frac{\|g\|_{L^{q'}(L^u)}}
 {|\bar x_{m+1}|^{\beta}},
\end{align*}
where
\begin{align*}
\beta
&=\alpha - \frac{r_{k_{\mu}}}{p'_1}
   -\frac{ r_{k_{\mu}} - r_{k_{\mu+1}}}{p'_{\tau_d}}
   +\frac{ r_{k_{\mu}} - r_{k_{\mu+1}}}{u}
   -\frac{ n-(r_{k_{\mu}} - r_{k_{\mu+1}})}{p'_{\tau_d}}\\
&\qquad
   -\sum_{i=d+1}^m \frac{n-(r_{\tau_i} - r_{\tau_i+1})}{p'_{\tau_i}}
   -\sum_{i=2}^{\iota_0}  \frac{n }{p'_{\tau_i}}
   -\frac{n-r_{k_{\nu}} - (r_{k_{\mu}} - r_{k_{\mu+1}})}{p'_{m+1}}
   \\
&= \frac{r_{k_{\nu}}+r_{k_{\mu}} - r_{k_{\mu+1}}}{p'_{m+1}}
    + \frac{ r_{k_{\mu}} - r_{k_{\mu+1}}}{u}
    -\frac{r_{k_{\mu}}}{q'} + \sum_{i=d+1}^m \frac{r_{\tau_i} - r_{\tau_i+1}}{p'_{\tau_i}}\\
&=
      \frac{ r_{k_{\mu}} - r_{k_{\mu+1}}}{u},
\end{align*}
and we use the fact that
$p_{m+1} = p_{k_{\nu}}=q$,
$\tau_d = k_{\mu}$,
$\{k_{\mu+1}, \ldots, k_{\nu-1}\} \!\subset \!\{\tau_{d+1}, \ldots, \tau_{\iota_0}\}$,
and $r_{\tau_i} = r_{\tau_i+1}$ provided $\tau_i\not\in \{k_{\mu+1}, \ldots, k_{\nu-1}\}$.
Moreover, if $\tau_i = k_l$ for some $\mu+1\le l\le \nu-1$, then
$p_{\tau_i} = q$
and $r_{\tau_i} - r_{\tau_i+1}
 = r_{k_l} - r_{k_{l+1}}$.

It follows that $S$ is bounded from
$L^{q'}(L^u)$ to $L_{\bar x_{m+1}}^{u,\infty}(L_{(\bar x_1,  x_{\tau_d}, \ldots, x_{\tau_m},
   \tilde  x_{m+1})}^{(p'_1,  p'_{\tau_d},\ldots, p'_{\tau_m},
   p'_{m+1})})$
for $u\in [u_1,u_2]$.
Since  $u_1 < q'=p'_{m+1}<u_2$,
  we see from the interpolation theorem (Lemma~\ref{Lm:interpolation})
that
$S$ is bounded from
$L^{q'}(L^{q'}) $ to $ L_{(\bar x_1,  x_{\tau_d}, \ldots, x_{\tau_m},
     x_{m+1})}^{(p'_1,  p'_{\tau_d},\ldots, p'_{\tau_m},
   p'_{m+1})}$.

Take $g(\bar x_1,  y)
=(M_2 h_1(\bar x_1, y))^{1/p'_1}$.
We see from (\ref{eq:s:e49}) that
(\ref{eq:s:e43}) is true.

(C2)\,\, $ \rank(A_{m+1,m+1})=0$.

In this case, $k_{\nu}\le m$ and $q\ge p_{m+1}$.
The proof is almost the same as that for Case (C1).
The main difference is that
$\{k_1,\ldots, k_{\nu}\}\subset \{\tau_2,\ldots, \tau_m\}$.
In this case,
$W_3$ is of the form
\begin{align*}
 W_3(x_2,\ldots,x_{m+1})&= \sum_{\substack{2\le i\le m, 1\le j\le n\\
   i\not\in\{\tau_2,\ldots,\tau_{d-1}\}\\
  (i,j)\ne (i_l,j_l), 1\le l\le n}} |x_i^{(j)}|
  + |x_{m+1}|.
\end{align*}
Consequently, we need not to split $x_{m+1}$ into two parts. We have
\begin{align*}
\|Sg\|_{L_{(\bar x_1, \bar x_{\tau_d},\tilde x_{\tau_d},x_{\tau_{d+1}},\ldots, x_{\tau_m} )}^{(p'_1, p'_{\tau_d},p'_{\tau_d},p'_{\tau_{d+1}},\ldots, p'_{\tau_m}  )}}
\lesssim
 \frac{\|g\|_{L^{q'}(L^u)}}
 {|  x_{m+1}|^{n/v(u)}},
\end{align*}
where $v(u)$ satisfies that
\begin{align*}
\frac{n}{v(u)}
&= \frac{n}{p'_{m+1}}
    + \frac{ r_{k_{\mu}} - r_{k_{\mu+1}}}{u}
    -\frac{r_{k_{\mu}}}{q'} + \sum_{i=d+1}^m \frac{r_{\tau_i} - r_{\tau_i+1}}{p'_{\tau_i}}\\
&= \frac{n }{p'_{m+1}}
    + \frac{ r_{k_{\mu}} - r_{k_{\mu+1}}}{u} -\frac{ r_{k_{\mu}} - r_{k_{\mu+1}}}{q'}.
\end{align*}
Choose two positive numbers $u_1$, $u_2$ such that
$p'_1< u_1<q'<u_2\le  p'_{k_{\mu}}$ and
$v(u_i)>1$ for $i=1,2$.
Again, we get the desired conclusion by the interpolation theorem.

Next we consider the case   $k_0<m+1$.
Since $p_k=1$ for $k>k_0$ and  $p_{k_{\nu}}=q>1$,
we have $k_l\le k_0$ for all $1\le l\le \nu$.
Consequently, the function
$W$ is of the following form,
\[
  W(x_1,\ldots,x_{m+1})
  =  \sum_{l=1}^n |x_1^{(l)} - x_{i_l}^{(j_l)}|
  + \sum_{\substack{2\le i\le k_0, 1\le j\le n\\
  (i,j)\ne (i_l,j_l), 1\le l\le n}} |x_i^{(j)}|,
\]
where
$ \{(i_l,j_l):\, 1\le l\le n\}
  =\bigcup_{1\le s\le \nu}\{(k_s, t):\, n+1-(r_{k_s}-r_{k_{s+1}})\le t\le n\}$,
thanks to Lemma~\ref{Lm:end points ii} together with Lemma~\ref{Lm:ADx-y}.
Therefore, the arguments in Cases (C1) and (C2) work if we replace $m+1$ by $k_0$.
This completes the proof.

\section{The Case   \texorpdfstring{$m=1$}{m=1}}

In this section, we show that the hypothesis that
the $(1,m+1)$ minor of $A$ is invertible is also necessary when $m=1$.
In other words, for the case   $m=1$, we get  necessary and sufficient conditions on $A$, $(p_1,p_2)$ and $q$ such that
  $T_{\lambda}$ is bounded from $L^{\vec p}$ to $L^q$.

\begin{Theorem}\label{thm:m=1}
Suppose that $1\le p_1, p_2\le \infty$,
 $0<\lambda<n$
 and $0<q\le \infty$  such that
\[
  \frac{1}{p_1} + \frac{1}{p_2} = \frac{1}{q} + \frac{n-\lambda}{n}.
\]
Let $T_{\lambda}f$ be defined by (\ref{eq:Tf}) with $m=1$,
where $A = \begin{pmatrix}
A_{11} & A_{12} \\
A_{21} & A_{22}
\end{pmatrix}$ 
and $A_{ij}$ are $n\times n$ matrices, $1\le i,j\le 2$.

Then $T_{\lambda}$ is bounded from $L^{\vec p}(\bbR^{2n})$ to $L^q(\bbR^n)$
if and only if $1<p_2<q<p_1\le \infty$ and $A$, $A_{21}$, $A_{22}$ are invertible matrices.
\end{Theorem}

\begin{proof}
Denote $A^{-1} = \begin{pmatrix}
\tilde A_{11} & \tilde A_{12}\\
\tilde A_{21} & \tilde A_{22}
\end{pmatrix}$.
By (\ref{eq:h}), $T_{\lambda}$ is bounded if and only if for
any $h\in L^{q'}$,
\begin{equation}\label{eq:ss:e12a}
 \left\|  \int_{\bbR^n} \frac{ |h(\tilde A_{21}x_1 + \tilde A_{22} x_2)|^{p'_1} }
    { |\tilde A_{11}x_1 + \tilde A_{12} x_2|^{\lambda p'_1 }} dx_1\right\|_{L_{x_2}^{p'_2/p'_1}}
     \lesssim \|h\|_{L^{q'}}^{p'_1}.
\end{equation}

By Theorem~\ref{thm:main}, we only need to prove the necessity,
for which  it suffices to show that $A_{21}$ is invertible.
Assume on the contrary that $  A_{21}$ is singular.
By Lemma~\ref{Lm:La}, $ \tilde  A_{21}$ is singular.
There are two cases.

 (i)\,\, $\tilde A_{11}$ is invertible.

In this case, we conclude that  $p_1>1$.
Assume on the contrary that   $p_1=1$. Then we have
\[
  \frac{1}{p_2} = \frac{1}{q} - \frac{\lambda}{n}
  <\frac{1}{q}.
\]
Consequently, $q<p_2$.
If $A_{22}\ne 0$, we see from Theorem~\ref{thm:main} that $q\ge p_2$, which is a contradiction.
If $A_{22}=0$, then
$A_{12}$ and $A_{21}$ are invertible.
Consequently, $\tilde A_{21}$ is non-singular,
which is a contradiction.

Since $p_1>1$, we have $p'_1<\infty$ and
\begin{align*}
& \int_{\bbR^n} \frac{|h(\tilde A_{21} x_1 + \tilde A_{22} x_2)|^{p'_1}}
   {|\tilde A_{11} x_1 + \tilde A_{12} x_2|^{\lambda p'_1}} dx_1
 \approx \int_{\bbR^n} \frac{|h(\tilde A_{21} x_1 + \tilde A_{22} x_2)|^{p'_1}}
   {|x_1 + \tilde A_{11}^{-1} \tilde A_{12} x_2|^{\lambda p'_1}} dx_1   \\
&  \approx
\int_{\bbR^n} \frac{|h(\tilde A_{21} x_1 + (\tilde A_{22} - \tilde A_{21} \tilde A_{11}^{-1} \tilde A_{12}) x_2)|^{p'_1}}
   {|x_1  |^{\lambda p'_1} }dx_1 .
\end{align*}
Hence (\ref{eq:ss:e12a}) is equivalent to
\[
\left\|
  \int_{\bbR^n} \frac{|h(\tilde A_{21} x_1 + (\tilde A_{22} - \tilde A_{21} \tilde A_{11}^{-1} \tilde A_{12}) x_2)|^{p'_1}}
   {|x_1  |^{\lambda p'_1} }dx_1
\right\|_{L_{x_2}^{p'_2/p'_1}} \lesssim \| |h|^{p'_1} \|_{q'/p'_1}.
\]

Since $\tilde A_{21}$ is singular, there exists some integer $r<n$, positive numbers $\lambda_1$,
$\ldots$, $\lambda_r$ and orthonormal matrices $U$ and $V$ such that
$\tilde A_{21} = U \begin{pmatrix} \Lambda &  \\   & 0
\end{pmatrix} V$.
Replacing $h$ and $x_1$ by $h(U^{-1}\cdot)$ and $V^{-1}x_1$, respectively, we get
\begin{equation}\label{eq:sc:e1}
\left\|
    I_h(x_2)
\right\|_{L_{x_2}^{p'_2/p'_1}} \lesssim \| |h|^{p'_1} \|_{q'/p'_1},
\end{equation}
where
\[
  I_h(x_2) :=  \int_{\bbR^n} \frac{\Big|h\Big(\Big(\begin{array}{ll} \Lambda &  \\   & 0
\end{array} \Big)  x_1 + U^{-1} (\tilde A_{22} - \tilde A_{21} \tilde A_{11}^{-1} \tilde A_{12}) x_2\Big)\Big|^{p'_1}}
   {|x_1  |^{\lambda p'_1}} dx_1.
\]
It follows that $r>0$.

If $\lambda p'_1 \le n-r$, then
\begin{align*}
&\hskip -5mm  \int_{\bbR^{n-r}} \frac{dx_{1}^{(r+1)}\ldots dx_{1}^{(n)}}{|x_1|^{\lambda p'_1} } \\
&=
\int_{\bbR^{n-r}}  \frac{dx_{1}^{(r+1)}\ldots dx_{1}^{(n)}}{(|x_{1}^{(1)}| + \ldots + |x_{1}^{(r)}|+|x_{1}^{(r+1)}|+\ldots + |x_{1}^{(n)}|)^{\lambda p'_1}}
\\
&= \infty.
\end{align*}
Hence $I_h(x_2) = \infty$ for all $x_2\in \bbR^n$ and $h\ne 0$, which contradicts with (\ref{eq:sc:e1}).

If $\lambda p'_1 > n-r$, then
\begin{align*}
 \int_{\bbR^{n-r}} \frac{dx_{1}^{(r+1)}\ldots dx_{1}^{(n)}}{|x_1|^{\lambda p'_1} }
  \approx \frac{1}{(|x_{1}^{(1)}| + \ldots + |x_{1}^{(r)}|)^{\lambda p'_1 -n+r}}.
\end{align*}
Since
\[
  \begin{pmatrix}
  I & 0 \\
  -I & I
  \end{pmatrix}
  \begin{pmatrix}
  \tilde A_{21}\tilde A_{11}^{-1} &  0 \\
  0  & I
  \end{pmatrix}
   \begin{pmatrix}
   \tilde A_{11} & \tilde A_{12} \\
   \tilde A_{21} & \tilde A_{22}
   \end{pmatrix}
 = \begin{pmatrix}
   \tilde A_{21}  & \tilde A_{21} \tilde A_{11}^{-1} \tilde A_{12} \\
   0   &\tilde A_{22} -\tilde A_{21} \tilde A_{11}^{-1} \tilde A_{12}
 \end{pmatrix},
\]
we have
\[
  \rank (\tilde A_{21}) + n = \rank(\tilde A_{21}) + \rank(\tilde A_{22} -\tilde A_{21} \tilde A_{11}^{-1} \tilde A_{12} ).
\]
Hence $D:=U^{-1}(\tilde A_{22} -\tilde A_{21} \tilde A_{11}^{-1} \tilde A_{12}) $ is invertible.
Therefore, (\ref{eq:sc:e1}) is equivalent to
\begin{align*}
& \left\|
  \int_{\bbR^r}
  \frac{ |h ( (\lambda_1 x_{1}^{(1)} ,\ldots, \lambda_r x_{1}^{(r)}, 0, \ldots, 0 )^*
   + Dx_2)  |^{p'_1}}
     {(|x_{1}^{(1)}| + \ldots + |x_{1}^{(r)}|)^{\lambda p'_1 -n+r}}  dx_{1}^{(1)} \ldots dx_{1}^{(r)}\right\|_{L_{x_2}^{p'_2/p'_1}} \\
&  \hskip 10mm   \lesssim \| |h|^{p'_1} \|_{q'/p'_1}.
\end{align*}
Replacing $h$ and $x_2$ by $h( \diag [\Lambda^{-1},I] \cdot)$ and $D^{-1} \diag [\Lambda, I] x_2$, respectively, we get
\begin{align}
& \left\|
  \int_{\bbR^r}
  \frac{ |h ( (  x_{1}^{(1)} ,\ldots,   x_{1}^{(r)}, 0, \ldots, 0 )^* +  x_2)  |^{p'_1}}
     {(|x_{1}^{(1)}| + \ldots + |x_{1}^{(r)}|)^{\lambda p'_1 -n+r}}  dx_{1}^{(1)} \ldots dx_{1}^{(r)}\right\|_{L_{x_2}^{p'_2/p'_1}}
  \lesssim \| |h|^{p'_1} \|_{q'/p'_1}. \label{eq:sc:e3}
\end{align}
Define the operator $S$ by
\[
  Sh(x_2) = \int_{\bbR^r}
  \frac{ |h ( (  x_{1}^{(1)} ,\ldots,   x_{1}^{(r)}, 0, \ldots, 0 )^* +  x_2)  |^{p'_1}}
     {(|x_{1}^{(1)}| + \ldots + |x_{1}^{(r)}|)^{\lambda p'_1 -n+r}}  dx_{1}^{(1)} \ldots dx_{1}^{(r)}.
\]
Then for any $z\in\bbR^n$,
$S(h(\cdot-z))(x_2)=(Sh)(x_2-z)$, i.e., $S$ commutes with translations.
Therefore,
(\ref{eq:sc:e3}) is true only if $p'_2/p'_1 \ge q'/p'_1$,
which is equivalent to $q\ge p_2$.

For the case  $q> p_2$, set
\[
  h(x_1)  = \frac{1}{ (|x_{1}^{(r+1)}|+\ldots + |x_{1}^{(n)}|)^{\alpha} }
    \chi^{}_{\{x_1:\,  |x_{1}^{(i)}| <1, 1\le i\le n\}} (x_1),
\]
where $\alpha$ is a number such that $\alpha q' < n-r < \alpha p'_2$.
Then we have
$h\in L^{q'}$ and
\begin{align*}
&  \left\|
  \int_{\bbR^r}
  \frac{ |h ( (  x_{1}^{(1)} ,\ldots,   x_{1}^{(r)}, 0, \ldots, 0 )^* +  x_2)  |^{p'_1}}
     {(|x_{1}^{(1)}| + \ldots + |x_{1}^{(r)}|)^{\lambda p'_1 -n+r}}  dx_{1}^{(1)} \ldots dx_{1}^{(r)}\right\|_{L_{x_2}^{p'_2/p'_1}}
\\
&= \left\|
\frac{1}{ (|x_{2}^{(r+1)}|+\ldots + |x_{2}^{(n)}|)^{\alpha p'_1} }
    \chi^{}_{\{x_2:\,  |x_{2}^{(i)}| <1, r+1\le i\le n\}} (x_2) \right. \\
&\quad \times     \left.
  \int_{ |x_{1}^{(i)} + x_{2}^{(i)}| <1, 1\le i\le r}
  \frac{dx_{1}^{(1)} \ldots dx_{1}^{(r)}  }
     {(|x_{1}^{(1)}| + \ldots + |x_{1}^{(r)}|)^{\lambda p'_1 -n+r}}  \right\|_{L_{x_2}^{p'_2/p'_1}}
     \\
&=\infty.
\end{align*}
And for the case
 $q=p_2$, we have $\lambda /n = 1/p'_1$ and therefore $\lambda p'_1 = n$.
Since $\lambda p'_1 > n-r$, we have $r\ge 1$.
Let $\alpha$ be such that
$\alpha q'<n-r$. Then  the above equalities are also valid,
which contradicts with (\ref{eq:sc:e3}).

(ii)\,\,  $\tilde A_{11}$  is singular.

As in Case (i) we have $p_1>1$ and  $A_{22}\ne 0$. Hence $q\ge p_2\ge 1$. Therefore,
(\ref{eq:ss:e12a}) is true.

Suppose that $\tilde A_{21} = U \diag [\lambda_1, \ldots, \lambda_r, 0,\ldots, 0] V$,
where $r<n$ and $U$ and $V$ are orthonormal matrices.
If $r=0$, then $\tilde A_{21}=0$ and therefore, $\tilde A_{11}$ is invertible,
which contradicts with the assumption. Hence $r>0$.
By a change of variable of the form $x_1 \rightarrow  V^{-1}\tilde \Lambda x_1$, where
$\tilde \Lambda = \diag [1/\lambda_1, \ldots, 1/\lambda_r$, $1$, $\ldots$, $1]$,
and replacing $h$ by $h(U^{-1}\cdot)$, we see from (\ref{eq:ss:e12a}) that
\begin{equation}\label{eq:sc:e6}
  \left\|
  \int_{\bbR^n}
   \frac{ |h( ( x_1^{(1)}, \ldots,  x_1^{(r)}, 0,\ldots, 0)^* + U^{-1} \tilde A_{22} x_2)|^{p'_1}}
   {|\tilde A_{11} V^{-1}\tilde \Lambda  x_1 + \tilde A_{12}x_2|^{\lambda p'_1}} dx_1
  \right\|_{L_{x_2}^{p'_2/p'_1}}    \lesssim   \| |h|^{p'_1}\|_{q'/p'_1}.
\end{equation}

Denote $\tilde A_{11} V^{-1}\tilde \Lambda  = (v_1, \ldots, v_n)$, where $v_1$, $\ldots$, $v_n$ are $n$-dimensional column vectors.
If the last $n-r$ columns $v_{r+1}$, $\ldots$, $v_n$  are linearly dependent, then we can find some $(n-r)\times (n-r)$
invertible matrix $Q'$ such that the last column of $   ( v_{r+1}, \ldots, v_n) Q '$ consists of zeros.
By a change of variable  of the form $(x_1^{(r+1)}, \ldots, x_1^{(n)})^* \rightarrow  Q' (x_1^{(r+1)}, \ldots, x_1^{(n)})^* $, we get
\begin{align*}
& \int_{\bbR^n}
   \frac{ |h( ( x_1^{(1)}, \ldots,   x_1^{(r)}, 0,\ldots, 0)^* + U^{-1} \tilde A_{22} x_2)|^{p'_1}}
   {|\tilde A_{11} V^{-1} \tilde \Lambda x_1 + \tilde A_{12}x_2|^{\lambda p'_1}} dx_1 \\
&= \int_{\bbR^{n-1}}    \mbox{\{A function of $(x_1^{(1)}, \ldots, x_{1}^{(n-1)})$
and $x_2$\} \,\, } dx_1^{(1)} \ldots dx_{1}^{(n-1)}
   \int_{\bbR}
  dx_1^{(n)}\\
&=\infty.
\end{align*}

Next we assume that $\rank(  v_{r+1}, \ldots, v_n) = n-r$. Then there are
$n\times n$ orthonormal matrix  $P$, $(n-r)\times (n-r)$ orthonormal matrix $Q$  and positive numbers
$\lambda_{r+1}$, $\ldots$, $\lambda_n$ such that
\[
  ( v_{r+1}, \ldots, v_n) = P \begin{pmatrix}
  \lambda_{r+1} & & &\\
    & \cdots & & \\
  &&& \lambda_n \\
 0& & &\\
    & \cdots & & \\
  &&& 0 \\
  \end{pmatrix} Q.
\]
Note that
\[
  |\tilde A_{11} V^{-1} \tilde \Lambda x_1 + \tilde A_{12}x_2| \approx |P^{-1}(\tilde A_{11} V^{-1} \tilde \Lambda x_1 + \tilde A_{12}x_2)|.
\]
Let $B$ and $B'$ be the submatrices consisting  of the last $r$ rows and the first $n-r$ rows of $P^{-1} (v_1, \ldots, v_r)$, respectively.
Denote $P^{-1}\tilde A_{12} = \binom{D_4}{D_3}$, where $D_3$ and $D_4$ are $r\times n$ and $(n-r)\times n$ matrices, respectively.
By a change of variable of the form $(x_1^{(r+1)}, \ldots, x_1^{(n)})^* \rightarrow Q^{-1}  \diag [\lambda_{r+1}^{-1},\ldots, \lambda_n^{-1}](x_1^{(r+1)}, \ldots, x_1^{(n)})^*$, we get
\begin{align*}
& \int_{\bbR^{n-r}}
   \frac{1}
   {|\tilde A_{11} V^{-1} \tilde \Lambda x_1 + \tilde A_{12}x_2|^{\lambda p'_1}} dx_1^{(r+1)}\ldots dx_1^{(n)} \\
&\approx
\int_{\bbR^{n-r}}
   \frac{1}
   {|P^{-1}(\tilde A_{11} V^{-1} \tilde \Lambda x_1 + \tilde A_{12}x_2)|^{\lambda p'_1}} dx_1^{(r+1)}\ldots dx_1^{(n)} \\
&\approx
  \int_{\bbR^{n-r}}
   \frac{dx_1^{(r+1)}\ldots dx_1^{(n)}}
   {(|(x_1^{(r+1)}, \ldots, x_1^{(n)})^* + B'  y + D_4 x_2 | + |B  y + D_3 x_2|)^{\lambda p'_1}}  \\
&\approx
   \frac{1}
   {|B y + D_3 x_2|^{\lambda p'_1-n+r}},
\end{align*}
where $y= (x_1^{(1)}, \ldots, x_1^{(r)})^*$.

Denote $U^{-1}\tilde A_{22} = \binom{D_1}{D_2}$,
where $D_1$ and $D_2$ are $r\times n$ and $(n-r)\times n$ matrices, respectively.
Then (\ref{eq:sc:e6}) turns out to be
\[
  \left\|
  \int_{\bbR^r}
   \frac{ |h( y + D_1 x_2, D_2 x_2)|^{p'_1}}
   {|B y + D_3 x_2|^{\lambda p'_1}} dy
  \right\|_{L_{x_2}^{p'_2/p'_1}}    \lesssim   \| |h|^{p'_1}\|_{q'/p'_1}.
\]
That is,
\begin{equation}\label{eq:sc:e15}
  \left\|
  \int_{\bbR^r}
   \frac{ |h( y , D_2 x_2)|^{p'_1}}
   {|B y + (D_3 - BD_1) x_2|^{\lambda p'_1}} dy
  \right\|_{L_{x_2}^{p'_2/p'_1}}    \lesssim   \| |h|^{p'_1}\|_{q'/p'_1}.
\end{equation}
Note that
\begin{align*}
\begin{pmatrix}
P^{-1} & \\
& U^{-1}
\end{pmatrix}\!
\begin{pmatrix}
\tilde A_{11} & \tilde A_{12}\\
\tilde A_{21}  & \tilde A_{22}
\end{pmatrix}\!
\begin{pmatrix}
V^{-1}\tilde \Lambda & \\
 & I_n
\end{pmatrix}\!
\begin{pmatrix}
I_r &  & \\
& Q^{-1} &  \\
& & I_n
\end{pmatrix}
\!= \begin{pmatrix}
 B' &   \Lambda & D_4\\
 B & 0 & D_3 \\
I_r & 0  & D_1\\
0 & 0 & D_2
\end{pmatrix},
\end{align*}
where $\Lambda = \diag[\lambda_{r+1},
\ldots, \lambda_n]$.
We have
\[
 \rank \begin{pmatrix}
 B &  D_3 \\
I_r &   D_1\\
0 &   D_2
\end{pmatrix}
=
\rank \begin{pmatrix}
 B &   D_3 -BD_1\\
I_r &   0\\
0 &   D_2
\end{pmatrix}
=n+r.
\]
Hence  $  \rank \binom{D_2}{D_3- BD_1}=n$.
By a change of variable of the form
\[
  x_2 \rightarrow \binom{D_2}{D_3- BD_1}^{-1} \binom{z}{z'},
\]
where $z\in\bbR^{n-r}$ and $z'\in\bbR^r$, we get
\[
    \left\|
  \int_{\bbR^r}
   \frac{ |h( y , z)|^{p'_1}}
   {|B y + z'|^{\lambda p'_1}} dy
  \right\|_{L_{(z,z')}^{p'_2/p'_1}}    \lesssim   \| |h|^{p'_1}\|_{q'/p'_1}.
\]
Set $h(y,z)=h_1(y) h_2(z)$.  We have
\[
    \left\|
  \int_{\bbR^r}
   \frac{ |h_1( y)|^{p'_1}}
   {|B y + z'|^{\lambda p'_1}} dy
  \right\|_{L_{z'}^{p'_2/p'_1}}   \|h_2\|_{p'_2}^{p'_1}  \lesssim   \| |h_1|^{p'_1}\|_{q'/p'_1} \|h_2\|_{q'}^{p'_1} .
\]

Recall that $q\ge p_2$.
If $q>p_2$, by taking some $h_2\in L^{q'}\setminus L^{p'_2}$,
 we get a contradiction.

If $q=p_2$, then $\lambda p'_1=n$. It follows that
\begin{equation}\label{eq:a:e61}
    \left\|
  \int_{\bbR^r}
   \frac{ |h_1( y)|^{p'_1}}
   {|B y + z'|^{n}} dy
  \right\|_{L_{z'}^{p'_2/p'_1}}   \lesssim   \| |h_1|^{p'_1}\|_{q'/p'_1}.
\end{equation}
Note that $n>r>0$. If $B$ is invertible or $0$, by setting $h_1=\chi_{\{|y|\le 1\}}$ we get a contradiction.

Next we assume that $B$ is singular.
Suppose that $\rank(B) = r_1$.  Then there are  orthonormal matrices
$P_1$, $Q_1$ and diagonal matrix $\Lambda_1 = \diag[b_1$, $\ldots$, $b_{r_1}]$ such that
\[
  B = P_1 \begin{pmatrix}
   \Lambda_1 & 0 \\
   0   & 0
  \end{pmatrix} Q_1.
\]
Denote $y = (\tilde y, \bar y)$
and $P_1^{-1}z' = (\tilde z, \bar z )$,
where $\tilde y, \tilde z \in \bbR^{r_1}$
and $\bar y, \bar z\in\bbR^{r-r_1}$.
Set $h_1(y) = h_{11}(\tilde y) h_{12}(\bar  y )$.
By a change of variable of the form
$y\rightarrow Q_1^{-1}y$
and replacing $h_1$ by $h_1(Q_1\cdot)$, we see from (\ref{eq:a:e61}) that
\[
    \left\|
  \int_{\bbR^r}
   \frac{ |h_{11}(\tilde y) h_{12}(\bar y)|^{p'_1}}
   {(|\Lambda_1 \tilde y + \tilde z| +|\bar z|)^n} dy
  \right\|_{L_{(\tilde z, \bar z)}^{p'_2/p'_1}}   \lesssim   \|  |h_{11}|^{p'_1}\|_{L^{q'/p'_1}}
  \|  |h_{12}|^{p'_1}\|_{L^{q'/p'_1}}.
\]
Note that the above inequality is true for all
$h_{12}\in\bbR^{r-r_1}$.
Hence $p'_1=q'=p'_2$. Therefore,
\[
  \int_{\bbR^{r_1+r}}
   \frac{ |h_{11}(\tilde y)|^{p'_1}  }
   {(|\Lambda_1 \tilde y + \tilde z| +|\bar z|)^n} d\tilde y d\tilde z d\bar z
     \lesssim   \|  |h_{11}|^{p'_1}\|_{L^1},
\]
which is impossible since
\[
  \int_{\bbR^{r}}
   \frac{ 1 }
   {(|\Lambda_1 \tilde y + \tilde z| +|\bar z|)^n}   d\tilde z d\bar z=\infty.
\]
This completes the proof.
\end{proof}

Kenig and Stein \cite{KenigStein1999}, Grafakos and  Kalton
\cite{GrafakosKalton2001}
and  Grafakos and  Lynch
\cite{GrafakosLynch2015}
studied the bi-linear fractional integral
of the following form,
\[
  \int_{\bbR^n} \frac{f_1(x-t) f_2(x+t)}{|t|^{\lambda}} dt.
\]
They showed that for $1<p_1,p_2\le \infty$, $0<q<\infty$ and $0<\lambda<n$ which satisfy
\begin{equation}\label{eq:ss:homogeneity}
   \frac{1}{p_1} + \frac{1}{p_2}  = \frac{1}{q} + \frac{n-\lambda}{n},
\end{equation}
the above bi-linear fractional integral
is bounded from $L^{p_1}\times L^{p_2}$
to $L^q$.

As a consequence of Theorem~\ref{thm:m=1}, we get that the above bi-linear operator  extends
to a linear operator defined on $L^{\vec p}$.

Let $f$ be a measurable function defined on $\bbR^{2n}$. Define
\[
  L_{\lambda} f(x) = \int_{\bbR^n} \frac{f(x-t,x+t)}{|t|^{\lambda}} dt.
\]
\begin{Corollary}
Let $1\le p_1, p_2\le \infty$, $0<q\le\infty$   and $0<\lambda< n$ be constants which meet (\ref{eq:ss:homogeneity}).
Denote $\vec p = (p_1, p_2)$.
Then  the inequality
\begin{equation}\label{eq:ss:e1}
  \|L_{ \lambda}f\|_q \lesssim \|f\|_{L^{\vec p}}
\end{equation}
holds for any $f\in L^{\vec p}(\bbR^{2n})$ if and only if $1<p_2<q<p_1\le \infty$.
\end{Corollary}

\section{The case   \texorpdfstring{$n=1$}{n=1}}

In this section, we study the case  $n=1$. In this case,  we give necessary and sufficient
conditions on $A$, $\vec p$, $q$ and $\lambda$
such that
$T_{\lambda}$ is bounded from $L^{\vec p}$ to $L^q$.

Let $A_{(i_1,\ldots,i_k)}^{(j_1,\ldots,j_k)}$
stand  for the submatrix consisting of
the $i_1$-th, $\ldots$, $i_k$-th rows
and the $j_1$-th, $\ldots$, $j_k$-th columns
of $A$.

\begin{Theorem}\label{thm:A:n=1}
Let $1\le p_i\le \infty$ for  $1\le i\le m+1$ and $q, \lambda>0$ be  constants which satisfy (\ref{eq:ss:e0}).
Set  $\vec p = (p_1, \ldots, p_{m+1})$.
Suppose that $A$ is an $(m+1)\times (m+1)$ matrix.
Then
$T_{\lambda}$ is bounded from $L^{\vec p}$ to $L^q$ if and only if the following four conditions are satisfied.

\begin{enumerate}
\item The matrix $A$ is invertible.

\item There exist some positive integer $\tilde m \le m$ and
$\{j_1,\ldots,j_{\tilde m}\}
\subset \{1,\ldots,m\}$ such that
$
   A_{(m-{\tilde m}+2,\ldots,m+1)}^{
  (j_1,\ldots,j_{\tilde m})} $ and
$
  A_{(m-{\tilde m}+1,\ldots,m+1)}^{(j_1,\ldots,j_{\tilde m},m+1)} $ are invertible matrices
and $( A_{ m-{\tilde m}+2,m+1},
\ldots, A_{m+1,m+1})\ne \vec 0$.
Let $m_0$ be the maximum of all such $\tilde m$.
Denote
$k_1 = \max\{i:\, A_{i,m+1}\ne 0, m-m_0+2\le i\le m+1\}$.

\item
There is some $k\ge m-m_0+2$ such that $p_k>1$. Let $k_0$ be the maximum of such $k$.

\item $0<\lambda - \sum_{i=1}^{m-m_0} 1/p'_i< m_0$ and the indices $\vec p$ and  $q$  satisfy
    \begin{equation}\label{eq:w:e62}
      p_{k_1}<q \mathrm{\   and\ } p_{k_0}\le q<p_{m-m_0+1}.
    \end{equation}
    \end{enumerate}
\end{Theorem}

\begin{proof}
\textit{Necessity}.\,\,
(i) By Lemma~\ref{Lm:q}, $A$ is invertible and $q\ge 1$.

(ii)
We prove the conclusion by induction on $m$.

For $m=1$, we see from Theorem~\ref{thm:m=1} that
(ii) is true for $\tilde m=1$ and $j_1=1$.

Now assume that the conclusion is true whenever $m$ is replaced by $m-1$. Let us consider the case   $m$.
We see from the proof of Theorem~\ref{thm:main} that
$\|T_{\lambda}\|_{L^{\vec p}\rightarrow L^q}<\infty$
if and only if
\begin{equation}\label{eq:hh}
  \left\|\frac{ h((A^{-1}x)_{m+1})}
    {(|(A^{-1}x)_1|+\ldots+|(A^{-1}x)_m|)^{\lambda}} \right\|_{L^{\vec p'}} \lesssim \|h\|_{L^{q'}},
    \qquad \forall h\in L^{q'}.
\end{equation}

If $\rank(A_{(2,\ldots,m+1)}^{(1,\ldots,m)}) = m$,
we see from Theorem~\ref{thm:main} that
(ii) is true for $\tilde m=m$ and $j_i=i$, $1\le i\le m$.

Next we suppose that
$\rank(A_{(2,\ldots,m+1)}^{(1,\ldots,m)}) < m$.
Then there is some column, say the $i$-th column of $A_{(2,\ldots,m+1)}^{(1,\ldots,m)}$,
which is a linear combination of other columns of
$A_{(2,\ldots,m+1)}^{(1,\ldots,m)}$.
Set $\{j_1,\ldots,j_{m-1}\}
=\{1,\ldots,m\}\setminus\{i\}$
and
\[
  B
=  A_{(2,\ldots,m+1)}^{(j_1,\ldots,j_{m-1},m+1)}.
\]
Since the submatrix consisting of the last $m$ rows of $A$ is of rank $m$, we have
\[
  \rank(B) = m.
\]
Let $Q$ be  obtained by swapping the first and the $i$-th rows of the $m\times m$ identity matrix.
Then the first column of $A_{(2,\ldots,m+1)}^{(1,\ldots,m)} Q$ is a linear combination of the second, $\ldots$, the $(m-1)$-th columns.
Hence there is an $m\times m$ matrices $P$,
which is the product of elementary matrices and is
of the form
\[
  P = \begin{pmatrix}
  a_1 &  &   &  \\
  a_2 & 1&   &   \\
      & \cdots &  & \\
  a_m &    &     &  1
  \end{pmatrix},
\]
such that
\[
  A \begin{pmatrix}
  Q & \\
    & 1
  \end{pmatrix}
  \begin{pmatrix}
  P &\\
    & 1
  \end{pmatrix}      = \begin{pmatrix}
    1 & * \\
    0 & B
  \end{pmatrix}.
\]
Hence
\[
  \begin{pmatrix}
  P^{-1} &\\
    & 1
  \end{pmatrix}
 \begin{pmatrix}
  Q^{-1} & \\
    & 1
  \end{pmatrix} A^{-1}       = \begin{pmatrix}
    1 & \alpha  \\
    0 & B^{-1}
  \end{pmatrix},
\]
where $\alpha$ is a $1\times m$ vector.

Denote $U=  \begin{pmatrix}
  P^{-1} &\\
    & 1
  \end{pmatrix}
 \begin{pmatrix}
  Q^{-1} & \\
    & 1
  \end{pmatrix}$. Then  we have
\begin{equation}\label{eq:PQA}
  \sum_{i=1}^m|(A^{-1}x)_i|
\approx \sum_{i=1}^m|(U A^{-1}x)_i|.
\end{equation}
Set $y=(x_2,\ldots,x_{m+1})^*$. We have
\begin{align*}
&\int_{\bbR}
\frac{ |h((A^{-1}x)_{m+1})|^{p'_1}}
    {(|(A^{-1}x)_1|+\ldots+|(A^{-1}x)_m|)^{\lambda p'_1}} dx_1 \\
&\approx
\int_{\bbR}
\frac{ |h((A^{-1}x)_{m+1})|^{p'_1}}
    {(|(UA^{-1}x)_1|+\ldots+|(UA^{-1}x)_m|)^{\lambda p'_1}} dx_1 \\
&=
\int_{\bbR}
\frac{ |h((B^{-1}y)_m)|^{p'_1}}
    {(|x_1+\alpha y|+|(B^{-1}y)_1|\ldots+|(B^{-1}y)_{m-1}|)^{\lambda p'_1}} dx_1\\
&=  \frac{ |h((B^{-1}y)_m)|^{p'_1}}
    {( |(B^{-1}y)_1|\ldots+|(B^{-1}y)_{m-1}|)^{\lambda p'_1-1}}.
\end{align*}
Hence (\ref{eq:hh}) is equivalent to
\begin{equation}\label{eq:hhh}
    \left\|\frac{ |h((B^{-1}y)_m)| }
    {( |(B^{-1}y)_1|\ldots+|(B^{-1}y)_{m-1}|)^{\lambda  -1/p'_1}} \right\|_{L^{\vec {\tilde p}    '}} \lesssim \|h\|_{L^{q'}},
    \qquad \forall h\in L^{q'},
\end{equation}
where $\vec{\tilde p}= (p_2, \ldots, p_{m+1})$.

Define the operator $\tilde T_{\lambda}$ by
\[
  \tilde T_{\lambda} g(y_m)
  = \int_{\bbR^{m-1}}
  \frac{g(B y)}{(|y_1|+\ldots +|y_{m-1}|)^{\lambda-1/p'_1}} dy_1\ldots dy_{m-1}.
\]
Then $\tilde T_{\lambda}$ is bounded from
$L^{\vec{\tilde p}}$ to $L^q$ if and only if
(\ref{eq:hhh}) is true.

Since (\ref{eq:hh}) and (\ref{eq:hhh}) are equivalent,
 $\| T_{\lambda}\|_{L^{\vec  p }\rightarrow L^{q'}} < \infty$
if and only if
$ \|\tilde T_{\lambda}\|_{L^{\vec{\tilde p}}\rightarrow L^{q'}} < \infty$.
By the inductive assumption, we get the conclusion as desired.

(iii)\,\,
If $m_0=m$, then we see from Theorem~\ref{thm:main} that there is some $2\le k\le m+1$ such that
$p_k>1$.

Next we assume that $m_0<m$.
Since the submatrix consisting of the last $m_0+2$ rows of $A$ is of rank $m_0+2$
and $\rank(A_{(m-m_0+1,\ldots,m+1)}^{(j_1,\ldots,j_{m_0},m+1)})
=m_0+1$,
there exists some $j_{m_0+1}\not\in
\{j_i:\, 1\le i\le m_0\}$ such that
$\rank(A_{(m-m_0,\ldots,m+1)}^{(j_1,\ldots,j_{m_0+1},m+1)})
=m_0+2$.
Note that $m_0$ is the maximum of $\tilde m$ which meets (ii). So  we have
$\rank(A_{(m-m_0+1,\ldots,m+1)}^{(j_1,\ldots,j_{m_0+1})})
<m_0+1$.

If $m_0+1<m$, then we consider
 submatrix consisting of the last $m_0+3$ rows of $A$, which is of rank $m_0+3$
and has a submatrix
$A_{(m-m_0,\ldots,m+1)}^{(j_1,\ldots,j_{m_0+1},m+1)}$
whose rank is equal to $m_0+2$.
There exists some $j_{m_0+2}\not\in
\{j_i:\, 1\le i\le m_0+1\}$ such that
$\rank(A_{(m-m_0-1,\ldots,m+1)}^{(j_1,\ldots,j_{m_0+2},m+1)})
=m_0+3$
and
$\rank(A_{(m-m_0,\ldots,m+1)}^{(j_1,\ldots,j_{m_0+2})})
<m_0+2$.

Repeating the above procedure many times, we get a rearrangement
$\{j_l:\, 1\le l\le m\}$ of $\{1,\ldots, m\}$,
such that
\begin{align}
 \rank\Big(A_{(m-m_0+2,\ldots,m+1)}^{(j_1,\ldots,
j_{m_0})}\Big)&=m_0 ,
    \label{eq:k00}\\
 \rank\Big(A_{(m-m_0+2-l,\ldots,m+1)}^{(j_1,\ldots,
j_{m_0+l})}\Big)&<m_0+l ,\qquad 1\le l\le m- m_0,
 \label{eq:k01}\\
\rank\Big(A_{(m-m_0+1-l,\ldots,m+1)}^{(j_1,\ldots,
j_{m_0+l},m+1)}\Big)
&=m_0+l+1,
\qquad 0\le l\le m- m_0.
 \label{eq:k02}
\end{align}

For $0\le l\le m- m_0$, define the operator $T_{\lambda,l}$ on $L^{(p_{m-m_0-l+1},\ldots,p_{m+1})}$ by
\[
  T_{\lambda,l}f(y_{m_0+l+1})
   = \int_{\bbR^{m_0+l}}
     \frac{f\Big(A_{(m-m_0+1-l,\ldots,m+1)}^{(j_1,\ldots,
j_{m_0+l},m+1)} y\Big) dy_1\ldots dy_{m_0+l} }
{(|y_1|+\ldots+|y_{m_0+l}|)^{\lambda - \sum_{i=1}^{m-m_0-l}1/p'_i}}.
\]
Then we have $T_{\lambda,m-m_0} = T_{\lambda}$.
Moreover, we see from arguments in (ii) that
for $1\le l\le m-m_0$,
\[
  \|T_{\lambda,l}\|_{L^{(p_{m-m_0-l+1},\ldots,p_{m+1})}
\rightarrow L^q}<\infty
\]
if and only if
\[
  \|T_{\lambda,l-1}\|_{L^{(p_{m-m_0-l+2},\ldots,p_{m+1})}
\rightarrow L^q}<\infty.
\]
Since $\|T_{\lambda}\|_{L^{\vec p}\rightarrow L^q}\!<\infty$,
we have
$\|T_{\lambda,0}\|_{L^{(p_{m-m_0+1},\ldots,p_{m+1})}
\rightarrow L^q}\!<\!\infty$.
Applying  Theorem~\ref{thm:main}(ii) for $T_{\lambda,0}$, we get that there is some $k\ge m-m_0+2$ such that $p_k>1$.

(iv)\,\,
Use notations introduced in (iii).
Since $\|T_{\lambda,0}\|_{L^{(p_{m-m_0+1},\ldots,p_{m+1})}
\rightarrow L^q}$ $<\infty$, by applying Theorem~\ref{thm:main} with $T_{\lambda}$ being replaced
by $T_{\lambda,0}$, we get (\ref{eq:w:e62}).

\textit{Sufficiency}.\,\,
Suppose that (i) - (iv) are true.
Use notations introduced in the proof of the necessity. Then we have $T_{\lambda}$ is bounded from $L^{\vec p}$ to $L^q$
if and only if
$T_{\lambda,0}$ is bounded from $ L^{(p_{m-m_0+1},\ldots,p_{m+1})}$
to $L^q$.

On the other hand, by applying Theorem~\ref{thm:main} for $T_{\lambda,0}$,
we see from the hypotheses that
$T_{\lambda,0}$ is bounded from $ L^{(p_{m-m_0+1},\ldots,p_{m+1})}$
to $L^q$. Hence $T_{\lambda}$ is bounded from $L^{\vec p}$ to $L^q$.
\end{proof}

The following is an immediate consequence.

\begin{Corollary}
Suppose that $A$ is an $(m+1)\times (m+1)$ matrix
such that
$A_{m+1,m}\ne 0$,
$A_{m+1,m+1}\ne 0$,
and
$\rank(A_{(m,m+1)}^{(m,m+1)}) = 2$.
Then there exist some $\vec p=(p_1,\ldots,p_{m+1})$
and $q,\lambda>0$ such that
$T_{\lambda}$ is bounded from $L^{\vec p}$ to $L^q$.
\end{Corollary}

\small


\end{document}